\documentclass[a4paper,12pt]{preprint}
\usepackage{graphicx}
\usepackage{lmodern}
\usepackage[margin=1.3in]{geometry}  
\usepackage[full]{textcomp}
\usepackage[osf]{newtxtext}

\usepackage{mathalfa}
\usepackage{microtype}

\usepackage{booktabs}
\usepackage{amsmath}
\usepackage{amsfonts}
\usepackage{amssymb}
\usepackage{stmaryrd}
\usepackage{enumitem}
\usepackage{amsmath}
\usepackage{amsfonts} 
\usepackage{amssymb}             

\usepackage{comment}

\usepackage{mathrsfs}
\usepackage{booktabs}
\usepackage{mathtools}

\usepackage{tikz}
\usepackage{amsthm}                
\usepackage{mhequ}
\usepackage{hyperref}

\hypersetup{pdfstartview=}

\addtocontents{toc}{\protect\hypersetup{hidelinks}}

\DeclareSymbolFont{sfoperators}{OT1}{ptm}{m}{n}
\DeclareSymbolFontAlphabet{\mathsf}{sfoperators}

\makeatletter
\def\operator@font{\mathgroup\symsfoperators}
\makeatother

\numberwithin{equation}{section}

\newtheorem{thm}{Theorem}[section]
\newtheorem{defn}[thm]{Definition}
\newtheorem{lem}[thm]{Lemma}
\newtheorem{prop}[thm]{Proposition}

\newtheorem{cor}[thm]{Corollary}

\newtheorem{assumption}[thm]{Assumption}
\theoremstyle{remark}
\newtheorem{remark}[thm]{Remark}
\newtheorem{rmk}[thm]{Remark}

\makeatletter
\def\th@newremark{\th@remark\thm@headfont{\bfseries}}

\def\bdiamond{\mathop{\mathpalette\bdi@mond\relax}}
\newcommand\bdi@mond[2]{%
	\vcenter{\hbox{\m@th
			\scalebox{\ifx#1\displaystyle 2.6\else1.8\fi}{$#1\diamond$}%
	}}%
}

\def\bDiamond{\mathop{\mathpalette\bDi@mond\relax}}
\newcommand\bDi@mond[2]{%
	\vcenter{\hbox{\m@th
			\scalebox{\ifx#1\displaystyle 2.6\else1.2\fi}{$#1\Diamond$}%
	}}%
}
\makeatletter

\definecolor{darkgreen}{rgb}{0.1,0.7,0.1}
\definecolor{darkred}{rgb}{0.7,0.1,0.1}
\definecolor{darkblue}{rgb}{0,0,0.7}
\usepackage{amsmath,amssymb}
\newcommand{\iii}[1]{{\left\vert\kern-0.25ex\left\vert\kern-0.25ex\left\vert #1 
    \right\vert\kern-0.25ex\right\vert\kern-0.25ex\right\vert}}

\newcommand{\EE}{\mathbb{E}}     

\newcommand{\NN}{\mathbb{N}}
\newcommand{\PP}{\mathbb{P}}     

\newcommand{\RR}{\mathbb{R}}      
\newcommand{\TT}{\mathbb{T}}
\newcommand{\UU}{\mathbb{U}}

\newcommand{\ZZ}{\mathbb{Z}}      

\newcommand{\aA}{\mathcal{A}}

\newcommand{\cC}{\mathcal{C}}
\newcommand{\dD}{\mathcal{D}}

\newcommand{\iI}{\mathcal{I}}

\newcommand{\kK}{\mathcal{K}}
\newcommand{\lL}{\mathcal{L}}

\newcommand{\oO}{\mathcal{O}}
\newcommand{\pP}{\mathcal{P}}
\newcommand{\qQ}{\mathcal{Q}}
\newcommand{\rR}{\mathcal{R}}

\newcommand{\tT}{\mathcal{T}}
\newcommand{\uU}{\mathcal{U}}

\makeatletter 
\newcommand{\cov}{{\operator@font cov}}
\newcommand{\var}{{\operator@font var}}
\newcommand{\corr}{{\operator@font corr}}
\newcommand{\diam}{{\operator@font diam}}
\newcommand{\Av}{{\operator@font Av}}
\newcommand{\trig}{{\operator@font trig}}
\newcommand{\Enh}{{\operator@font Enh}}
\newcommand{\EEnh}{\overline {\operator@font Enh}}
\newcommand{\Com}{{\operator@font Com}}
\newcommand{\Er}{{\operator@font Er}}
\makeatother

\newcommand{\KPZ}{\text{\tiny KPZ}}
\newcommand{\HS}{\text{\tiny HS}}

\newcommand{\E}{\EE}

\newcommand{\N}{\mathbf{N}}

\newcommand{\sM}{\mathscr{M}}

\newcommand{\hPi}{\widehat{\Pi}}

\newcommand{\eps}{\varepsilon}

\colorlet{symbols}{blue!90!black}
\colorlet{testcolor}{green!60!black}

\def\${|\!|\!|}

\usetikzlibrary{shapes.misc}
\usetikzlibrary{shapes.symbols}
\usetikzlibrary{decorations}
\usetikzlibrary{decorations.markings}
\usetikzlibrary{decorations.pathmorphing}

\def\drawx{\draw[-,solid] (-3pt,-3pt) -- (3pt,3pt);\draw[-,solid] (-3pt,3pt) -- (3pt,-3pt);}
\tikzset{
	root/.style={circle,fill=testcolor,inner sep=0pt, minimum size=2mm},
	dot/.style={circle,fill=black,inner sep=0pt, minimum size=1mm},
	edot/.style={circle,fill=black,inner sep=0pt, minimum size=1mm},
	odot/.style={circle,draw=black,inner sep=0pt, minimum size=1mm},
	var/.style={circle,fill=black!10,draw=black,inner sep=0pt, minimum size=
		2mm},
	svar/.style={circle,fill=black!10,draw=black,inner sep=0pt, minimum size=
		1.5mm},
    noise0/.style={rectangle,draw=symbols,fill=white,inner sep=0pt, minimum size=2.5mm},
    noise1/.style={circle,draw=symbols,fill=white,inner sep=0pt, minimum size=2.2mm},
    noise2/.style={circle,draw=symbols,fill=symbols,inner sep=0pt, minimum size=2.2mm},
    noise-1/.style={rectangle,draw=symbols,fill=white,inner sep=0pt, minimum size=2.5mm},
    noise-2/.style={rectangle,draw=symbols,fill=white,inner sep=0pt, minimum size=2.5mm},
	dotred/.style={circle,fill=symbols!50,inner sep=0pt, minimum size=2mm},
	generic/.style={semithick,shorten >=1pt,shorten <=1pt},
	ageneric/.style={semithick},
	dist/.style={ultra thick,draw=testcolor,shorten >=1pt,shorten <=1pt},
	testfcn/.style={ultra thick,testcolor,shorten >=1pt,shorten <=1pt,<-},
	testfcnx/.style={ultra thick,testcolor,shorten >=1pt,shorten <=1pt,<-,
		postaction={decorate,decoration={markings,mark=at position 0.6 with {\drawx}}}},
	kepsilon/.style={semithick,shorten >=1pt,shorten <=1pt,densely dashed,->},
	kprimex/.style={semithick,shorten >=1pt,shorten <=1pt,densely dashed,->,
		postaction={decorate,decoration={markings,mark=at position 0.4 with {\drawx}}}},
	kernel/.style={ultra thick,postaction={decorate,decoration={markings,mark=at position 0.45 with {\draw[thick] (0,0.1) -- (0.1,0) -- (0,-0.1); }}}},
	akernel/.style={semithick,->},
	multx/.style={shorten >=1pt,shorten <=1pt,
		postaction={decorate,decoration={markings,mark=at position 0.5 with {\drawx}}}},
	kernelx/.style={semithick,shorten >=1pt,shorten <=1pt,->,
		postaction={decorate,decoration={markings,mark=at position 0.4 with {\drawx}}}},
	kernel1/.style={postaction={decorate,decoration={markings,mark=at position 0.45 with {\draw[fill=symbols] (0,-0.08) -- (0,0.08) -- (0.1,0) -- cycle;}}}},
    kernel2/.style={postaction={decorate,decoration={markings,mark=at position 0.45 with {\draw[fill=symbols] (0,-0.08) -- (0.1,0) -- (0.09,0) -- cycle; \draw[fill=symbols] (0,0.08) -- (0.09,0) -- (0.1,0) -- cycle;}}}},
	kernelBig/.style={semithick,shorten >=1pt,shorten <=1pt,decorate, decoration={zigzag,amplitude=1.5pt,segment length = 3pt,pre length=2pt,post length=2pt}},
	gepsilon/.style={dotted,semithick,shorten >=1pt,shorten <=1pt},
	renorm/.style={shape=circle,fill=white,inner sep=1pt},
	labl/.style={shape=rectangle,fill=white,inner sep=1pt},
	xi/.style={circle,fill=symbols!10,draw=symbols,inner sep=0pt,minimum size=1.2mm},
	xix/.style={crosscircle,fill=symbols!10,draw=symbols,inner sep=0pt,minimum size=1.2mm},
	xib/.style={circle,fill=symbols!10,draw=symbols,inner sep=0pt,minimum size=1.6mm},
	xibx/.style={crosscircle,fill=symbols!10,draw=symbols,inner sep=0pt,minimum size=1.6mm},
	not/.style={circle,fill=symbols,draw=symbols,inner sep=0pt,minimum size=0.5mm},
	>=stealth,
	highlight/.style={line width=7pt,blue,draw opacity=0.2,line cap=round,line join=round},
	cover/.style={line width=7pt,blue,line cap=round,line join=round},
	smalldot/.style={circle,fill=symbols,draw=symbols, solid,inner sep=0pt,minimum size=0.5mm},
    whitedot/.style={circle,fill=white,draw=testcolor, solid,inner sep=0pt,minimum size=1mm},
}

\def\scal#1{\langle#1\rangle}
\def\cent#1{\mathopen{{\langle\kern-0.3em\rangle}}#1\mathclose{{\langle\kern-0.3em\rangle}}}
\def\bscal#1{\big\langle#1\big\rangle}
\def\Bscal#1{\Big\langle#1\Big\rangle}

\makeatletter
\def\DeclareSymbol#1#2#3{\expandafter\gdef\csname MH@symb@#1\endcsname{\tikz[baseline=#2,scale=0.15,draw=symbols]{#3}}\expandafter\gdef\csname MH@symb@#1s\endcsname{\scalebox{0.5}{\tikz[baseline=#2,scale=0.15,draw=symbols]{#3}}}}
\def\<#1>{\csname MH@symb@#1\endcsname}
\makeatother

\makeatletter
\def\Declaresymbol#1#2#3{\expandafter\gdef\csname MH@symb@#1\endcsname{\tikz[baseline=#2,scale=0.3,draw=darkgreen]{#3}}\expandafter\gdef\csname MH@symb@#1s\endcsname{\scalebox{0.5}{\tikz[baseline=#2,scale=0.15,draw=darkgreen]{#3}}}}
\def\<#1>{\csname MH@symb@#1\endcsname}
\makeatother

\DeclareSymbol{0'}{0}{\node at (0,0.75) [noise0] {}; }
\DeclareSymbol{1'}{0}{\node at (0,0.75) [noise1] {}; }
\DeclareSymbol{2'}{0}{\node at (0,0.75) [noise2] {}; }
\DeclareSymbol{-1'}{0}{\node at (0,0.75) [noise0] {}; \draw (0,0.75) node {\tiny -$1$} ;}
\DeclareSymbol{-2'}{0}{\node at (0,0.75) [noise0] {}; \draw (0,0.75) node {\tiny -$2$} ;}

\DeclareSymbol{2'-}{0}{ \draw [black] (-1.2,-0.25) -- (-1.5,-0.25) -- (-1.5,1.75) -- (-1.2,1.75); \draw [black] (1.2,-0.25) -- (1.5,-0.25) -- (1.5,1.75) -- (1.2,1.75); \node at (0,0.75) [noise2] {}; }
\DeclareSymbol{D2'-}{0}{ \draw [black] (-1.2,-0.25) -- (-1.5,-0.25) -- (-1.5,1.75) -- (-1.2,1.75); \draw [black] (1.2,-0.25) -- (1.5,-0.25) -- (1.5,1.75) -- (1.2,1.75); \node at (0,0.75) [noise2] {};\node at (-3,0.75)  { \scriptsize$D_u$} }
\DeclareSymbol{D_{1}2'-}{0}{ \draw [black] (-1.2,-0.25) -- (-1.5,-0.25) -- (-1.5,1.75) -- (-1.2,1.75); \draw [black] (1.2,-0.25) -- (1.5,-0.25) -- (1.5,1.75) -- (1.2,1.75); \node at (0,0.75) [noise2] {};\node at (-3.2,0.75)  { \scriptsize$D_{u_1}$} }
\DeclareSymbol{D_{2}2'-}{0}{ \draw [black] (-1.2,-0.25) -- (-1.5,-0.25) -- (-1.5,1.75) -- (-1.2,1.75); \draw [black] (1.2,-0.25) -- (1.5,-0.25) -- (1.5,1.75) -- (1.2,1.75); \node at (0,0.75) [noise2] {};\node at (-3.2,0.75)  { \scriptsize$D_{u_2}$} }
\DeclareSymbol{DD2'-}{0}{ \draw [black] (-1.2,-0.25) -- (-1.5,-0.25) -- (-1.5,1.75) -- (-1.2,1.75); \draw [black] (1.2,-0.25) -- (1.5,-0.25) -- (1.5,1.75) -- (1.2,1.75); \node at (0,0.75) [noise2] {}; \node at (-3.2,0.75)  { \scriptsize$D^2_{\vec{u}}$}}
\DeclareSymbol{1'-}{0}{ \draw [black] (-1.2,-0.25) -- (-1.5,-0.25) -- (-1.5,1.75) -- (-1.2,1.75); \draw [black] (1.2,-0.25) -- (1.5,-0.25) -- (1.5,1.75) -- (1.2,1.75); \node at (0,0.75) [noise1] {}; }
\DeclareSymbol{D1'-}{0}{ \draw [black] (-1.2,-0.25) -- (-1.5,-0.25) -- (-1.5,1.75) -- (-1.2,1.75); \draw [black] (1.2,-0.25) -- (1.5,-0.25) -- (1.5,1.75) -- (1.2,1.75); \node at (0,0.75) [noise1] {};\node at (-3,0.75)  { \scriptsize$D_u$} }
\DeclareSymbol{D_{1}1'-}{0}{ \draw [black] (-1.2,-0.25) -- (-1.5,-0.25) -- (-1.5,1.75) -- (-1.2,1.75); \draw [black] (1.2,-0.25) -- (1.5,-0.25) -- (1.5,1.75) -- (1.2,1.75); \node at (0,0.75) [noise1] {};\node at (-3.5,0.75)  { \scriptsize$D_{u_1}$} }
\DeclareSymbol{D_{2}1'-}{0}{ \draw [black] (-1.2,-0.25) -- (-1.5,-0.25) -- (-1.5,1.75) -- (-1.2,1.75); \draw [black] (1.2,-0.25) -- (1.5,-0.25) -- (1.5,1.75) -- (1.2,1.75); \node at (0,0.75) [noise1] {};\node at (-3.5,0.75)  { \scriptsize$D_{u_2}$} }
\DeclareSymbol{D_{3}1'-}{0}{ \draw [black] (-1.2,-0.25) -- (-1.5,-0.25) -- (-1.5,1.75) -- (-1.2,1.75); \draw [black] (1.2,-0.25) -- (1.5,-0.25) -- (1.5,1.75) -- (1.2,1.75); \node at (0,0.75) [noise1] {};\node at (-3.5,0.75)  { \scriptsize$D_{u_3}$} }
\DeclareSymbol{0'-}{0}{ \draw [black] (-1.2,-0.25) -- (-1.5,-0.25) -- (-1.5,1.75) -- (-1.2,1.75); \draw [black] (1.2,-0.25) -- (1.5,-0.25) -- (1.5,1.75) -- (1.2,1.75); \node at (0,0.75) [noise0] {}; }
\DeclareSymbol{D0'-}{0}{ \draw [black] (-1.2,-0.25) -- (-1.5,-0.25) -- (-1.5,1.75) -- (-1.2,1.75); \draw [black] (1.2,-0.25) -- (1.5,-0.25) -- (1.5,1.75) -- (1.2,1.75); \node at (0,0.75) [noise0] {};\node at (-3,0.75)  { \scriptsize$D_u$} }
\DeclareSymbol{D_{1}0'-}{0}{ \draw [black] (-1.2,-0.25) -- (-1.5,-0.25) -- (-1.5,1.75) -- (-1.2,1.75); \draw [black] (1.2,-0.25) -- (1.5,-0.25) -- (1.5,1.75) -- (1.2,1.75); \node at (0,0.75) [noise0] {};\node at (-3.5,0.75)  { \scriptsize$D_{u_1}$} }
\DeclareSymbol{D_{2}0'-}{0}{ \draw [black] (-1.2,-0.25) -- (-1.5,-0.25) -- (-1.5,1.75) -- (-1.2,1.75); \draw [black] (1.2,-0.25) -- (1.5,-0.25) -- (1.5,1.75) -- (1.2,1.75); \node at (0,0.75) [noise0] {};\node at (-3.5,0.75)  { \scriptsize$D_{u_2}$} }
\DeclareSymbol{0'0}{0}{\draw (0,-0.25) node[not] {} -- (0,1.75) node[noise0] {};}
\DeclareSymbol{1'0}{0}{\draw (0,-0.25) node[not] {} -- (0,1.75) node[noise1] {};}
\DeclareSymbol{2'0}{0}{\draw (0,-0.25) node[not] {} -- (0,1.75) node[noise2] {};}
\DeclareSymbol{1'1'}{0}{\draw (-1,2) node[noise1] {}  -- (0.5,0) node[noise1] {};}
\DeclareSymbol{2'1'}{0}{\draw (-1,2) node[noise2] {}  -- (0.5,0) node[noise1] {};}
\DeclareSymbol{2'1'0}{0}{\draw (0,2.25) node[noise2] {}  -- (1.5,0.75) node[noise1] {} -- (0,-0.75) node[not] {};}
\DeclareSymbol{2'2'0}{-2}{\draw (-1,1) node[noise2] {}  -- (0,-0.5) node[smalldot] {} -- (1,1) node[noise2] {};}
\DeclareSymbol{2'1'1'}{-1}{\draw (0,2.25) node[noise2] {}  -- (1.5,0.75) node[noise1] {} -- (0,-0.75) node[noise1] {};}
\DeclareSymbol{2'0'}{0}{\draw (-1,2) node[noise2] {}  -- (0.5,0) node[noise0] {};}
\DeclareSymbol{2'2'0'}{0}{\draw (-1,2) node[noise2] {}  -- (0,0) node[noise0] {} -- (1,2) node[noise2] {};}

\DeclareSymbol{0'zeta-}{0}{ \draw [black] (-1.2,-0.25) -- (-1.5,-0.25) -- (-1.5,1.75) -- (-1.2,1.75); \draw [black] (3,-0.25) -- (3.3,-0.25) -- (3.3,1.75) -- (3,1.75); \node at (0,0.75) [noise0] {};\node at (1.8,1.5)  {\tiny$(\zeta)$}; }
\DeclareSymbol{1'zeta-}{0}{ \draw [black] (-1.2,-0.25) -- (-1.5,-0.25) -- (-1.5,1.75) -- (-1.2,1.75); \draw [black] (2.4,-0.25) -- (2.7,-0.25) -- (2.7,1.75) -- (2.4,1.75); \node at (0,0.75) [noise1] {};\node at (1.5,1.5)  {\tiny$(\zeta)$}; }
\DeclareSymbol{D1'zeta-}{-10}{ \draw [black] (-1.2,-0.25) -- (-1.5,-0.25) -- (-1.5,1.75) -- (-1.2,1.75); \draw [black] (2.4,-0.25) -- (2.7,-0.25) -- (2.7,1.75) -- (2.4,1.75); \node at (0,0.75) [noise1]{};\node at (1.5,1.5)  {\tiny$(\zeta)$};\node at (-3,0.75)  { \scriptsize$D_u$} }
\DeclareSymbol{DD1'zeta-}{-10}{ \draw [black] (-1.2,-0.25) -- (-1.5,-0.25) -- (-1.5,1.75) -- (-1.2,1.75); \draw [black] (2.4,-0.25) -- (2.7,-0.25) -- (2.7,1.75) -- (2.4,1.75); \node at (0,0.75) [noise1]{};\node at (1.5,1.5)  {\tiny$(\zeta)$};\node at (-3,0.75)  { \scriptsize$D^2_{\vec{u}}$} }
\DeclareSymbol{2'zeta-}{0}{ \draw [black] (-1.2,-0.25) -- (-1.5,-0.25) -- (-1.5,1.75) -- (-1.2,1.75); \draw [black] (2.4,-0.25) -- (2.7,-0.25) -- (2.7,1.75) -- (2.4,1.75); \node at (0,0.75) [noise2] {};\node at (1.5,1.5)  {\tiny$(\zeta)$}; }
\DeclareSymbol{1'_}{-10}{\draw [testfcn] (0,0) -- (4.5,-4.5); \draw (5,-5) node[root]{}; \draw (0,0) node {$\<1'->$};}

\DeclareSymbol{0'_}{-10}{\draw [kernel2] (5,0) circle(0.35) -- (0,0); \draw (6,0) node{\scriptsize $u$}; \draw [testfcn] (0,0) -- (4.5,-4.5); \draw (5,-5) node[root]{}; \draw (0,0) node {$\<0'->$};}
\DeclareSymbol{1'zeta_}{-13}{\draw [testfcn] (-1,-0.5) -- (3.5,-5); \draw (4,-5.5) node[root]{}; \draw (0,0) node {$\<1'zeta->$};}
\DeclareSymbol{D1'zeta_}{-13}{\draw [testfcn] (-1,-0.5) -- (3.5,-5); \draw (4,-5.5) node[root]{}; \draw (-2,0) node {$\<D1'zeta->$};}
\DeclareSymbol{DD1'zeta_}{-13}{\draw [testfcn] (-1,-0.5) -- (3.5,-5); \draw (4,-5.5) node[root]{}; \draw (-4.5,0) node {$\<DD1'zeta->$};}
\DeclareSymbol{-1'_}{-10}{\draw [kernel2] (5,1) circle(0.35) -- (0,0); \draw [kernel2] (5,-1) circle(0.35) -- (0,0); \draw (6.5,1) node{\scriptsize $u_1$}; \draw (6.5,-1) node{\scriptsize $u_2$}; \draw [testfcn] (0,0) -- (4.5,-4.5); \draw (5,-5) node[root]{};}

\DeclareSymbol{D2',1',1'}{10}{ \draw [testfcn] (0,0) -- (4.5,-4.5); \draw [kernel2] (0,8)  -- (0,4); \draw [kernel1] (0,4) -- (0,-0);\draw (-2,8) node {$\<D2'->$};\draw (0,4) node {$\<1'->$};\draw (0,0) node {$\<1'->$};\draw (5,-5) node[root]{}  }
\DeclareSymbol{2',D1',1'}{10}{ \draw [testfcn] (0,0) -- (4.5,-4.5); \draw [kernel2] (0,8)  -- (0,4); \draw [kernel1] (0,4) -- (0,-0);\draw (0,8) node {$\<2'->$};\draw (-2,4) node {$\<D1'->$};\draw (0,0) node {$\<1'->$};\draw (5,-5) node[root]{}  }
\DeclareSymbol{2',1',D1'}{10}{ \draw [testfcn] (0,0) -- (4.5,-4.5); \draw [kernel2] (0,8)  -- (0,4); \draw [kernel1] (0,4) -- (0,-0);\draw (0,8) node {$\<2'->$};\draw (0,4) node {$\<1'->$};\draw (-2,0) node {$\<D1'->$};\draw (5,-5) node[root]{}  }
\DeclareSymbol{D2',D1',1'}{10}{ \draw [testfcn] (0,0) -- (4.5,-4.5); \draw [kernel2] (0,8)  -- (0,4); \draw [kernel1] (0,4) -- (0,-0);\draw (-2,8) node {$\<D2'->$};\draw (-2,4) node {$\<D1'->$};\draw (0,0) node {$\<1'->$};\draw (5,-5) node[root]{}  }
\DeclareSymbol{D2',1',D1'}{10}{ \draw [testfcn] (0,0) -- (4.5,-4.5); \draw [kernel2] (0,8)  -- (0,4); \draw [kernel1] (0,4) -- (0,-0);\draw (-2,8) node {$\<D2'->$};\draw (0,4) node {$\<1'->$};\draw (-2,0) node {$\<D1'->$};\draw (5,-5) node[root]{}  }
\DeclareSymbol{2',D1',D1'}{10}{ \draw [testfcn] (0,0) -- (4.5,-4.5); \draw [kernel2] (0,8)  -- (0,4); \draw [kernel1] (0,4) -- (0,-0);\draw (0,8) node {$\<2'->$};\draw (-2,4) node {$\<D1'->$};\draw (-2,0) node {$\<D1'->$};\draw (5,-5) node[root]{} }
\DeclareSymbol{D2',D1',D1'}{10}{ \draw [testfcn] (0,0) -- (4.5,-4.5); \draw [kernel2] (0,8)  -- (0,4); \draw [kernel1] (0,4) -- (0,-0);\draw (-2,8) node {$\<D2'->$};\draw (-2,4) node {$\<D1'->$};\draw (-2,0) node {$\<D1'->$};\draw (5,-5) node[root]{}  }

\DeclareSymbol{2'1'1'_}{-1}{\draw (0,2.25) node[noise2] {}  -- (1.5,0.75); \draw [kernel1] (1.5,0.75) -- (0,-0.75); \draw (0,-0.75) node[noise1] {}; \draw (1.5,0.75) node[noise1] {}; \draw (-1,2.25) node{\tiny $\langle$}; \draw (1,2.25) node{\tiny $\rangle$} ;}

\DeclareSymbol{2',1',1'}{5}{ \draw [testfcn] (0,0) -- (4.5,-4.5); \draw (5,-5) node[root]{}; \draw [kernel2] (0,8) node {$\<2'->$}  -- (0,4); \draw [kernel1] (0,4) -- (0,0); \draw (0,4) node {$\<1'->$}; \draw (0,0) node {$\<1'->$}; }
\DeclareSymbol{2',1',1'x}{5}{\draw [kernel2] (0,6) node {$\<2'->$}  -- (0,2); \draw [kernel1] (0,2) -- (0,-2); \draw (0,2) node {$\<1'->$}; \draw (0,-2) node {$\<1'->$}; }
\DeclareSymbol{2',1'y}{5}{\draw [kernel2] (0,4) node {$\<2'>$}  -- (0,0);\draw (0,0) node {$\<1'>$};}
\DeclareSymbol{2'-,1'-y}{5}{\draw [kernel2] (0,4) node {$\<2'->$}  -- (0,0);\draw (0,0) node {$\<1'->$};}
\DeclareSymbol{D2'-,D1'-y}{5}{\draw [kernel2] (0,4) -- (0,0); \draw (-1.8,4) node {$D_u\<2'->$} ;\draw (-1.8,0) node {$D_u\<1'->$};}
\DeclareSymbol{2'-,S1'y}{5}{\draw [kernel2] (0,4) node {$\<2'->$}  -- (0,0);\draw (-2,0) node {\small $\tT^{(1)}\<1'>$};}
\DeclareSymbol{S2',S1'y}{5}{\draw [kernel2] (0,4) -- (0,0); \draw (-3,4) node {\small $\tT^{(\leq2)}\<2'>$};\draw (-2,0) node {\small $\tT^{(1)}\<1'>$};}
\DeclareSymbol{D2',1'y}{5}{\draw [kernel2] (0,4) -- (0,0); \draw (-1.8,4) node {\small $D_u\<2'>$};\draw (0,0) node {\small $\<1'>$};}
\DeclareSymbol{2',1',-1'}{5}{ \draw [testfcn] (0,0) -- (4.5,-4.5); \draw (5,-5) node[root]{}; \draw [kernel2] (5,1)  circle(0.35) -- (0,0); \draw [kernel2] (5,-1) circle(0.35) -- (0,0); \draw(6.5,-1) 
node{\scriptsize $u_2$}; \draw(6.5,1) node{\scriptsize $u_1$}; \draw [kernel2] (0,8) node {$\<2'->$}  -- (0,4); \draw [kernel1] (0,4) -- (0,-0); \draw (0,4) node {$\<1'->$}; \draw (0,0) node {$\<-1'>$}; }

\DeclareSymbol{2',0',0'}{5}{ \draw [testfcn] (0,0) -- (4.5,-4.5); \draw (5,-5) node[root]{}; \draw[kernel2] (5,0)  circle(0.35) -- (0,0); \draw [kernel2] (5,4) circle(0.35) -- (0,4); \draw(6.5,0) 
node{\scriptsize $u_2$}; \draw(6.5,4) node{\scriptsize $u_1$}; \draw [kernel2] (0,8) node {$\<2'->$}  -- (0,4); \draw [kernel1] (0,4) -- (0,0); \draw (0,4) node {$\<0'->$}; \draw (0,0) node {$\<0'->$}; }

\DeclareSymbol{2',-1',1'}{5}{ \draw [testfcn] (0,0) -- (4.5,-4.5); \draw (5,-5) node[root]{}; \draw [kernel2] (5,5)  circle(0.35) -- (0,4) ;\draw [kernel2] (5,3) circle(0.35) -- (0,4); \draw(6.5,3) 
node{\scriptsize $u_2$}; \draw(6.5,5) node{\scriptsize $u_1$}; \draw [kernel2] (0,8) node {$\<2'->$}  -- (0,4); \draw [kernel1] (0,4) -- (0,-0); \draw (0,4) node {$\<-1'>$}; \draw (0,0) node {$\<1'->$}; }

\DeclareSymbol{1',1',0'}{5}{ \draw [testfcn] (0,0) -- (4.5,-4.5); \draw (5,-5) node[root]{}; \draw [kernel2] (5,0)  circle(0.35) -- (0,0); \draw [kernel2] (5,8) circle(0.35) -- (0,8); \draw(6.5,0) 
node{\scriptsize $u_2$}; \draw(6.5,8) node{\scriptsize $u_1$}; \draw [kernel2](0,8) node {$\<1'->$}  -- (0,4); \draw [kernel1] (0,4) -- (0,-0); \draw (0,4) node {$\<1'->$}; \draw (0,0) node {$\<0'->$}; }

\DeclareSymbol{1',0',1'}{6}{ \draw [testfcn] (0,0) -- (4.5,-4.5); \draw (5,-5) node[root]{}; \draw [kernel2] (5,4)  circle(0.35) -- (0,4); \draw [kernel2] (5,8) circle(0.35) -- (0,8); \draw(6.5,4) 
node{\scriptsize $u_2$}; \draw(6.5,8) node{\scriptsize $u_1$}; \draw [kernel2] (0,8) node {$\<1'->$}  -- (0,4); \draw [kernel1] (0,4) -- (0,-0); \draw (0,4) node {$\<0'->$}; \draw (0,0) node {$\<1'->$}; }

\DeclareSymbol{0',1',1'}{9}{ \draw [testfcn] (0,0) -- (4.5,-4.5); \draw (5,-5) node[root]{}; \draw [kernel2] (5,9)  circle(0.35) -- (0,8); \draw [kernel2] (5,7) circle(0.35) -- (0,8); \draw(6.5,7) 
node{\scriptsize $u_2$}; \draw(6.5,9) node{\scriptsize $u_1$}; \draw [kernel2] (0,8) node {$\<0'->$}  -- (0,4); \draw [kernel1] (0,4) -- (0,-0); \draw (0,4) node {$\<1'->$}; \draw (0,0) node {$\<1'->$};}

\DeclareSymbol{2',1',1'_2}{16}{ \draw [kernel2] (0,8) node {$\<2'->$} -- (0,4); \draw [kernel1] (0,4) -- (0,0); \draw(0,0) circle(0.35); \draw (0,4) node {$\<1'->$}; \draw (1,0) node {\scriptsize $x$}; }
\DeclareSymbol{z,1',x}{16}{\draw (0,8) circle(0.35);\draw (1,8) node {\scriptsize $z$};\draw [kernel2] (0,8)-- (0,4); \draw [kernel1] (0,4) -- (0,0); \draw(0,0) circle(0.35); \draw (0,4) node {$\<1'->$}; \draw (1,0) node {\scriptsize $x$}; }
\DeclareSymbol{z,DD1',x}{14}{ \draw (1,8) node {\scriptsize $z$};\draw [kernel2] (5,5) circle(0.35)--(0,4); \draw [kernel2] (5,3) circle(0.35) -- (0,4); \draw [kernel2] (0,8) -- (0,4); \draw (0,8) circle(0.35); \draw (6.5,5) node{\scriptsize $u_1$}; \draw (6.5,3) node{\scriptsize $u_2$}; \draw [kernel1] (0,4) -- (0,0); \draw (0,0) circle(0.35); \draw (1,0) node {\scriptsize $x$};}

\DeclareSymbol{2',1',1'_3}{16}{ \draw [kernel2] (5,8) circle(0.35) -- (0,8) ; \draw [kernel2] (0,8) node {$\<1'->$} -- (0,4); \draw (6,8) node{\scriptsize $u$}; \draw [kernel1] (0,4) -- (0,0); \draw(0,0) circle(0.35); \draw (0,4) node {$\<1'->$};\draw (1,0) node {\scriptsize $x$}; }
\DeclareSymbol{D2',1',1'_3}{18}{ \draw [kernel2] (0,8)  -- (0,4); \draw [kernel1] (0,4) -- (0,0); \draw(0,0) circle(0.35); \draw (0,4) node {$\<1'->$};\draw (1,0) node {\scriptsize $x$};\draw (-2,8)node {$\<D2'->$} }

\DeclareSymbol{2',1',1'_4}{16}{ \draw [kernel2] (5,4) circle(0.35) -- (0,4); \draw [kernel2] (0,8) node {$\<2'->$} -- (0,4); \draw (6,4) node{\scriptsize $u$}; \draw [kernel1] (0,4) -- (0,0); \draw(0,0) circle(0.35); \draw (0,4) node {$\<0'->$}; \draw (1,0) node {\scriptsize $x$}; }
\DeclareSymbol{1'u,1',x}{16}{ \draw [kernel2] (5,8) circle(0.35) -- (0,8); \draw [kernel2] (0,8) node {$\<1'->$} -- (0,4); \draw (6,8) node{\scriptsize $u$}; \draw [kernel1] (0,4) -- (0,0); \draw(0,0) circle(0.35); \draw (0,4) node {$\<1'->$}; \draw (1,0) node {\scriptsize $x$}; }
\DeclareSymbol{1',1'u,x}{16}{ \draw [kernel2] (5,8) circle(0.35) -- (0,4); \draw [kernel2] (0,8) node {$\<1'->$} -- (0,4); \draw (6,8) node{\scriptsize $u$}; \draw [kernel1] (0,4) -- (0,0); \draw(0,0) circle(0.35); \draw (0,4) node {$\<1'->$}; \draw (1,0) node {\scriptsize $x$}; }
\DeclareSymbol{D21'u_1,1',x}{16}{ \draw [kernel2] (5,8) circle(0.35) -- (0,8); \draw [kernel2] (0,8) -- (0,4); \draw (6.5,8) node{\scriptsize $u_1$}; \draw [kernel1] (0,4) -- (0,0); \draw(0,0) circle(0.35); \draw (0,4) node {$\<1'->$}; \draw (1,0) node {\scriptsize $x$}; \draw (-2.3,8) node{$\<D_{2}1'->$}}
\DeclareSymbol{D21'u_1,D21',x}{16}{ \draw [kernel2] (5,8) circle(0.35) -- (0,8); \draw [kernel2] (0,8) -- (0,4); \draw (6.5,8) node{\scriptsize $u_1$}; \draw [kernel1] (0,4) -- (0,0); \draw(0,0) circle(0.35); \draw (-2.3,4) node {$\<D_{2}1'->$}; \draw (1,0) node {\scriptsize $x$}; \draw (-2.3,8) node{$\<D_{2}1'->$}}
\DeclareSymbol{1'u_1,D21',x}{16}{ \draw [kernel2] (5,8) circle(0.35) -- (0,8); \draw [kernel2] (0,8) -- (0,4); \draw (6.5,8) node{\scriptsize $u_1$}; \draw [kernel1] (0,4) -- (0,0); \draw(0,0) circle(0.35); \draw (-2.3,4) node {$\<D_{2}1'->$}; \draw (1,0) node {\scriptsize $x$}; \draw (0,8) node{$\<1'->$}}
\DeclareSymbol{1'u}{16}{ \draw [kernel2] (5,4.5) circle(0.35) -- (0,4.5); \draw (6,4.5) node{\scriptsize $u$}; \draw (0,4.5) node {$\<0'->$}}
\DeclareSymbol{2'u}{16}{ \draw [kernel2] (5,4.5) circle(0.35) -- (0,4.5); \draw (6,4.5) node{\scriptsize $u$}; \draw (0,4.5) node {$\<1'->$}}
\DeclareSymbol{2',D1',1'_4}{16}{\draw [kernel2] (0,8) node {$\<2'->$} -- (0,4); \draw [kernel1] (0,4) -- (0,0); \draw(0,0) circle(0.35); \draw (-1.8,4) node {$\<D1'->$}; \draw (1,0) node {\scriptsize $x$}; }
\DeclareSymbol{2',1',1'_5}{14}{ \draw [kernel2] (5,5) circle(0.35)--(0,4); \draw [kernel2] (5,3) circle(0.35) -- (0,4); \draw [kernel2] (0,8) node {$\<2'->$} -- (0,4); \draw (6.5,5) node{\scriptsize $u_1$}; \draw (6.5,3) node{\scriptsize $u_2$}; \draw [kernel1] (0,4) -- (0,0); \draw (0,0) circle(0.35); \draw (0,4) node {$\<-1'>$}; \draw (1,0) node {\scriptsize $x$}; }
\DeclareSymbol{D2',D1',1'_5}{18}{ \draw [kernel2] (0,8) -- (0,4); \draw [kernel1] (0,4) -- (0,0); \draw (0,0) circle(0.35);  \draw (1,0) node {\scriptsize $x$};\draw (-2,8) node {$\<D2'->$};\draw (-2,4) node {$\<D1'->$} }
\DeclareSymbol{2',D1',1'_6}{16}{ \draw [kernel2] (5,4) circle(0.35) -- (0,4); \draw [kernel2] (0,8) node {$\<2'->$} -- (0,4); \draw (6.5,4) node{\scriptsize $u_1$}; \draw [kernel1] (0,4) -- (0,0); \draw (0,0) circle(0.35); \draw (-2.3,4) node {$\<D_{2}0'->$}; \draw (1,0) node {\scriptsize $x$}; }
\DeclareSymbol{D_20u1}{14}{ \draw [kernel2] (5,4) circle(0.35) -- (0,4); \draw (6.5,4) node{\scriptsize $u_1$}; \draw (-2.3,4) node {$\<D_{2}0'->$};}
\DeclareSymbol{D_21u1}{14}{ \draw [kernel2] (5,4) circle(0.35) -- (0,4); \draw (6.5,4) node{\scriptsize $u_1$}; \draw (-2.3,4) node {$\<D_{2}1'->$};}

\DeclareSymbol{2',1',1'_6}{14}{ \draw [kernel2] (5,4) circle(0.35) -- (0,4); \draw [kernel2] (5,8) circle(0.35) -- (0,8); \draw [kernel2] (0,8) node {$\<1'->$} -- (0,4); \draw (6.5,4) node{\scriptsize $u_2$}; \draw [kernel1] (0,4) -- (0,0); \draw (0,0) circle(0.35); \draw (0,4) node {$\<0'->$};}
\DeclareSymbol{D2',1',1'_7}{14}{ \draw [kernel2] (5,4) circle(0.35) -- (0,4); \draw [kernel2] (0,8)  -- (0,4); \draw (6.5,4) node{\scriptsize $u_1$}; \draw [kernel1] (0,4) -- (0,0); \draw (0,0) circle(0.35); \draw (0,4) node {$\<0'->$}; \draw (1,0) node {\scriptsize $x$}; \draw (-2.2,8) node {$\<D_{2}2'->$}}
\DeclareSymbol{D2',D1',1'_8}{16}{ \draw [kernel2] (5,4) circle(0.35) -- (0,4); \draw [kernel2] (0,8)  -- (0,4); \draw (6.5,4) node{\scriptsize $u_1$}; \draw [kernel1] (0,4) -- (0,0); \draw (0,0) circle(0.35);\draw (-2.2,4) node {$\<D_{2}0'->$}; \draw (1,0) node {\scriptsize $x$}; \draw (-2.2,8) node {$\<D_{2}2'->$}}

\DeclareSymbol{2',0',0'_1}{16}{ \draw [kernel2] (5,4) circle(0.35) -- (0,4); \draw(6.5,4) node{\scriptsize $u_1$}; \draw [kernel2] (0,8) node {$\<2'->$}  -- (0,4); \draw [kernel1] (0,4) -- (0,0); \draw (0,0) circle(0.35); \draw (0,4) node {$\<0'->$}; \draw (1,0) node{\scriptsize $x$}; }

\DeclareSymbol{2',-1',1'_1}{-3}{ \draw [testfcn] (0,0) -- (4.5,-4.5); \draw (5,-5) node[root]{}; \draw (0,4) circle(0.35); \draw [kernel1] (0,4) -- (0,0); \draw (0,0) node {$\<1'->$}; \draw (1,4) node{\scriptsize $y$} ;}
\DeclareSymbol{2',-1',D1'_1}{-6}{ \draw [testfcn] (0,0) -- (4.5,-4.5); \draw (5,-5) node[root]{}; \draw (0,4) circle(0.35); \draw [kernel1] (0,4) -- (0,0); \draw (-2,0) node {$\<D1'->$}; \draw (1,4) node{\scriptsize $y$} ;}
\DeclareSymbol{2',-1',D1'y'_1}{-6}{ \draw [testfcn] (0,0) -- (4.5,-4.5); \draw (5,-5) node[root]{}; \draw (0,4) circle(0.35); \draw [kernel1] (0,4) -- (0,0); \draw (-2,0) node {$\<D1'->$}; \draw (1,4.5) node{\scriptsize $y'$} ;}
\DeclareSymbol{1',1',0'_1}{6}{ \draw [testfcn] (0,0) -- (4.5,-4.5); \draw (5,-5) node[root]{}; \draw [kernel2] (5,9)circle(0.35) -- (0,8); \draw (6.5,9) node{\scriptsize $u_1$}; \draw [kernel2] (5,7)circle(0.35) -- (0,8); \draw(6.5,7) node{\scriptsize $u_2$}; \draw [kernel2] (0,8) node{$\<0'->$}  -- (0,4); \draw [kernel1] (0,4) -- (0,0); \draw (0,4) node{$\<1'->$}; \draw [kernel2] (5,0) node[smalldot]{} -- (0,0); }
\DeclareSymbol{DD2',1',1}{6}{ \draw [testfcn] (0,0) -- (4.5,-4.5); \draw (5,-5) node[root]{};\draw [kernel2] (0,8)  -- (0,4); \draw [kernel1] (0,4) -- (0,0); \draw (0,4) node{$\<1'->$}; \draw [kernel2] (5,0) node[smalldot]{} -- (0,0);\draw (-1.8,8) node{$\<DD2'->$}  }
\DeclareSymbol{DD2',D_{2}1',1}{6}{ \draw [testfcn] (0,0) -- (4.5,-4.5); \draw (5,-5) node[root]{};\draw [kernel2] (0,8)  -- (0,4); \draw [kernel1] (0,4) -- (0,0); \draw (-2.2,4) node{$\<D_{2}1'->$}; \draw [kernel2] (5,0) node[smalldot]{} -- (0,0);\draw (-1.8,8) node{$\<DD2'->$}  }
\DeclareSymbol{D_{1}2',D_{2}1',1}{6}{ \draw [testfcn] (0,0) -- (4.5,-4.5); \draw (5,-5) node[root]{};\draw [kernel2] (0,8)  -- (0,4); \draw [kernel1] (0,4) -- (0,0); \draw (-2.2,4) node{$\<D_{2}1'->$}; \draw [kernel2] (5,0) node[smalldot]{} -- (0,0);\draw (-2.2,8) node{$\<D_{1}2'->$}}
\DeclareSymbol{1',1',0'_2}{6}{ \draw [testfcn] (0,0) -- (4.5,-4.5); \draw (5,-5) node[root]{}; \draw [kernel2] (5,8)circle(0.35) -- (0,8); \draw (6.5,8) node{\scriptsize $u_1$}; \draw [kernel2] (5,0)circle(0.35) -- (0,0); \draw(6.5,0) node{\scriptsize $u_2$}; \draw [kernel2] (0,8) node{$\<1'->$}  -- (0,4); \draw [kernel1] (0,4) -- (0,0); \draw (0,4) node{$\<1'->$}; }
\DeclareSymbol{1'u1,1',0u2}{6}{ \draw [testfcn] (0,0) -- (4.5,-4.5); \draw (5,-5) node[root]{}; \draw [kernel2] (5,8)circle(0.35) -- (0,8); \draw (6.5,8) node{\scriptsize $u_1$}; \draw [kernel2] (5,0)circle(0.35) -- (0,0); \draw(6.5,0) node{\scriptsize $u_2$}; \draw [kernel2] (0,8) node{$\<1'->$}  -- (0,4); \draw [kernel1] (0,4) -- (0,0); \draw (0,4) node{$\<1'->$}; }
\DeclareSymbol{1',1'u1,0u2}{6}{ \draw [testfcn] (0,0) -- (4.5,-4.5); \draw (5,-5) node[root]{}; \draw [kernel2] (5,8)circle(0.35) -- (0,4); \draw (6.5,8) node{\scriptsize $u_1$}; \draw [kernel2] (5,0)circle(0.35) -- (0,0); \draw(6.5,0) node{\scriptsize $u_2$}; \draw [kernel2] (0,8) node{$\<1'->$}  -- (0,4); \draw [kernel1] (0,4) -- (0,0); \draw (0,4) node{$\<1'->$}; }
\DeclareSymbol{D_{3}1'u1,1',0u2}{6}{ \draw [testfcn] (0,0) -- (4.5,-4.5); \draw (5,-5) node[root]{}; \draw [kernel2] (5,8)circle(0.35) -- (0,8); \draw (6.5,8) node{\scriptsize $u_1$}; \draw [kernel2] (5,0)circle(0.35) -- (0,0); \draw(6.5,0) node{\scriptsize $u_2$}; \draw [kernel2] (0,8)   -- (0,4); \draw [kernel1] (0,4) -- (0,0); \draw (0,4) node{$\<1'->$};\draw (-2.3,8) node{$\<D_{3}1'->$} }
\DeclareSymbol{D_{3}1'u1,D_{3}1',0u2}{6}{ \draw [testfcn] (0,0) -- (4.5,-4.5); \draw (5,-5) node[root]{}; \draw [kernel2] (5,8)circle(0.35) -- (0,8); \draw (6.5,8) node{\scriptsize $u_1$}; \draw [kernel2] (5,0)circle(0.35) -- (0,0); \draw(6.5,0) node{\scriptsize $u_2$}; \draw [kernel2] (0,8)   -- (0,4); \draw [kernel1] (0,4) -- (0,0); \draw (-2.3,4) node{$\<D_{3}1'->$};\draw (-2.3,8) node{$\<D_{3}1'->$} }
\DeclareSymbol{1'u1,D_{3}1',0u2}{6}{ \draw [testfcn] (0,0) -- (4.5,-4.5); \draw (5,-5) node[root]{}; \draw [kernel2] (5,8)circle(0.35) -- (0,8); \draw (6.5,8) node{\scriptsize $u_1$}; \draw [kernel2] (5,0)circle(0.35) -- (0,0); \draw(6.5,0) node{\scriptsize $u_2$}; \draw [kernel2] (0,8)   -- (0,4); \draw [kernel1] (0,4) -- (0,0); \draw (-2.3,4) node{$\<D_{3}1'->$};\draw (0,8) node{$\<1'->$}}
\DeclareSymbol{D_{1}2',1',0'_2}{6}{ \draw [testfcn] (0,0) -- (4.5,-4.5); \draw (5,-5) node[root]{};\draw [kernel2] (5,0)circle(0.35) -- (0,0); \draw(6.5,0) node{\scriptsize $u_2$}; \draw [kernel2] (0,8)  -- (0,4); \draw [kernel1] (0,4) -- (0,0); \draw (0,4) node{$\<1'->$};\draw (-2.3,8) node{$\<D_{1}2'->$}  }
\DeclareSymbol{D_{1}2',D_{2}1',D_{2}1}{6}{ \draw [testfcn] (0,0) -- (4.5,-4.5); \draw (5,-5) node[root]{};\draw [kernel2] (5,0)circle(0.35) -- (0,0); \draw(6.5,0) node{\scriptsize $u_2$}; \draw [kernel2] (0,8)  -- (0,4); \draw [kernel1] (0,4) -- (0,0); \draw (-2,4) node{$\<D_{2}1'->$};\draw (-2.3,8) node{$\<D_{1}2'->$}}
\DeclareSymbol{DD2',1',D_{2}1}{6}{ \draw [testfcn] (0,0) -- (4.5,-4.5); \draw (5,-5) node[root]{};\draw [kernel2] (5,0)circle(0.35) -- (0,0); \draw(6.5,0) node{\scriptsize $u_2$}; \draw [kernel2] (0,8)  -- (0,4); \draw [kernel1] (0,4) -- (0,0); \draw (0,4) node{$\<1'->$};\draw (-1.8,8) node{$\<DD2'->$}}
\DeclareSymbol{DD2',D_{2}1',D_{2}1}{6}{ \draw [testfcn] (0,0) -- (4.5,-4.5); \draw (5,-5) node[root]{};\draw [kernel2] (5,0)circle(0.35) -- (0,0); \draw(6.5,0) node{\scriptsize $u_2$}; \draw [kernel2] (0,8)  -- (0,4); \draw [kernel1] (0,4) -- (0,0); \draw (-2,4) node{$\<D_{2}1'->$};\draw (-1.8,8) node{$\<DD2'->$}}

\DeclareSymbol{1',1',0'_2'}{6}{ \draw [testfcn] (0,0) -- (4.5,-4.5); \draw (5,-5) node[root]{}; \draw [kernel2] (5,8)circle(0.35) -- (0,4); \draw (6.5,8) node{\scriptsize $u_1$}; \draw [kernel2] (5,0)circle(0.35) -- (0,0); \draw(6.5,0) node{\scriptsize $u_2$}; \draw [kernel2] (0,8) node{$\<1'->$}  -- (0,4); \draw [kernel1] (0,4) -- (0,0); \draw (0,4) node{$\<1'->$}; }

\DeclareSymbol{1',1',0'_2_}{20}{  \draw [kernel2] (5,8)circle(0.35) -- (0,8); \draw (6.5,8) node{\scriptsize $u_1$}; \draw [kernel2] (0,8) --(0,2); \draw (0,2) circle(0.35); \draw (1,2) node{\scriptsize $y$}; \draw [dashed] (0,2).. controls (2,5)..(0,8); }
\DeclareSymbol{E1'u,1'}{20}{  \draw [kernel2] (5,8)circle(0.35) -- (0,8); \draw (6.5,8) node{\scriptsize $u$}; \draw [kernel2] (0,8) --(0,2); \draw (0,2) circle(0.35); \draw (1,2) node{\scriptsize $y$}; \draw [dashed] (0,2).. controls (2,5)..(0,8); }
\DeclareSymbol{E1'u1,1'}{20}{  \draw [kernel2] (5,8)circle(0.35) -- (0,8); \draw (6.5,8) node{\scriptsize $u_1$}; \draw [kernel2] (0,8) --(0,2); \draw (0,2) circle(0.35); \draw (1,2) node{\scriptsize $y$}; \draw [dashed] (0,2).. controls (2,5)..(0,8); }

\DeclareSymbol{E1',1'u}{20}{ \draw [kernel2] (5,8)circle(0.35) -- (0,2); \draw (6.5,8) node{\scriptsize $u$}; \draw [kernel2] (0,8) --(0,2); \draw (0,2) circle(0.35); \draw (1,2) node{\scriptsize $y$}; \draw [dashed] (0,2).. controls (2,5) .. (0,8); }
\DeclareSymbol{E1',1'u1}{20}{ \draw [kernel2] (5,8)circle(0.35) -- (0,2); \draw (6.5,8) node{\scriptsize $u_1$}; \draw [kernel2] (0,8) --(0,2); \draw (0,2) circle(0.35); \draw (1,2) node{\scriptsize $y$}; \draw [dashed] (0,2).. controls (2,5) .. (0,8); }

\DeclareSymbol{1',1',0'_3}{8}{ \draw [testfcn] (0,0) -- (4.5,-4.5); \draw (5,-5) node[root]{}; \draw [kernel2] (5,10)circle(0.35) -- (0,8); \draw [kernel2] (5,8)circle(0.35) -- (0,8); \draw [kernel2] (5,6)circle(0.35) -- (0,8); \draw (6.5,10) node{\scriptsize $u_1$}; \draw (6.5,8) node{\scriptsize $u_2$}; \draw (6.5,6) node{\scriptsize $u_3$}; \draw [kernel2] (5,0)circle(0.35) -- (0,0); \draw(6.5,0) node{\scriptsize $u_4$}; \draw [kernel2] (0,8) node{$\<-1'>$}  -- (0,4); \draw [kernel1] (0,4) -- (0,0); \draw (0,4) node{$\<1'->$}; }

\DeclareSymbol{1',1',0'_4}{7}{ \draw [testfcn] (0,0) -- (4.5,-4.5); \draw (5,-5) node[root]{}; \draw [kernel2] (5,9)circle(0.35) -- (0,8); \draw [kernel2] (5,7)circle(0.35) -- (0,8); \draw [kernel2] (5,4)circle(0.35) -- (0,4); \draw (6.5,9) node{\scriptsize $u_1$}; \draw (6.5,7) node{\scriptsize $u_2$}; \draw (6.5,4) node{\scriptsize $u_3$}; \draw [kernel2] (5,0)circle(0.35) -- (0,0); \draw(6.5,0) node{\scriptsize $u_4$}; \draw [kernel2] (0,8) node{$\<0'->$}  -- (0,4); \draw [kernel1] (0,4) -- (0,0); \draw (0,4) node{$\<0'->$};}

\DeclareSymbol{1',1',0'_5}{6}{ \draw [testfcn] (0,0) -- (4.5,-4.5); \draw (5,-5) node[root]{}; \draw [kernel2] (5,8)circle(0.35) -- (0,8); \draw [kernel2] (5,5)circle(0.35) -- (0,4); \draw [kernel2] (5,3)circle(0.35) -- (0,4); \draw (6.5,8) node{\scriptsize $u_1$}; \draw (6.5,5) node{\scriptsize $u_2$}; \draw (6.5,3) node{\scriptsize $u_3$}; \draw [kernel2] (5,0)circle(0.35) -- (0,0); \draw(6.5,0) node{\scriptsize $u_4$}; \draw [kernel2] (0,8) node{$\<1'->$}  -- (0,4); \draw [kernel1] (0,4) -- (0,0); \draw (0,4) node{$\<-1'>$};}

\DeclareSymbol{1',1',0'_6}{7}{ \draw [kernel2] (5,4) circle(0.35) -- (0,4); \draw (6.5,4) node{\scriptsize $u_1$}; \draw [kernel2] (0,4) node {$\<1'->$}  -- (0,0); \draw (0,0) circle(0.35); \draw (1,0) node{\scriptsize $y$}; }

\DeclareSymbol{1',1',0'_6'}{6}{ \draw [kernel2] (5,4) circle(0.35) -- (0,4); \draw [kernel2] (0,4) node {$\<1'->$} -- (0,0); \draw(0,0) circle(0.35); \draw (6,4) node{\scriptsize $u$}; \draw (1,0) node {\scriptsize $y$};}
\DeclareSymbol{D2',1',0'_6'}{10}{ \draw [kernel2] (0,4)  -- (0,0); \draw(0,0) circle(0.35);\draw (1,0) node {\scriptsize $y$};\draw (-2,4) node {$\<D2'->$}}
\DeclareSymbol{D2',1',0'y'_6'}{6}{ \draw [kernel2] (0,4)  -- (0,0); \draw(0,0) circle(0.35);\draw (1.3,0) node {\scriptsize $y'$};\draw (-2,4) node {$\<D2'->$}}
\DeclareSymbol{DDD2'}{14}{ \draw [kernel2] (5,5) circle(0.35)--(0,4); \draw [kernel2] (5,3) circle(0.35) -- (0,4); \draw [kernel2] (0,8) -- (0,4); \draw (0,8) circle(0.35); \draw (6.5,5) node{\scriptsize $u_1$}; \draw (6.5,3) node{\scriptsize $u_2$}; \draw [kernel2] (0,4) -- (0,0); \draw (0,0) circle(0.35); \draw (1,0) node {\scriptsize $y$}; \draw (1,8) node {\scriptsize $u$}; }

\DeclareSymbol{1',1',0'_7_}{-3}{ \draw [testfcn] (0,0) -- (4.5,-4.5); \draw (5,-5) node[root]{}; \draw [kernel2] (5,1) circle(0.35) -- (0,0); \draw [kernel2] (5,-1) circle(0.35) -- (0,0); \draw (1,4) node{\scriptsize $y$}; \draw(6.5,1) node{\scriptsize $u_1$}; \draw(6.5,-1) node{\scriptsize $u_2$}; \draw [kernel1] (0,4) -- (0,0); \draw(0,4) circle(0.35);}

\DeclareSymbol{1',1',0'_7}{-3}{ \draw [testfcn] (0,0) -- (4.5,-4.5); \draw (5,-5) node[root]{}; \draw [kernel2] (5,0) circle(0.35) -- (0,0); \draw (1,4) node{\scriptsize $y$}; \draw(6.5,0) node{\scriptsize $u_2$}; \draw [kernel1] (0,4) -- (0,0); \draw(0,4) circle(0.35); \draw (0,0) node {$\<0'->$};}

\DeclareSymbol{1',1',0'_7'}{-3}{ \draw [testfcn] (0,0) -- (4.5,-4.5); \draw (5,-5) node[root]{}; \draw [kernel2] (5,0) circle(0.35) -- (0,0); \draw (1,4) node{\scriptsize $y$}; \draw(6.5,0) node{\scriptsize $u_2$}; \draw [kernel1] (0,4) -- (0,0); \draw(0,4) circle(0.35); }
\DeclareSymbol{1',1',0'_7'u2}{-3}{ \draw [testfcn] (0,0) -- (4.5,-4.5); \draw (5,-5) node[root]{}; \draw [kernel2] (5,0) circle(0.35) -- (0,0); \draw (1,4) node{\scriptsize $y$}; \draw(6.5,0) node{\scriptsize $u_2$}; \draw [kernel1] (0,4) -- (0,0); \draw(0,4) circle(0.35); }
\DeclareSymbol{1',1',0'_7'u2'}{-3}{ \draw [testfcn] (0,0) -- (4.5,-4.5); \draw (5,-5) node[root]{}; \draw [kernel2] (5,0) circle(0.35) -- (0,0); \draw (1.2,4) node{\scriptsize $y'$}; \draw(6.5,0) node{\scriptsize $u_2$}; \draw [kernel1] (0,4) -- (0,0); \draw(0,4) circle(0.35); }

\DeclareSymbol{2,1,1}{5}{ \draw [testfcn] (0,0) -- (4.5,-4.5); \draw (5,-5) node[root]{}; \draw [kernel2] (0,8) -- (0,4); \draw [kernel1] (0,4) -- (0,0); \draw [kernel2] (5,0) node[smalldot]{} -- (0,0); \draw [kernel2] (5,4) node[smalldot]{} -- (0,4); \draw [kernel2] (5,9) node[smalldot]{} -- (0,8); \draw [kernel2] (5,7) node[smalldot]{} -- (0,8);}

\DeclareSymbol{2',1',1}{5}{ \draw [testfcn] (0,0) -- (4.5,-4.5); \draw (5,-5) node[root]{}; \draw [kernel2] (0,8) node {$\<2'->$}  -- (0,4); \draw [kernel1] (0,4) -- (0,0); \draw (0,4) node {$\<1'->$}; \draw [kernel2] (5,0) node[smalldot]{} -- (0,0);}
\DeclareSymbol{2',1',S1x}{5}{  \draw [kernel2] (0,6) node {$\<2'->$}  -- (0,2); \draw [kernel1] (0,2) -- (0,-2); \draw (0,2) node {$\<1'->$}; \draw (-2.2,-2) node {\small $\tT^{(1)}\<1'>$};}
\DeclareSymbol{2',S1',S1x}{5}{  \draw [kernel2] (0,6) node {$\<2'->$}  -- (0,2); \draw [kernel1] (0,2) -- (0,-2);\draw (-2.2,2) node {\small $\tT^{(1)}\<1'>$}; \draw (-2.2,-2) node {\small $\tT^{(1)}\<1'>$};}
\DeclareSymbol{2',1',S1}{5}{ \draw [testfcn] (0,0) -- (4.5,-4.5); \draw (5,-5) node[root]{}; \draw [kernel2] (0,8) node {$\<2'->$}  -- (0,4); \draw [kernel1] (0,4) -- (0,0); \draw (0,4) node {$\<1'->$}; \draw (-2.2,0) node {\small $\tT^{(1)}\<1'>$};}

\DeclareSymbol{Du2',1',1}{5}{ \draw  (-2,8) node{$\<D2'->$};\draw [testfcn] (0,0) -- (4.5,-4.5); \draw (5,-5) node[root]{}; \draw [kernel2] (0,8)   -- (0,4); \draw [kernel1] (0,4) -- (0,0); \draw (0,4) node {$\<1'->$}; \draw [kernel2] (5,0) node[smalldot]{} -- (0,0);}

\DeclareSymbol{S2,S1,S1}{5}{ \draw [testfcn] (0,0) -- (4.5,-4.5); \draw (5,-5) node[root]{}; \draw [kernel2] (0,8) -- (0,4); \draw [kernel1] (0,4) -- (0,0); \draw (-3,8) node {$\tT^{(\leq2)}\<2'>$}; \draw (-2.2,4) node {$\tT^{(1)}\<1'>$}; \draw (-2.2,0) node {\small $\tT^{(1)}\<1'>$};}
\DeclareSymbol{S2,S1,S1x}{5}{ \draw [kernel2] (0,6) -- (0,2); \draw [kernel1] (0,2) -- (0,-2); \draw (-3,6) node {$\tT^{(\leq2)}\<2'>$}; \draw (-2.2,2) node {$\tT^{(1)}\<1'>$}; \draw (-2.2,-2) node {\small $\tT^{(1)}\<1'>$};}
\DeclareSymbol{S2',1',1'x}{5}{ \draw [kernel2] (0,6) -- (0,2); \draw [kernel1] (0,2) -- (0,-2); \draw (-3,6) node {$\tT^{(\leq2)}\<2'>$}; \draw (0,2) node {$\<1'->$}; \draw (0,-2) node {$\<1'->$};}
\DeclareSymbol{S2',1',S1'x}{5}{ \draw [kernel2] (0,6) -- (0,2); \draw [kernel1] (0,2) -- (0,-2); \draw (-3,6) node {$\tT^{(\leq2)}\<2'>$}; \draw (0,2) node {$\<1'->$};  \draw (-2.2,-2) node {$\tT^{(1)}\<1'>$};}
\DeclareSymbol{S2',S1',1'x}{5}{ \draw [kernel2] (0,6) -- (0,2); \draw [kernel1] (0,2) -- (0,-2); \draw (-3,6) node {$\tT^{(\leq2)}\<2'>$}; \draw (-2.2,2) node {$\tT^{(1)}\<1'>$}; \draw (0,-2) node {$\<1'->$};}
\DeclareSymbol{2',S1',1'x}{5}{ \draw [kernel2] (0,6) -- (0,2); \draw [kernel1] (0,2) -- (0,-2); \draw (0,6) node {$\<2'->$}; \draw (-2.2,2) node {$\tT^{(1)}\<1'>$}; \draw (0,-2) node {$\<1'->$};}
\DeclareSymbol{2',Du1',1}{5}{\draw (0,8) node {$\<2'->$}; \draw [testfcn] (0,0) -- (4.5,-4.5); \draw (5,-5) node[root]{}; \draw [kernel2] (0,8) -- (0,4); \draw [kernel1] (0,4) -- (0,0); \draw (-2,4) node {$\<D1'->$}; \draw [kernel2] (5,0) node[smalldot]{} -- (0,0); }
\DeclareSymbol{2',1',Du1}{5}{\draw (0,8) node {$\<2'->$}; \draw [testfcn] (0,0) -- (4.5,-4.5); \draw (5,-5) node[root]{}; \draw [kernel2] (0,8) -- (0,4); \draw [kernel1] (0,4) -- (0,0); \draw (0,4) node {$\<1'->$}; \draw [kernel2] (5,0)circle(0.35) -- (0,0); \draw(6.5,0) node{\scriptsize $u$}; }
\DeclareSymbol{Du2',Du1',1}{5}{ \draw  (-2,8) node{$\<D2'->$};\draw [testfcn] (0,0) -- (4.5,-4.5); \draw (5,-5) node[root]{}; \draw [kernel2] (0,8)   -- (0,4); \draw [kernel1] (0,4) -- (0,0);  \draw (-2,4) node {$\<D1'->$}; \draw [kernel2] (5,0) node[smalldot]{} -- (0,0);}
\DeclareSymbol{Du2',1',Du1}{5}{ \draw  (-2,8) node{$\<D2'->$};\draw [testfcn] (0,0) -- (4.5,-4.5); \draw (5,-5) node[root]{}; \draw [kernel2] (0,8)   -- (0,4); \draw [kernel1] (0,4) -- (0,0);  \draw (0,4) node {$\<1'->$}; \draw [kernel2] (5,0)circle(0.35) -- (0,0); \draw(6.5,0) node{\scriptsize $u$};}
\DeclareSymbol{2',Du1',Du1}{5}{ \draw  (0,8) node{$\<2'->$};\draw [testfcn] (0,0) -- (4.5,-4.5); \draw (5,-5) node[root]{}; \draw [kernel2] (0,8)   -- (0,4); \draw [kernel1] (0,4) -- (0,0);  \draw (-2,4) node {$\<D1'->$}; \draw [kernel2] (5,0)circle(0.35) -- (0,0); \draw(6.5,0) node{\scriptsize $u$};}
\DeclareSymbol{Du2',Du1',Du1}{5}{ \draw  (-2,8) node{$\<D2'->$};\draw [testfcn] (0,0) -- (4.5,-4.5); \draw (5,-5) node[root]{}; \draw [kernel2] (0,8)   -- (0,4); \draw [kernel1] (0,4) -- (0,0);  \draw (-2,4) node {$\<D1'->$}; \draw [kernel2] (5,0)circle(0.35) -- (0,0); \draw(6.5,0) node{\scriptsize $u$};}
\DeclareSymbol{1',1',0'_8}{30}{ \draw [kernel2] (5,10)circle(0.35) -- (0,8); \draw [kernel2] (5,8)circle(0.35) -- (0,8); \draw [kernel2](5,6)circle(0.35) -- (0,8); \draw (6.5,10) node{\scriptsize $u_1$}; \draw (6.5,8) node{\scriptsize $u_2$}; \draw (6.5,6) node{\scriptsize $u_3$}; \draw (0,8) circle(0.35); \draw (-1,8) node{\scriptsize$z$}; }
\DeclareSymbol{zu1u3}{30}{ \draw [kernel2] (5,10)circle(0.35) -- (0,8); \draw [kernel2](5,6)circle(0.35) -- (0,8); \draw (6.5,10) node{\scriptsize $u_1$}; \draw (6.5,6) node{\scriptsize $u_3$}; \draw (0,8) circle(0.35); \draw (-1,8) node{\scriptsize$z$}; }

\DeclareSymbol{1',1',0'_9}{3}{ \draw [testfcn] (0,0) -- (4.5,-4.5); \draw (5,-5) node[root]{}; \draw [kernel2] (5,0)circle(0.35) -- (0,0); \draw(6.5,0) node{\scriptsize $u_2$}; \draw [kernel2] (0,8) -- (0,4); \draw [kernel1] (0,4) -- (0,0); \draw (0,4) node{$\<1'->$}; \draw (0,8) circle(0.35); \draw (1,8) node{\scriptsize $z$}; }

\DeclareSymbol{1',1',0'_10}{3}{ \draw [testfcn] (0,0) -- (4.5,-4.5); \draw (5,-5) node[root]{}; \draw [kernel2] (5,0)circle(0.35) -- (0,0); \draw(6.5,0) node{\scriptsize $u_2$}; \draw [kernel2] (0,8) -- (0,4); \draw [kernel1] (0,4) -- (0,0); \draw [kernel2] (5,3)circle(0.35) -- (0,4); \draw [kernel2] (5,5)circle(0.35) -- (0,4); \draw (6.5,3) node{\scriptsize $u_4$};\draw (6.5,5) node{\scriptsize $u_5$}; \draw (0,8) circle(0.35); \draw (1,8) node{\scriptsize $z$}; }

\DeclareSymbol{2'0_}{8}{\draw [kernel2] (0,4) node {$\<2'->$}-- (0,0); \draw (0,0) circle(0.35) ; \draw (1,0) node{\scriptsize $y$}; }
\DeclareSymbol{D_21u1y}{10}{\draw [kernel2] (0,4) -- (0,0); \draw (0,0) circle(0.35) ; \draw (1,0) node{\scriptsize $y$};\draw (0.6,4) node {$\<D_21u1>$} }
\DeclareSymbol{1'u1y}{9}{\draw [kernel2] (0,4) -- (0,0); \draw (0,0) circle(0.35) ; \draw (1,0) node{\scriptsize $y$};\draw (0,4) node{$\<1'->$}; \draw (6.5,4) node{\scriptsize $u_1$};\draw [kernel2] (5,4) circle(0.35) -- (0.55,4)}
\DeclareSymbol{1'uy}{11}{\draw [kernel2] (0,4) -- (0,0); \draw (0,0) circle(0.35) ; \draw (1,0) node{\scriptsize $y$};\draw (0,4) node{$\<1'->$}; \draw (6,4) node{\scriptsize $u$};\draw [kernel2] (5,4) circle(0.35) -- (0.55,4)}
\DeclareSymbol{2'0_2}{7}{\draw [kernel2] (0,4) node {$\<2'>$}-- (0,0); \draw (0,0) circle(0.35) ; \draw (1,0) node{\scriptsize $y$}; }

\DeclareSymbol{0'0_}{7}{\draw [kernel2] (0,4) -- (0,0); \draw (0,0) circle(0.35) ; \draw (1,0) node{\scriptsize $y$}; \draw [kernel2] (5,5) -- (0,4); \draw [kernel2] (5,3) -- (0,4); \draw (5,5) circle (0.35); \draw (5,3) circle (0.35); \draw (6.5,3) node{\scriptsize $u_2$}; \draw (6.5,5) node{\scriptsize $u_1$}; }

\DeclareSymbol{2',1'}{-6}{ \draw [testfcn] (0,-1) -- (4.5,-5.5); \draw (5,-6) node[root]{}; \draw [kernel2] (0,4) node [noise2]{}  -- (0,0); \draw (1,-0.5) node {$(\<1'> - \<1'>^{(\zeta)})$}; }

\DeclareSymbol{2',0'}{-6}{ \draw [testfcn] (0,-1) -- (4.5,-5.5); \draw (5,-6) node[root]{}; \draw [kernel2] (0,4) node [noise2]{}  -- (0,0); \draw (-0.5,-0.5) node {\scriptsize $D_u$\normalsize$(\<1'> - \<1'>^{(\zeta)})$}; }

\DeclareSymbol{1',1'}{-6}{ \draw [testfcn] (0,-1) -- (4.5,-5.5); \draw (5,-6) node[root]{}; \draw [kernel2] (0,4) node [noise2]{}  -- (0,0); \draw (1,-0.5) node {$(\<1'> - \<1'>^{(\zeta)})$}; \draw (-2.3,4) node{\scriptsize $D_u$};}

\DeclareSymbol{1',0'_u}{-6}{ \draw [testfcn] (0,-1) -- (4.5,-5.5); \draw (5,-6) node[root]{}; \draw [kernel2] (0,4) node [noise2]{}  -- (0,0); \draw (-0.5,-0.5) node { \scriptsize$D_u$ \normalsize$(\<1'> - \<1'>^{(\zeta)})$}; \draw (-2.3,4) node{\scriptsize $D_u$};}

\DeclareSymbol{0',1'}{-6}{ \draw [testfcn] (0,-1) -- (4.5,-5.5); \draw (5,-6) node[root]{}; \draw [kernel2] (0,4) node [noise2]{}  -- (0,0); \draw (1,-0.5) node {$(\<1'> - \<1'>^{(\zeta)})$}; \draw (-2.5,4) node{\scriptsize $D^2_{\vec{u}}$}; }
\DeclareSymbol{z,1'}{-6}{ \draw [testfcn] (0,-1) -- (4.5,-5.5); \draw (5,-6) node[root]{}; \draw [kernel2] (0,4)  -- (0,0); \draw (1,-0.5) node {$(\<1'> - \<1'>^{(\zeta)})$}; \draw (0,4) circle(0.35); \draw (1,4) node{\scriptsize $y$} }
\DeclareSymbol{z,D_31'}{-6}{ \draw [testfcn] (0,-1) -- (4.5,-5.5); \draw (5,-6) node[root]{}; \draw [kernel2] (0,4)  -- (0,0); \draw (1,-0.5) node {$(\<1'> - \<1'>^{(\zeta)})$};\draw (-6,-0.5) node{\scriptsize $D_{v}$}; \draw (0,4) circle(0.35); \draw (1,4) node{\scriptsize $y$}}

\DeclareSymbol{1',0'}{-6}{ \draw [testfcn] (0,-1) -- (4.5,-5.5); \draw (5,-6) node[root]{}; \draw [kernel2] (0,4) node [noise2]{}  -- (0,0); \draw (-2.5,4) node{\scriptsize $D_{u_1}$}; \draw (-1,-0.5) node {\scriptsize $D_{u_2}$\normalsize$(\<1'> - \<1'>^{(\zeta)})$}; }

\DeclareSymbol{0',0'_u}{-6}{ \draw [testfcn] (0,-1) -- (4.5,-5.5); \draw (5,-6) node[root]{}; \draw [kernel2] (0,4) node [noise2]{}  -- (0,0); \draw (-1,-0.5) node {\scriptsize $D_{u_2}$\normalsize$(\<1'> - \<1'>^{(\zeta)})$}; \draw (-2.5,4) node{\scriptsize $D^2_{\vec{u}}$}; }

\DeclareSymbol{0}{0}{\draw[white] (-.4,0) -- (.4,0); \draw (0,0)  -- (0,1.2) node[smalldot] {};}
\DeclareSymbol{1}{0}{\draw[white] (-.4,0) -- (.4,0); \draw (0,0)  -- (0,1.5) node[smalldot] {};}
\DeclareSymbol{2}{0}{\draw (-0.5,1.2) node[smalldot] {} -- (0,0) -- (0.5,1.2) node[smalldot] {};}
\DeclareSymbol{11}{0}{\draw (0,1.8) node[smalldot] {} -- (-0.7,0.9) -- (0,0) -- (0.7,1) node[smalldot] {};}
\DeclareSymbol{10}{0}{\draw (0,1.8) node[smalldot] {} -- (-0.8,0.9) -- (0,0);}
\DeclareSymbol{21}{0.7}{\draw (-1,1.8) node[smalldot] {} -- (-0.5,0.9); \draw (0,1.8) node[smalldot] {} -- (-0.5,0.9) -- (0,0) -- (0.5,0.9) node[smalldot] {};}
\DeclareSymbol{20}{0.7}{\draw (-1,1.8) node[smalldot] {} -- (-0.5,0.9); \draw (0,1.8) node[smalldot] {} -- (-0.5,0.9) -- (-0.5,0);}
\DeclareSymbol{210}{1.1}{\draw (-0.5,2.4) node[smalldot] {} -- (-1,1.6);\draw (-1.5,2.4) node[smalldot] {} -- (-0.5,0.8); \draw (0,1.6) node[smalldot] {} -- (-0.5,0.8) -- (-0.5,0);}
\DeclareSymbol{211}{1.1}{\draw (-0.5,2.4) node[smalldot] {} -- (-1,1.6);\draw (-1.5,2.4) node[smalldot] {} -- (-0.5,0.8); \draw (0,1.6) node[smalldot] {} -- (-0.5,0.8) -- (0,0) -- (0.5,0.8) node[smalldot] {};}
\DeclareSymbol{22}{0.7}{\draw (-1.5,1.8) node[smalldot] {} -- (-1,0.9) -- (0,0) -- (1,0.9) -- (1.5,1.8) node[smalldot] {}; \draw (-0.5,1.8) node[smalldot] {} -- (-1,0.9);\draw (0.5,1.8) node[smalldot] {} -- (1,0.9);}

\setlist[itemize]{topsep=3pt,itemsep=1.5pt,parsep=0pt}

\def\d{\partial}

\title{Hairer-Quastel universality for KPZ -- polynomial smoothing mechanisms, general nonlinearities and Poisson noise}
\author{Fanhao Kong, Haiyi Wang and Weijun Xu}
\institute{Peking University, China}

\begin{document}
\allowdisplaybreaks[4]

\maketitle

\begin{abstract}
    We consider a class of weakly asymmetric continuous microscopic growth models with polynomial smoothing mechanisms, general nonlinearities and a Poisson type noise. We show that they converge to the KPZ equation after proper rescaling and re-centering, where the coupling constant depends nontrivially on all details of the smoothing and growth mechanisms in the microscopic model. This confirms some of the predictions in \cite{HQ}, and provides a first example of Hairer-Quastel type with both a generic nonlinearity (non-polynomial) and a non-Gaussian noise. 

    The proof builds on the general discretisation framework of regularity structures (\cite{general_discrete}), and employs the idea of using the spectral gap inequality to control stochastic objects as developed and systematised in \cite{Otto_spectral_gap, BPHZ_spectral_gap}, together with a new observation on the specific structure of the (discrete) Malliavin derivatives in our situation. This structure enables us to reduce the control of mixed $L^p$ spacetime norms (of arbitrarily large $p$) by certain $L^2$-norms in spacetime. 
\end{abstract}

\setcounter{tocdepth}{3}
\tableofcontents

\section{Introduction}
The aim of this article is to study the large-scale behaviour of continuous weakly asymmetric microscopic growth models of the type
    \begin{equation}\label{e:micromodel}
		\d_{t}\widetilde{h}=\lL \widetilde{h}+\sqrt{\eps}F(\d_{x}\widetilde{h})+\widetilde{\xi}\;, \quad (t,x) \in \RR^+ \times (\TT / \eps)\;
    \end{equation}
on the one dimensional torus of size $\eps^{-1}$. Here, $\lL$ and $F$ are suitable smoothing and nonlinear growth mechanisms respectively, and $\widetilde{\xi}$ is a smeared out Poisson type noise. The small parameter $\sqrt{\eps}$ in front of the nonlinearity corresponds to the growth being weakly asymmetric. Our main result is that, under quite general assumptions on $\lL$ and $F$, the large-scale behaviour of $\widetilde{h}$ is described by the solution to the KPZ equation. The precise assumptions on $\lL$, $F$ and $\widetilde{\xi}$ will be specified in Section~\ref{sec:intro_results} below.

\subsection{Motivation}
    The 1+1 dimensional KPZ$(a)$ equation on the torus is formally given by
	\begin{equation} \label{e:KPZ}
		\d_{t} h = \d_{x}^{2} h + a (\d_{x} h)^{2} + \xi\;, \quad (t,x) \in \RR^+ \times \TT.
	\end{equation}
	Here, $\xi$ is the one dimensional space-time white noise, and $a \in \RR$ is the coupling constant that describes the strength of the asymmetry. 

    Due to singularity of $\xi$, \eqref{e:KPZ} is not classically well-posed. A rigorous solution theory had been sought for a long time. By now, there are a number of ways to make rigorous sense of this equation, including the Cole-Hopf transform (\cite{BG97}), energy solution (\cite{Energy, Energy_unique}), pathwise solutions via rough paths (\cite{HairerKPZ}), regularity structures (\cite{Hai14a}), or para-controlled distributions (\cite{GIP12, GP17}), and renormalisation group approaches (\cite{Kupiainen_Marcozzi_KPZ, Duch_flow_general}). The most relevant ones to us are the pathwise solution notions provided by regularity structures and para-controlled distributions. These frameworks can now treat a very large class of singular equations far beyond the current case. In the particular example of KPZ, it states that there exists a sequence $C_\eps = \frac{c}{\eps} + \oO(1)$ such that the solution $h_\eps$ to the regularised and renormalised equation
    \begin{equation*}
        \d_t h_\eps = \d_x^2 h_\eps + a (\d_x h_\eps)^2 + \xi_\eps - C_\eps
    \end{equation*}
    converges to a one-dimensional family of limits as $\eps\rightarrow0$. This family of limits is parametrised by the $\oO(1)$ quantity in $C_\eps$, and is independent of the regularisation. We denote this family of limits by the KPZ$(a)$ solutions. 

    One reason to study the KPZ equation is that it is expected to be a universal model for weakly asymmetric interface growth. We refer to \cite[Section~1]{Arka_Chatterjee} for survey of recent progresses and relevant literatures. In the current article, we focus on the Hairer-Quastel type model \eqref{e:micromodel} proposed in \cite{HQ}, where the authors considered the case $\lL = \Delta$, $F$ arbitrary even polynomial and $\widetilde{\xi}$ space-time Gaussian field with smooth and short range correlations. They showed that there exists $C_\eps \rightarrow +\infty$ such that the rescaled and re-centered macroscopic process
        \begin{equation}
        \label{e:rescale}
            h_{\eps}(t,x)\coloneqq\sqrt{\eps} \widetilde{h}(t/\eps^2, x/\eps) - C_{\eps} t\;
	\end{equation}
    converges to the KPZ$(a)$ solutions as $\eps\rightarrow0$. One interesting point is that the value of $a$ is a linear combination of coefficients of all terms in $F$, not just its quadratic term. \cite{HQ} also proposed a number of possible extensions, including $F$ being a general function, $\widetilde{\xi}$ being non-Gaussian, and $\lL$ being a general smoothing operator. 
    
    Some of these extensions have been achieved so far, including either general non-polynomial $F$ or non-Gaussian $\widetilde{\xi}$ in the microscopic model \eqref{e:micromodel}. In \cite{KPZCLT}, the authors showed that similar universality results hold for even polynomial $F$ and general non-Gaussian noise $\widetilde{\xi}$. Later, \cite{HX19} extended \cite{HQ} to general nonlinear functions $F$ with sufficient regularity, which was further improved in \cite{kong_zhao}. But both \cite{HX19} and \cite{kong_zhao} need to assume $\widetilde{\xi}$ being Gaussian. 
    
    With the notion of energy solution, \cite{HQ_stationary} showed the convergence for Lipschitz $F$ and Gaussian $\widetilde{\xi}$ (white in time and smooth in space) with stationary (Brownian bridge) initial data. Later, \cite{HQ_non_stationary} removed the stationarity assumption. The convergences here (to the energy solution) are in law instead of pathwise. 
    
    There are parallel pathwise results for dynamical $\Phi^4_3$ as universal limit for $3$D weakly phase coexistence models. Convergence from microscopic models with polynomial nonlinearity and Gaussian noise was shown in \cite{Phi4_poly}, following the general strategy in \cite{HQ}. Then it was extended to general non-Gaussian noise with polynomial nonlinearity (\cite{Phi4_non_Gaussian}), and general nonlinearity but with Gaussian noise (\cite{Phi4_general}). These are $\Phi^4_3$ counter-parts to \cite{HQ, KPZCLT, HX19} in the KPZ equation, but the techniques in treating general non-polynomial nonlinearities and Gaussian noises in \cite{HX19} and \cite{Phi4_general} are very different. \cite{phi4_smoothing} treated the situation with a general smoothing mechanism, but restricted to polynomial nonlinearity and Gaussian noise. 

    \begin{rmk}
        The model \eqref{e:micromodel} belongs to the weak asymmetry regime. \cite{HQ} also considered intermediate disorder regime, which has a different scaling than \eqref{e:micromodel}. For intermediate disorder scaling, only the quadratic behaviour of $F$ near the origin appears in the limit (and higher order terms all vanish). Hence, situations with both non-polynomial $F$ and non-Gaussian $\widetilde{\xi}$ in this scaling regime are more accessible (see \cite{Arka_Chatterjee} which covers such a situation). The techniques developed in \cite{HQ, KPZCLT} can also be applied to treat general situations in this scaling. The situation for weakly asymmetric regime is different; see discussions below. 
    \end{rmk}
    
    Back to the weakly asymmetric model \eqref{e:micromodel}, to summarise, the techniques developed so far cover situations where either $F$ being a general nonlinear function (non-polynomial) or $\widetilde{\xi}$ being a non-Gaussian noise, but unfortunately not both. For non-polynomial $F$, the main difficulty is that the stochastic objects have \textit{infinite} chaos-like expansions (in contrast to finite expansions in polynomial situation). Controlling each term in the series separately (with the general cumulant bounds in \cite{rs_analytic}) will lead to a non-summable series, unless one imposes very strong assumption on $F$ (e.g., its Fourier transform has compact support). If $\widetilde{\xi}$ is Gaussian, the problem was resolved independently in \cite{Phi4_general} via Malliavin calculus methods, and in \cite{HX19} via a clustering argument. It is not clear how these arguments could be extended to general non-Gaussian noise. The case with both a general nonlinearity and non-Gaussian noise was still open (see \cite[Remark~6.1]{progress_in_spde}), even for $\lL=\Delta$. 

    The works \cite{HX19} and \cite{Phi4_general} rely on different aspects of Gaussianity. While it is uneasy to extend to general non-Gaussian situations, it is reasonable to expect from \cite{Phi4_general} that one might possibly cover certain non-Gaussian noises that have a suitable Malliavin calculus\footnote{This was suggested to the third author by Martin Hairer several years ago.}. Recently, \cite{Otto_spectral_gap} and \cite{BPHZ_spectral_gap} developed systematic ways to control various singular stochastic objects based on a spectral gap inequality assumption. Hence, it is natural for us to re-visit \eqref{e:micromodel} with general non-polynomial $F$ and a Poisson type noise. Based on these ideas, we still need to resolve two additional difficulties in our situation: one from the specific form of the spectral gap inequality for Poisson, and the other from $F$ being generic (non-polynomial). We will come back with more discussions in Section~\ref{sec:intro_results} below.

    \subsection{Main Result}
    \label{sec:intro_results}
    
    The main result of this article is to prove a weak universality statement from the microscopic model \eqref{e:micromodel} with general $F \in \cC^{2+}$ and a Poisson type noise $\widetilde{\xi}$\footnote{Both $\widetilde{h}$ and $\widetilde{\xi}$ in \eqref{e:micromodel} depend on $\eps$ since they are defined on $\RR \times (\TT / \eps)$. We omit the $\eps$ for notational simplicity.}. With the general discretisation framework \cite{general_discrete}, we also extend $\lL$ to polynomial smoothing mechanisms of the form $\lL = - \qQ (i \d_x)$\footnote{This means $\widehat{\lL f}(k) = -\qQ(-2\pi k) \widehat{f}(k)$ for $k\in \ZZ$.} for polynomial $\qQ$ satisfying Assumption~\ref{as:Q} below. Applying the same rescaling and re-centering procedure as in \eqref{e:rescale} (but with a different $C_\eps$ in general), we derive the equation for $h_{\eps}$ as
	\begin{equation} \label{e:macro}
		\d_{t} h_{\eps} = \lL_\eps h_{\eps} + \eps^{-1} F( \sqrt{\eps} \d_{x}h_{\eps}) + \xi_\eps - C_{\eps}\;, 
	\end{equation}
	where $\xi_\eps \coloneqq \eps^{-\frac{3}{2}} \widetilde{\xi}(t/\eps^2,x/\eps)$ is a non-Gaussian approximation to the space-time white noise $\xi$, and $\lL_{\eps} \coloneqq -\eps^{-2} \qQ(i\eps\d_x)$ in the sense that $\widehat{\lL_\eps f}(k) = - \eps^{-2} \qQ(-2\pi \eps k) \widehat{f}(k)$ for $k \in \ZZ$. We first give our precise assumptions on $\qQ$ and $F$. 

    \begin{assumption} \label{as:Q}
		$\qQ: \RR \rightarrow \RR$ is a positive (except $\qQ(0)=0$) even polynomial with $\frac{1}{2} \qQ''(0)=1$.
	\end{assumption}

    \begin{assumption}
    \label{as:F}
    $F: \RR \rightarrow \RR$ is an even function. Furthermore, there exist $C, M>0$ and $\beta\in(0,1)$ such that
    \begin{equation*}
        \sup_{0 \leq \ell \leq 2} |F^{(\ell)}(w)| \leq C (1+|w|)^{M}, \quad |F''(w+h) - F''(w)| \leq C |h|^\beta (1+|w|+|h|)^{M}
    \end{equation*}
    for all $w,h \in \RR$.
    \end{assumption}

    \begin{rmk}
        The assumption $\qQ(0) = \qQ'(0) =0$ (the latter implied by $\qQ$ being an even polynomial) and $\qQ''(0)>0$ guarantees that $\lL_\eps$ approximates the Laplacian (with normalised coefficient $\frac{1}{2} \qQ''(0)=1$). Positivity of $\qQ$ is necessary for $\lL_\eps$ being a ``smoothing" operator at all scales. As indicated in \cite[Remark~4.11]{general_discrete}, these assumptions imply that $\{e^{t\lL_\eps}\}$ has the same singularity as the standard kernel in the region $\eps$-away from the origin. 

        The assumption that $\qQ$ being a polynomial is mainly for its Green's function to satisfy the bounds in the framework of regularity structures. It might be possible to relax to general even functions, though it is not clear to us at this moment how to achieve it technically. 

        The assumption on $F$ is same as that of \cite{kong_zhao}. It is a heuristic threshold for pathwise convergence -- minimal requirement for a direct Taylor expansion argument in the PDE part (Theorem~\ref{thm:abstract_pde} below). 
    \end{rmk}
    
   We now specify the Poisson type noise $\widetilde{\xi}$ in \eqref{e:micromodel}. It is a primary example of non-Gaussian noise (see \cite[Example~2.3]{KPZCLT}). Let $\eta^{(\eps)}$ be a Poisson point process on $\RR\times(\TT/\eps)$ with uniform intensity measure. Let $\theta: \RR^2 \rightarrow \RR$ be a smooth spacetime function that is symmetric in the spatial variable $x$ and with decay
   \begin{equation*}
        |\theta(t,x)| \lesssim  (1+ \sqrt{t} + |x|)^{-4-\delta_0}
    \end{equation*}
    for some $\delta_0>0$. We also assume $\theta$ is normalised in the sense that $\int_{\RR^2} \theta(t,x) dt dx =1$. For $\eps\in(0,1)$, let
    \begin{equation*}
        \widetilde{\xi}(t,x) = \int_{\RR \times (\TT/\eps)} \theta^{(\eps)}(t-s,x-y) \, \eta^{(\eps)}(ds,dy) - 1,
    \end{equation*}
    where $\theta^{(\eps)}(t,x) \coloneqq\sum_{k\in\ZZ} \theta(t,x+k/\eps)$ be its $\frac{1}{\eps}$-periodisation in space. As mentioned earlier, $\widetilde{\xi}$ also depends on $\eps$, but we omit it in notation for simplicity. 
    
    \begin{rmk}
        The symmetry of $\theta$ in its spatial variable ensures that the appropriate rescaling procedure of $h$ is given by \eqref{e:rescale}. Otherwise, one needs to include a shift in the space variable (see \cite[Theorem~1.3]{KPZCLT}). 
    \end{rmk}

    The macroscopic noise $\xi_\eps$ in \eqref{e:macro} is defined by
    \begin{equation*}
        \xi_\eps (t,x) := \eps^{-\frac{3}{2}} \widetilde{\xi} (t/\eps^2, x/\eps)\;, \qquad (t,x) \in \RR \times \TT\;.
    \end{equation*}
    Let $P_\eps$ be the Green's function of $\d_t - \lL_\eps$ on $\RR \times \TT$, and $P_\eps'$ be its derivative with respect to the spatial variable. An essential building block for all the stochastic objects in this article is the stationary field
    \begin{equation*}
        \Psi_\eps := P_\eps' * \xi_\eps\;,
    \end{equation*}
    where $*$ denotes convolution in both space and time. Define
    \begin{equation} \label{e:a_eps}
        a_\eps := \frac{1}{2} \E F'' \big( \sqrt{\eps} \Psi_\eps \big) = \frac{1}{2}\EE F'' \big( \sqrt{\eps} P_\eps' * \xi_\eps \big) \;.
    \end{equation}
    This expression does not depend on the spacetime point $(t,x))$ by stationarity. We will show in Proposition~\ref{prop:coupling_constant} below that $a_\eps \rightarrow a$ for some $a \in \RR$ as $\eps \rightarrow 0$. Our main theorem is that the macroscopic process $h_\eps$ in \eqref{e:macro} converges to the KPZ($a$) family. One interesting point is that although the smoothing operator in the limiting equation is the Laplacian with coefficient $\frac{1}{2} \qQ''(0) = 1$, the coupling constant $a$ actually depends on all coefficients from $\qQ$ (see Remark~\ref{rmk:coupling} for a probabilistic representation of $a$). We first state our main theorem below. 
    
    \begin{thm} \label{thm:main}
        Suppose $F$ satisfies Assumption~\ref{as:F} and $\qQ$ satisfies Assumption~\ref{as:Q}. Let $h_\eps(0,\cdot)\in\cC^{\gamma,\eta}_\eps$ be a sequence of functions on $\TT$ and $h(0,\cdot)\in\cC^{\eta}$ such that $\|h_\eps(0,\cdot), h(0,\cdot)\|_{\gamma,\eta;\eps} \rightarrow 0$ in the sense of \cite[Eq.(3.6)]{HX19}\footnote{Roughly speaking, this norm means $\cC^{\gamma}$ at scales larger than $\eps$, and $\cC^\eta$ at scales smaller than $\eps$. Since our main focus is the bounds for the stochastic objects, we do not repeat details for setting up the function spaces, but instead refer to relevant literature for precise definitions.} for some $\gamma \in (\frac{3}{2}, \frac{5}{3})$ and $\eta \in (\frac{1}{2}-\frac{1}{M+4}, \frac{1}{2})$. Then there exists $C_\eps\rightarrow +\infty$ such that the solution $h_\eps$ to \eqref{e:macro} with initial data $h_\eps(0,\cdot)$ converges in law to the KPZ($a$) family with initial data $h(0,\cdot)$ in $\cC^{\eta}([0,1] \times \TT)$, where the coupling constant $a$ is given by \eqref{e:coupling}.
    \end{thm}
    \begin{proof}
        Once the assumptions of Theorem~\ref{thm:abstract_pde} are satisfied, the convergence to the desired limit will follow from continuity of the reconstruction operator as in \cite[Theorem~5.7]{HX19}. The assumptions of Theorem~\ref{thm:abstract_pde} (convergence of models in regularity structures) follow from Theorem~\ref{thm:convergence}. Hence, we have the desired convergence of $h_\eps$ to the KPZ$(a)$ solution $h$. That the coupling constant $a$ is the limit of $a_\eps$ and has the representation \eqref{e:coupling} is proved in Proposition~\ref{prop:coupling_constant}. 
    \end{proof}

    \begin{rmk} \label{rmk:coupling}
    The coupling constant $a$ has the following probabilistic representation. Let $\bar{\eta}$ be a Poisson point process on $\RR^2$ with uniform intensity measure, and
    \begin{equation*}
        \bar{\xi}(t,x) = \int_{\RR^2} \theta (t-s,x-y) \bar{\eta}(ds,dy) - 1\;, \qquad (t,x) \in \RR^2\;,
    \end{equation*}
    where $\theta$ is the same spacetime function mentioned above. Let $\bar{P}$ denote the Green's function of $\d_t - \lL$ on $\RR^2$, and $\bar{P}'$ denote its spatial derivative. Then we have
	\begin{equation} \label{e:coupling}
		a = \frac{1}{2} \EE F''\big( (\bar{P}'\ast\bar{\xi})(0) \big)\;,
	\end{equation}
    where $\ast$ is the space-time convolution. This expression suggests that the limiting coupling constant $a$ depends on all details of $F''$ and $\lL$: even if $\lL_\eps$ formally approximates the Laplacian (with coefficient $\frac{1}{2} \qQ''(0)=1$), $\bar{P}'$ depends on all higher coefficients of $\qQ$. 
    \end{rmk}

    The key to the proof of Theorem~\ref{thm:main} is to show convergence of stochastic objects built from non-polynomial $F$ and the non-Gaussian $\xi_\eps$ in Theorem~\ref{thm:convergence}. The systematic bounds developed in \cite{Otto_spectral_gap} and \cite{BPHZ_spectral_gap} provide a possible way to do it since $\xi_\eps$ has a suitable spectral gap inequality (see also \cite{random_models_rs} which revisited the results of \cite{Otto_spectral_gap, BPHZ_spectral_gap} in a slightly different setup). 
    
    However, there are two differences in our situation that result in additional subtleties. The first one is specific to the spectral gap inequality for Poisson --- it controls the $p$-th moment of a random variable built from a Poisson point process in terms of the $p$-th moment of the mixed $L^2$ and $L^p$ spacetime norms of its Malliavin derivative, in contrast to the mere $L^2$ spacetime norm in the usual spectral gap inequality. This requires us to control $L^p$ spacetime norms for high order Malliavin derivatives of our stochastic objects for arbitrarily large $p$, which seems to be an extremely complicated task. At this point, a key observation is that with the particular structure of the Malliavin derivative in our situation (related to approximate heat kernels), its mixed $L^2$ and $L^p$ norms can in fact be controlled by its $L^2$ norms only (with certain modifications). This allows us to proceed after applying the spectral gap inequality. This bound is in Lemma~\ref{lem:Lvecpcontrol} below, and is applied to various situations arising from our objects (see the lemmas after that). 

    Second, even with the spectral gap inequality, there is another difficulty for stochastic objects consisting of at least two appearances of $F$ (or its derivatives). Since $F$ is not a polynomial, and any high order Malliavin derivative of such an object necessarily contains terms in which no derivative hits on some of the appearances of $F$. Hence, no regularity gain could happen for those parts of the stochastic object. This is in contrast to the polynomial situation, where sufficiently many Malliavin differentiation necessarily annihilates the object. This is the main reason that we did not have a systematic inductive argument as in \cite{BPHZ_spectral_gap}, but instead cut the objects at hand into various sub-processes in an ad hoc and not necessarily the most canonical way (see for example the objects in Lemma~\ref{lem:210_holder_boundlem} and in \eqref{e:2',-1',1'_1zeta}). These sub-processes are controlled in a way that even if some of them have ``naive" and seemingly useless bounds, one can leverage the joint effects of kernel convolution and multiplication of $\eps$'s so that their combination as a whole process has the correct bounds. Relevant bounds for these sub-processes are derived in Section~\ref{sec:211_preliminaries}, and are combined together in Section~\ref{sec:211_main}. Similar cutting procedures have been used in \cite{Phi4_general} for second-order processes from $\Phi^4_3$ with Gaussian noise. In the KPZ case, there is a third order process, and hence the cutting and composition argument is much more involved. 

    To summarise, to the best of our knowledge, Theorem~\ref{thm:main} provides a first example for Hairer-Quastel weak universality of the type \eqref{e:micromodel} with non-polynomial $F$ and non-Gaussian $\widetilde{\xi}$. Furthermore, we also cover a general (polynomial) smoothing mechanism $\lL$ with the help of the general discrete regularity structure framework \cite{general_discrete}. On the other hand, it is restricted to a specific type of Poisson noise, and the bounds for the stochastic objects are still somewhat technical and ad hoc. We hope the methods could be generalised and systematised in the future.

\subsection*{Notations}
    In what follows, we let $\TT = \RR / \ZZ$ be the circle of length $1$, and write
    \begin{equation*}
        \TT_\eps := \TT / \eps
    \end{equation*}
    be the circle of length $1/\eps$. Since integration in both domains will be frequently encountered, we use letters $x,y,z,r$ to denote spacetime points in $\RR \times \TT$, and $u,v,w$ to denote spacetime points in the larger domain $\RR \times \TT_\eps$. 
    
    Since operations and bounds in space-time will be considered as a whole (either in $\RR \times \TT$ or $\RR \times \TT_\eps$), we do not use different letters to distinguish space and time components of points (except that they have different scaling behaviours). Instead, we use subscripts $0$ and $1$ in the letter to denote time and space components respectively, for example, 
    \begin{equation*}
        x = (x_0, x_1) \in \RR \times \TT\;, \qquad u = (u_0, u_1) \in \RR \times \TT_\eps\;.
    \end{equation*}
    We use $|\cdot|$ to denote the parabolic metric on spacetime domains so that
    \begin{equation*}
        |x| = |(x_0,x_1)| = \sqrt{|x_0|}+|x_1|,
    \end{equation*}
    and for $\lambda>0$, we denote multiple of $x$ in the parabolic scaling by $\lambda$ as
    \begin{equation*}
        \lambda x = (\lambda^2 x_0, \lambda x_1)\;.
    \end{equation*}
    Throughout, we fix the function $\theta: \RR^2 \rightarrow \RR$ with decay
    \begin{equation} \label{e:theta_decay}
        |\theta(x)| \lesssim (1 + |x|)^{-(4+\delta_0)}
    \end{equation}
    for some $\delta_0 > 0$, symmetric in its space component (that is, $\theta(x_0, x_1) = \theta(x_0, - x_1)$) and normalised such that $\int_{\RR^2} \theta = 1$. For $\eps\in(0,1)$, we use
    \begin{equation} \label{e:theta_periodic}
        \theta^{(\eps)}(u) := \sum_{k \in \ZZ} \theta (u_0, u_1 + k/\eps)
    \end{equation}
    to denote its periodisation on the spatial torus $\TT_\eps$. 
    
    For $\eps\in(0,1)$, let $\eta^{(\eps)}$ be the Poisson point process on $\RR \times \TT_\eps$ with uniform intensity (with respect to Lebesgue measure), and define the stationary field $\widetilde{\xi}$ on $\RR \times \TT_\eps$ and its rescaled version $\xi_\eps$ on $\RR \times \TT$ by
    \begin{equation} \label{e:noise}
        \widetilde{\xi}(u) = \int_{\RR \times \TT_\eps} \theta^{(\eps)}(u-v) \, \eta^{(\eps)}(dv) - 1\;, \quad \xi_\eps(x_0, x_1) := \eps^{-\frac{3}{2}} \widetilde{\xi}(x_0 / \eps^2, x_1 / \eps)\;.
    \end{equation}
    As mentioned above, $\widetilde{\xi}$ also depends on $\eps$, though we omit it in the notation for simplicity. 
  
    We denote the Green's functions of the operators $\partial_t-\lL_\eps$ on $\RR\times\TT$ for $\eps\in(0,1]$ by $P_\eps$, and $P_0$ corresponds to the heat kernel on $\RR\times\TT$. For $\eps\in[0,1]$, $K_\eps$ represents a proper truncation of $P_\eps$ at a neighbourhood of the origin. The free field
    \begin{equation} \label{e:free_field}
        \Psi_\eps(x) := \int_{\RR \times \TT} P_{\eps}'(x-y) \xi_\eps(y) \, dy
    \end{equation}
    is the building block of all the stochastic objects. 
    
    For every $\alpha>0$, we use $\bar{C}_c^\alpha$ to denote the class of test functions
    \begin{equation} \label{e:test_function_space}
     \Big\{ \varphi\in\cC^\infty_c(\RR^+ \times \TT) \;|\;\text{supp}\;\varphi\subset [0,1]\times\TT,\; \|\varphi\|_{\cC^\alpha}\leq1\; \Big\},
    \end{equation}
    where $\|\cdot\|_{\cC^\alpha}$ is the H\"older-$\alpha$ norm. For $z \in \RR \times \TT$ and $\lambda>0$, we define the re-centered and rescaled test function $\varphi_z^\lambda$ by
    \begin{equation} \label{e:test_function}
        \varphi_z^\lambda (x) := \lambda^{-3} \varphi \big( (x-z) / \lambda \big) = \lambda^{-3} \varphi \Big( \frac{x_0 - z_0}{\lambda^2}, \frac{x_1 - z_1}{\lambda} \Big)\;.
    \end{equation}
    We further write $\varphi^\lambda$ for $\varphi_0^\lambda$. 

    We use the notation $A\lesssim B$ to represent that there exists a proportionality constant $C>0$ such that $A\leq CB$. Moreover, the notation $\lesssim_n$ implies that the proportionality constant depends on the parameter $n$.

\subsection*{Structure of the article}

The proof of Theorem~\ref{thm:main} is divided into two parts: a PDE part and a stochastic part. In Section~\ref{sec:set_up}, we establish the regularity structure and solve the abstract fixed point problem. In Section~\ref{sec:spectral_gap}, we provide the spectral gap inequality of Poisson point process. Then we demonstrate the convergence of the stochastic terms via the spectral gap inequality in Section~\ref{sec:convergence}.

\subsection*{Acknowledgements}

W. Xu is grateful to Ajay Chandra for valuable discussions on cumulant methods, and to Martin Hairer for suggesting to consider the Poisson type noise for this problem, for which Malliavin calculus techniques can be useful. 

W. Xu was supported by the National Science Foundation China via the standard project grant (no. 8200906145) and the Ministry of Science and Technology via the National Key R\&D Programs of China (no. 2020YFA0712900 and no. 2023YFA1010102). Part of the work was done when the authors were visiting NYU Shanghai in Autumn 2023. We thank its Institute of Mathematical Sciences for hospitality.

\section{Regularity Structures}
\label{sec:set_up}

This section sketches the set up of regularity structures. We prove desired bounds on the kernel that is consistent with the assumptions in the general framework of \cite{general_discrete}. The definition of the models are essentially the same as in \cite{HX19}. In Theorem~\ref{thm:abstract_pde}, we give the abstract fixed point theorem corresponding to the equation~\eqref{e:macro}.

\subsection{Integration kernel} \label{sec:kernel}
    
    Before introducing the regularity structures for our case, we present the decomposition of the Green's function $P_\eps$ corresponding to the operator $\partial_t-\lL_\eps$ on $\RR\times\TT$ (the case $\eps=0$ corresponds to the heat kernel on $\RR\times\TT$) in the following proposition.
    
    \begin{prop} \label{prop:kernel}
        Suppose $\qQ$ satisfies Assumption~\ref{as:Q}. For every $\eps\in[0,1]$, there exist non-anticipated and symmetric functions $K_\eps$ and $R_\eps$ such that
        \begin{equation*}
            P_\eps = K_\eps + R_\eps.
        \end{equation*}
        Furthermore, $K_\eps$ is supported in $[0,1]\times\TT$, and $R_\eps$ is smooth uniformly in $\eps$.

        Consequently, for every $x=(x_0,x_1)\in\RR\times\TT$, $\delta\in(0,1)$ and $m,\ell\in\NN$, we have the bounds
        \begin{equation} \label{e:kernel_singularity'}
            |\d_0^m \d_1^\ell P_\eps(x)| \big( \boldsymbol{1}_{m=0;l=0;|x|\lesssim 1} + \boldsymbol{1}_{m=0;l\geq1} + \boldsymbol{1}_{m\ge1;|x|\gtrsim\eps} \big) \lesssim|x|^{-2m-\ell-1}
        \end{equation}
        and
        \begin{equation} \label{e:kernel_singularity}
            |P'_\eps(x)-P'_0(x)| \lesssim \eps^{\delta}|x|^{-2-\delta},
        \end{equation}
        where the derivative without indication represents the spatial derivative, and the proportionality constants are independent of $x$ and $\eps$. Furthermore, for every $\delta\in[0,1]$, we also have
        \begin{equation} \label{e:kernel_x-y}
            |P'_\eps(x-y) - P'_\eps (-y)| \lesssim
            \left\{ \begin{aligned}
                &\frac{1}{|y|^2}, &|y|\leq \frac{|x|}{2}\\
                &\frac{1}{|x-y|^2}, &|x-y|\leq \frac{|x|}{2}\\
                & \frac{\boldsymbol{1}_{|y|<\eps}}{|y|^2} +  \frac{\boldsymbol{1}_{|y|\geq \eps}|x|^\delta}{|y|^{2+\delta}}, &\text{others}
            \end{aligned}\right.,
        \end{equation}
        where the proportionality constant is independent of $x$, $y\in\RR\times\TT$ and $\eps$. Moreover, all the estimates hold if we replace $P_\eps$ by $K_\eps$.
    \end{prop}
    \begin{proof}
        The proof is essentially the same as \cite[Remark~4.11]{general_discrete}. We only provide details for the estimate \eqref{e:kernel_singularity'}. The estimate \eqref{e:kernel_singularity} can be treated similarly, and the estimate \eqref{e:kernel_x-y} is a direct consequence of \eqref{e:kernel_singularity'}.

        For convenience, we provide the proof for Green's function on the whole space $\RR^2$. The proof for Green's function on $\RR\times \TT$ is similar, which only need to replace the integral with respect to $k$ by the form of summation (except the case $k=l=0$ since in this case the $0$-th Fourier mode does not decay in $x_0$) and use the discrete version of integration by parts. By the definition of $\lL_\eps$, we have
        \begin{equation*}
            \d_0^m \d_1^\ell P_\eps (x) = \int_{\RR} (2\pi i k)^\ell \Big( -\frac{\qQ(2\pi\eps k)}{\eps^2} \Big)^m \exp\bigg(-\frac{\qQ(2\pi\eps k)}{\eps^2} x_0 \bigg) e^{2\pi i kx_1} dk.
        \end{equation*}
        By $e^{-r} \lesssim r^{-m}$ for $r>0$ and $\qQ(r)\gtrsim r^2$, we obtain $|\d_0^m \d_1^\ell P_\eps (x)| \lesssim |x_0|^{-m-\frac{\ell+1}{2}}$. We then need to demonstrate that $\d_0^m \d_1^\ell P_\eps (x) \lesssim |x_1|^{-2m-\ell-1}$.
        
    First we consider the case $x_0>\eps^2$. Let $c\coloneqq\eps/\sqrt{x_0} \in (0,1)$. We can write $\d_0^m \d_1^\ell P_\eps (x)$ as
    \begin{equation*}
        x_0^{-m-\frac{\ell}{2}} c^{-2m} (-1)^m \int_{\RR} (2\pi i (k\sqrt{x_0}) )^\ell \qQ(2\pi c (k\sqrt{x_0}))^m \exp\big(-c^{-2} \qQ(2\pi c (k\sqrt{x_0}) ) \big) e^{2\pi i kx_1} dk.
    \end{equation*}
    Changing the variable $k\sqrt{x_0}\mapsto k$ and integrating by parts $2m+\ell+1$ times, we get
    \begin{equation*}
        |\d_0^m \d_1^\ell P_\eps (x)| \lesssim |x_1|^{-2m-\ell-1} c^{-2m} \int_{\RR} \Big| \Big( k^\ell \qQ(2\pi ck)^m \exp\big(-c^{-2} \qQ(2\pi ck ) \big) \Big)^{(2m+\ell+1)} \Big| dk.
    \end{equation*}
    Note that every term in the derivative is of the form
    \begin{equation*}
        k^b c^d \Big( \prod_{i=1}^{j} \qQ^{(a_i)} (2\pi ck) \Big) \exp\big(-c^{-2} \qQ(2\pi ck )\big),
    \end{equation*}
    where $a_i,b,d,j\in\ZZ$ with restrictions
    \begin{equation} \label{e:restriction}
        b\ge0\;, \quad 2j-\sum_{i=1}^j a_i +d =2m \quad\text{and}\quad \sum_{i=1}^{j}a_i - b = 2m+1.
    \end{equation}
    Applying $e^{-r} \lesssim (1+r)^{-j}$ to this term, we can bound the absolute value of this term by
    \begin{equation} \label{e:deri_bound}
        \frac{|k|^bc^{d+2j-\sum_{i=1}^{j}a_i}}{(1+|k|)^{\sum_{i=1}^{j}a_i}}  \Big( \prod_{i=1}^{j} \frac{|\qQ^{(a_i)} (2\pi ck)|(c+c|k|)^{a_i}}{c^2+\qQ(2\pi ck )}\Big) \exp\big(-\pi^2 k^2\big).
    \end{equation}
    Since $\qQ$ is a polynomial, we have for $i=1,2,\dots,j$,
    \begin{equation*}
        \frac{|\qQ^{(a_i)} (2\pi ck)|(c+c|k|)^{a_i}}{c^2+\qQ(2\pi ck )}\lesssim1.
    \end{equation*}
    Then \eqref{e:deri_bound} is bounded by $c^{2m}{(1+|k|)^{-(2m+1)}} \exp(-\pi^2 k^2).$
    Therefore, we obtain
    \begin{equation*}
        |\d_0^m \d_1^\ell P_\eps (x)| \lesssim |x_1|^{-2m-\ell-1} \int_{\RR} \exp\big(-\pi^2 k^2\big) dk.
    \end{equation*}
        
    For the case $0<x_0\leq\eps^2$, we define $c\coloneqq\sqrt{x_0}/\eps \in (0,1)$ and then write $\d_0^m \d_1^\ell P_\eps (x)$ as
        \begin{equation*}
        \begin{split}
            x_0^{-m} (\eps^{\frac{n-1}{n}} x_0^{\frac{1}{2n}})^{-\ell} c^{2m} (-1)^m\int_{\RR} &(2\pi i (k(\eps^{\frac{n-1}{n}} x_0^{\frac{1}{2n}}) )^\ell \qQ\big(2\pi c^{-\frac{1}{n}} (k(\eps^{\frac{n-1}{n}} x_0^{\frac{1}{2n}}) ) \big)^m\\
            &\exp\big(-c^{2} \qQ\big(2\pi c^{-\frac{1}{n}} (k(\eps^{\frac{n-1}{n}} x_0^{\frac{1}{2n}}) ) \big) e^{2\pi i kx_1} dk,
        \end{split}
        \end{equation*}
        where $2n$ is the degree of $\qQ$. Changing the variable $k(\eps^{\frac{n-1}{n}}x_0^{\frac{1}{2n}})\mapsto k$ and integrating by parts $2mn+\ell+1$ times, we get
        \begin{equation*}
        \begin{split}
            |\d_0^m \d_1^\ell P_\eps (x)| \lesssim \, &|x_1|^{-2mn-\ell-1} \eps^{2mn-2m} c^{2m}\\
            &\int_{\RR} \Big| \Big( k^\ell \qQ(2\pi c^{-\frac{1}{n}}k)^m
            \exp\big(-c^{2} \qQ(2\pi c^{-\frac{1}{n}}k ) \big) \Big)^{(2mn+\ell+1)} \Big| dk.
        \end{split}
        \end{equation*}
        If $m=0$, then the desired bound follows as above. If $m\geq1$, it suffices to prove the bound for $|x_1|\gtrsim \eps$. Note that every term in the derivative takes the form
        \begin{equation*}
              k^b c^d \Big( \prod_{i=1}^{j} \qQ^{(a_i)} (2\pi c^{-\frac1n}k) \Big) \exp\big(-c^{2} \qQ(2\pi c^{-\frac1n}k )\big),
        \end{equation*}
        where $a_i,b,j\in\ZZ$ and $b\ge0$. Similar to the proof above, this term is bounded by $c^{-2m}(1+|k|)^{-(2mn+1)}\exp(-\pi^{2n}q_nk^{2n})$, where $q_n$ represents the coefficient of the highest order term of $\qQ$. As a result, we obtain
        \begin{equation*}
            |\d_0^m \d_1^\ell P_\eps (x)| \big( \boldsymbol{1}_{m=0} + \boldsymbol{1}_{m=1;|x_1|\gtrsim\eps} \big) \lesssim |x_1|^{-2m-\ell-1} \int_{\RR} \exp\big(-\pi^{2n} q_n k^{2n}\big) dk.
        \end{equation*}
         This concludes the proof.
    \end{proof}

    \begin{cor} \label{cor:kernel}
        The kernel $K_\eps$ satisfies all the assumptions on kernels in \cite[Section~4]{general_discrete}.
    \end{cor}

    \begin{proof}
        As explained in \cite[Remark~4.11]{general_discrete}, this is a direct corollary of \eqref{e:kernel_singularity'}.
    \end{proof}

\subsection{Models} \label{sec:model}

    We now define the regularity structures following the approach in \cite[Section~3.2]{HX19}. In our case, given the infinite family of operators $\lL_\eps$, it is necessary to identify abstract integration operators $\iI$ and $\iI'$ such that the representation $\Pi^\eps$ realises $K_\eps$ for $\iI$ and the spatial derivative $K'_\eps$ for $\iI'$. The definition of realisation can be found in \cite[Definition~4.7]{general_discrete}. Fortunately, as shown in \cite[Theorem~4.18]{general_discrete}, such an $\iI$ exists by extending the regularity structures.

    We use the graphic notations defined by
    \begin{equation*}
    \begin{split}
        &\<1'0> = \iI'(\<1'>)\;,\quad \<2'0> = \iI'(\<2'>)\;,\quad \<1'1'> = \<1'0> \cdot \<1'>\;, \quad \<2'0'> = \<2'0> \cdot \<0'>\;,\\
        &\<2'1'> = \<2'0> \cdot \<1'>\;, \quad \<2'2'0> = \<2'0> \cdot\<2'0>\;,\quad  \<2'2'0'> = \<2'2'0> \cdot\<0'>\;,\quad \<2'1'1'> = \iI'(\<2'1'>) \cdot \<1'>\;.
    \end{split}
    \end{equation*}
    Recall from \eqref{e:free_field} that the free field $\Psi_\eps := P_\eps' * \xi_\eps$ is the building block of all the stochastic objects, where $*$ denotes convolution in spacetime. By stationarity of $\eta^{(\eps)}$, symmetry of $\theta^{(\eps)}$ and anti-symmetry of $P_\eps'$ in their spatial variables, we have
    \begin{equation} \label{e:symmetry_law}
        \Psi_\eps (x) \overset{\text{law}}{=} - \Psi_\eps(x)
    \end{equation}
    for every $\eps>0$ and every $x \in \RR \times \TT$. Note that the equality in law \eqref{e:symmetry_law} holds pointwise but not for the field in general. Consequently, $a_\eps\coloneqq \frac12\EE[F''(\sqrt{\eps}\Psi_\eps)]$ given in \eqref{e:a_eps} also does not depend on the spacetime point. The representation $\Pi^\eps$ for the regularity structures is defined by
    \begin{equation} \label{e:rs_def_1}
    \begin{split}
        &(\Pi^\eps \<0'>)(z) \coloneqq \frac{1}{2a_\eps} F''(\sqrt{\eps}\Psi_\eps(z)) - 1, \qquad (\Pi^\eps \<1'>)(z) \coloneqq \frac{1}{2a_\eps\sqrt{\eps}} F'(\sqrt{\eps}\Psi_\eps(z)),\\
        &(\Pi^\eps \<2'>)(z) \coloneqq \frac{1}{a_\eps \eps} F(\sqrt{\eps}\Psi_\eps(z)) - C^{(\eps)}_{\<2's>},
    \end{split}
    \end{equation}
    where $C^{(\eps)}_{\<2's>}$ is chosen to satisfy $\EE(\Pi^\eps \<2'>) =0$. 

    We then set $\hPi^\eps\tau = \Pi^\eps\tau$ for $\tau\in\{\;\<0'>\;,\<1'>\;,\<2'>\;\}$ and $(\hPi^\eps\tau\bar{\tau})(z) = (\hPi^\eps\tau)(z)\cdot(\hPi^\eps\bar{\tau})(z)$ except that
    \begin{equation} \label{e:rs_def_2}
    \begin{split}
        (\hPi^\eps\<2'1'>)(z) &= (\hPi^\eps\<2'0>)(z)\cdot (\hPi^\eps\<1'>)(z) - C^{(\eps)}_{\<2'1's>},\\
        (\hPi^\eps\<2'2'0>)(z) &= (\hPi^\eps\<2'0>)(z)^2 - C^{(\eps)}_{\<2'2'0s>},\\
        (\hPi^\eps\<2'2'0'>)(z) &= (\hPi^\eps\<2'2'0>)(z)\cdot (\hPi^\eps\<0'>)(z) - C^{(\eps)}_{\<2'2'0's>},\\
        (\hPi^\eps\<2'1'1'>)(z) &= (\hPi^\eps\<2'1'0>)(z)\cdot (\hPi^\eps\<1'>)(z) - C^{(\eps)}_{\<2'1'1's>},
    \end{split}
    \end{equation}
    where $C^{(\eps)}_\tau$ is chosen to satisfy $\EE\hPi^\eps\tau = 0$. Furthermore, we also set $(\hPi_z^\eps X^k)(z) = 0$ for all $k\in\NN$, where $X^k$ is the element of polynomial regularity structure. 

We denote the reconstruction operator associated with $\hPi^\eps$ by $\rR^\eps$. The convolution map $\kK^\eps_{\gamma}$ associated with $P_\eps$ on modelled distributions is given by \cite[Equation~(4.6)]{general_discrete}. For the remainder part $R_\eps$ of Green's function, there exists an operator $R^\eps_{\gamma}$ constructed as in \cite[Section~7.1]{Hai14a} such that $\rR^\eps R^\eps_{\gamma} = R_\eps$.

\begin{rmk}
One difference between the models here (built from general $F$ and Poisson $\xi_\eps$) as compared to \cite{KPZCLT, HQ, HX19} is that the ``first chaos" component of $\hPi^\eps \<2'>$ is not zero for fixed $\eps>0$, and that one subtracts the mean from $\hPi^\eps \<2'1'>$. We refer to Section~\ref{sec:square_process} for the corresponding remark on $\<2'>$, and Section~\ref{sec:211} for detailed calculation of the object $\<2'1'1'>$, which contains $\<2'1'>$ as a sub-process. 

The renormalisation constant $C_\eps$ in the macroscopic equation \eqref{e:macro} has the expression
\begin{equation} \label{e:renormalisation_constant}
    C_\eps = a_\eps C^{(\eps)}_{\<2's>} + 2 a_\eps^2 C^{(\eps)}_{\<2'1's>} + a_\eps^3 \big( C^{(\eps)}_{\<2'2'0s>} + C^{(\eps)}_{\<2'2'0's>} + 4 C^{(\eps)}_{\<2'1'1's>} \big)\;.
\end{equation}
Also note that the further subtraction of the constant $C^{(\eps)}_{\<2'1's>}$ from $\hPi^\eps\<2'1'>$ only changes the value of $C_\eps$ in the equation \eqref{e:macro}, but does not change the form of the equation. 
\end{rmk}

\begin{rmk} \label{rmk:A}
    In our setting, the representation $\hPi^\eps$ satisfies $(\hPi_z^\eps X^k)(z) = 0$, a condition not generally required in regularity structures. By \cite[Remark~4.14]{general_discrete}, this choice ensures that $\aA_\eps\coloneqq\rR^\eps \kK^\eps_{\gamma} - K_\eps\rR^\eps = 0$. Otherwise, there would be a small discrepancy between $\rR^\eps \kK^\eps_{\gamma}$ and $K_\eps\rR^\eps$. Consequently, in our case the term $\aA^\eps$ in \cite[Theorem~6.4]{general_discrete} can be omitted.
\end{rmk}

For convenience, we list all the symbols appearing in the regularity structures with their corresponding homogeneities. Here $\tau_\eps\coloneqq\hPi^\eps \tau$, and $\alpha-$ means $\alpha-\kappa$ for sufficiently small $\kappa$. 
\begin{equation}\label{e:symbols}
	\renewcommand\arraystretch{2}
	\begin{tabular}{ccccccccccc}
		\toprule
		$\tau_\eps$: & $\<0'>_{\eps}$ & $\<1'>_{\eps}$ & $\<2'>_{\eps}$   & $\<2'0>_{\eps}$ & $\<1'1'>_{\eps}$ & $\<2'0'>_{\eps}$ & $\<2'1'>_{\eps}$ & $\<2'2'0>_{\eps}$ & $\<2'2'0'>_{\eps}$ & $\<2'1'1'>_{\eps}$ \\
		\hline
		$|\tau|$: &$0-$ & $-\frac{1}{2}-$ &  $-1-$ &$0-$ & $0-$ & $0-$ & $-\frac{1}{2}-$ & $0-$ & $0-$ & $0-$ \\
		\bottomrule
	\end{tabular}
\end{equation}
It is well known that in order to prove the main convergence result in Theorem~\ref{thm:main} with the identified limit, one needs two ingredients: a continuity result for an abstract PDE in regularity structures, and convergence of the models to the limiting model describing the corresponding stochastic objects in the limiting equation. The abstract PDE result is given in Theorem~\ref{thm:abstract_pde} below. For the second ingredient, one needs to show the convergence of the models $\widehat{\Pi}^\eps$ given above to the limiting model which characterise the stochastic objects in the standard KPZ equation, which we denote by $\Pi^{\KPZ}$. The characterisation of the KPZ model $\Pi^{\KPZ}$ is now very well known, and can be found for example in \cite[Appendix~A]{HX19}. 

To show the convergence of $\widehat{\Pi}^\eps$ to $\Pi^{\KPZ}$, we compare $\widehat{\Pi}^{\eps}$ to a class of intermediate models studied in \cite{KPZCLT}, which we denote by $\hPi^{\HS(\eps)}$. The models $\hPi^{\HS(\eps)}$ are defined in the same way as \eqref{e:rs_def_1} and \eqref{e:rs_def_2} except that the integration kernel is $K_0$ and $K_0'$, the building block is $\<1>_\eps := P_0' * \xi_\eps$ (instead of $\Psi_\eps := P_\eps' * \xi_\eps$), and that the nonlinearity is the square function $|\cdot|^2$. We list in the table below the differences between $\hPi^\eps$ and $\hPi^{\HS(\eps)}$. 
\begin{equation} \label{e:models_comparison}
	\renewcommand\arraystretch{2}
	\begin{tabular}{cccc}
		\toprule
		\phantom{111} & \text{kernel(s)} &\, \text{building block} & \text{nonlinearity} \\
		\hline
		$\hPi^\eps$: &$K_\eps$ and $K_\eps'$ & $\Psi_\eps = P_\eps' * \xi_\eps$ &  $F$ \\
        \hline
        $\hPi^{\HS(\eps)}$: &$K_0$ and $K_0'$ & $\<1>_\eps = P_0' * \xi_\eps$ &  $|\cdot|^2$ \\
		\bottomrule
	\end{tabular}
\end{equation}
In particular, both models are built from the same Poisson noise $\xi_\eps$\footnote{The convergence theorem in \cite{KPZCLT} covers a much larger class of noises satisfying certain cumulants assumptions, including our Poisson noise $\xi_\eps$ as a primary example. In our situation where the nonlinearity $F$ is generically non-polynomial, we need to restrict ourselves at this moment to the Poisson noise in order to use Malliavin calculus.}, but that the free field as a building block are obtained by convoluting $\xi_\eps$ with different kernels, and that the nonlinearity in the constructions are different. 

It is shown in \cite[Theorem~6.5]{KPZCLT} that the models $\hPi^{\HS(\eps)}$ converge to the limiting model $\Pi^{\KPZ}$. We will show in Theorem~\ref{thm:convergence} below that the differences between $\hPi^\eps$ and $\hPi^{\HS(\eps)}$ vanish as $\eps \rightarrow 0$.

\subsection{Fixed point problem}

First we recall the $\eps$-dependent spaces of modelled distributions $\dD^{\gamma,\eta}_\eps$ given in \cite[Section~3.1]{general_discrete}. Only in this subsection, we make an abuse of notation $\eta$ to denote the degree of singularity when close to the hyperplane $\{(z_0, z_1) \in \RR \times \TT: z_0 = 0\}$. Following \cite[Definition~3.9]{general_discrete}, the space $\dD^{\gamma,\eta}_\eps$ consists of modelled distributions $U$ such that
    \begin{equation*}
        \|U\|_{\gamma,\eta;\eps}\coloneqq \|U\|_{\gamma,\eta;\geq\eps} + \|U\|_{\gamma,\eta;<\eps}<\infty, 
    \end{equation*}
    where $\|\cdot\|_{\gamma,\eta;\geq \eps}$ and $\|\cdot\|_{\gamma,\eta;<\eps}$ measure the large and small scale behaviours of the modelled distributions, and are respectively given by
    \begin{align*}
        \|U\|_{\gamma,\eta;\geq\eps} \coloneqq &\sup_{z,\alpha} |U(z)|_\alpha + \sup_{z_0 \geq\eps^2} \sup_{\alpha<\gamma} \frac{|U(z)|_{\alpha}}{|z_0|^{\frac{(\eta-\alpha)\wedge 0}{2}}}\\
        &+ \sup_{\substack{|z-z'|\leq\sqrt{|z_0|\wedge|z_0'|}\\|z-z'|\geq\eps}} \sup_{\alpha<\gamma} \frac{|U(z) - \Gamma_{zz'}U(z')|_{\alpha}}{|z-z'|^{\gamma-\alpha}(|z_0|\wedge|z_0'|)^{\frac{(\eta-\gamma)\wedge 0}{2}}}\;,\\
        \|U\|_{\gamma,\eta;<\eps} \coloneqq &\sup_{z_0<\eps^2} \sup_{\alpha>\eta} \frac{|U(z)|_{\alpha}}{\eps^{\eta-\alpha}}\\
        &+ \sup_{\substack{|z-z'|\leq\sqrt{|z_0|\wedge|z_0'|}\\|z-z'|<\eps}} \sup_{\alpha<\gamma} \frac{|U(z) - \Gamma_{zz'}U(z')|_{\alpha}}{|z-z'|^{\gamma-\alpha}\eps^{\eta-\gamma}}\;.
    \end{align*}
Here, the supremum of $z$ is taken over some compact domain of $\RR\times\TT$ and the norm above also depends on that domain. We drop its dependence in notation for simplicity. We can also compare $U^\eps\in\dD^{\gamma,\eta}_\eps$ and $U\in\dD^{\gamma,\eta}$ by
    \begin{equation*}
        \|U^\eps;U\|_{\gamma,\eta;\eps}\coloneqq \|U^\eps;U\|_{\gamma,\eta;\geq\eps} + \|U^\eps\|_{\gamma,\eta;<\eps} + \|U\|_{\gamma,\eta;<\eps},
    \end{equation*}
    where $\|\cdot;\cdot\|_{\gamma,\eta;\geq\eps}$ is given by \cite[Equation~(3.27)]{general_discrete}. Its form is very similar to $\|\cdot\|_{\gamma,\eta;\geq \eps}$ above, so we do not repeat the long formula here.  

    \begin{rmk}
         We utilise the framework in \cite{general_discrete} concerning scales larger than $\eps$. For scales smaller than $\eps$, we extract small positive powers of $\eps$ to make it vanish as $\eps\rightarrow0$ according to the definition of $\|\cdot\|_{\gamma,\eta;<\eps}$.
    \end{rmk}
    
    Following \cite[Section~3.3]{HX19}, we denote the collection of models $\Pi^\eps$ such that $\$ \Pi^\eps\$_\eps<\infty$ by $\sM_\eps$ and compare $\Pi^\eps\in\sM_\eps$ and $\Pi\in\sM$ using $\$ \Pi^\eps;\Pi\$_{\eps;0}$. We now present the fixed point theorem.

    \begin{thm} \label{thm:abstract_pde}
    Let $\gamma\in(\frac{3}{2},\frac{5}{3})$, $\bar{\gamma} = \gamma-\frac{3}{2}-\kappa$ and $\eta\in(\frac{1}{2}-\frac{1}{M+4},\frac{1}{2})$, where $\kappa>0$ is sufficiently small and $M$ is given in Assumption~\ref{as:F}. Let $\{\psi_{\eps}\}_{\eps\in(0,1)}$ be a family of space-time functions such that
    \begin{equation*}
        \sup_{\eps\in(0,1)}\;\sup_{x\in[0,1]\times\TT} \eps^{\frac{1}{2}+\kappa} |\psi_{\eps}(x)| < +\infty.
    \end{equation*}	
    Consider the fixed point problem
    \begin{equation} \label{e:fixed_point}
        \begin{split}
        U_\eps = &\widehat{P_\eps}u_0^\eps + (\kK_{\bar{\gamma}}^\eps + R_{\gamma}^\eps \rR^\eps) \boldsymbol{1}^{+} \Big(a_\eps( \<1'>+\dD U_\eps)^2\\
        &+ a_\eps \<0'> \; (\dD U_\eps)^2 + \eps^{-1} G(\sqrt{\eps}\psi_\eps, \sqrt{\eps}\rR\dD        U_\eps)\cdot\textcolor{symbols}{\boldsymbol{1}} \Big)\;,
        \end{split}
    \end{equation}
    where $\widehat{P_\eps}$ is the harmonic extension of a classical function into $\dD^{\gamma,\eta}_\eps$ space, $\dD$ is abstract differentiation, and
        \begin{equation*}
            G(x,y): = F(x+y) -F(x) - F'(x)y - \frac{1}{2}F''(x) y^2.
        \end{equation*}
    For every $u_0^\eps\in\cC^{\gamma,\eta}_\eps$ and $\Pi^\eps\in\sM_\eps$, there exists $T^\eps>0$ such that \eqref{e:fixed_point} has a unique solution $U^\eps\in\dD^{\gamma,\eta}_\eps(\Pi^\eps)$ on $[0,T^\eps]$. Moreover, if $\$ \Pi^\eps \$_\eps$ and $\|u_0^\eps\|_{\gamma,\eta;\eps}$ are uniformly bounded in $\eps\in(0,1)$, then so is $\|U^\eps\|_{\gamma,\eta;\eps}$.

    Furthermore, suppose that $u_0^\eps$ converges to $u_0\in\cC^{\eta}$ in the sense of \cite[Equation~(3.6)]{HX19}, $a_\eps\rightarrow a$ and $\$ \Pi^\eps;\Pi \$_{\eps;0} \rightarrow0$ as $\eps\rightarrow0$. Let $U\in\dD^{\gamma,\eta}(\Pi)$ be the unique solution to the fixed point problem
        \begin{equation} \label{e:fixed_point_0}
            U= (\kK_{\bar{\gamma}} + R_{\gamma} \rR) \boldsymbol{1}^{+} \big(a( \<1'>+\dD U)^2 + a\<0'> \; (\dD U)^2 \big) + \widehat{P}u_0
        \end{equation}
    on $[0,T]$. Then for every sufficiently small $\eps>0$, $U^\eps$ exists up to the same time $T$ and $\lim_{\eps\rightarrow0}\|U^\eps;U\|_{\gamma,\eta;\eps}=0$.

    Finally, the reconstructed solution $u_\eps := \rR^\eps U^\eps$ with the reconstruction operator $\rR^\eps$ associated to the model $\hPi^\eps$ satisfies the macroscopic equation \eqref{e:macro} with renormalisation constant $C_\eps$ given in \eqref{e:renormalisation_constant}. 
	\end{thm}
    \begin{proof}
        For the part $\|U^\eps;U\|_{\gamma,\eta;\geq\eps}$, the result is derived from \cite[Theorem~6.4]{general_discrete}, so we only need to verify the assumptions of this theorem.
        
        We begin by verifying \cite[Assumption~6.1]{general_discrete}. \cite[Equation~(6.3), (6.4)]{general_discrete} are direct consequences of \cite[Lemma~7.3]{Hai14a}. Using the same lemma as well as the uniform smoothness of $R_\eps$, we obtain \cite[Equation~(6.5)]{general_discrete}. Moreover, it should be noted that the constant $C(\eps)$ in \cite[Equation~(6.5)]{general_discrete} vanishes as $\eps\rightarrow0$. 
        
        Next, we verify the assumptions of \cite[Lemma~6.2]{general_discrete}. Recall from Corollary~\ref{cor:kernel} that our kernel $K_\eps$ satisfies all the assumptions in \cite[Section~4]{general_discrete}, and hence \cite[Equation~(6.6)]{general_discrete} holds. For \cite[Equation~(6.7)]{general_discrete}, we need some modifications. The definition of $\|\cdot\|_{\gamma,\eta;<\eps}$ implies the modified version of \cite[Equation~(6.7)]{general_discrete} where we replace the coefficient $T^{\frac{\kappa}{3}}$ by $\eps^{\frac{\kappa}{3}}$ on its right hand. As a consequence, the conclusions of \cite[Lemma~6.2]{general_discrete} hold if we replace the coefficients $T^{\frac{\kappa}{3}}$ by $(T+\eps)^{\frac{\kappa}{3}}$. Note that the changing of the coefficients does not impact the proof of the fixed point problem.

        \cite[Assumption~6.3]{general_discrete} is automatically satisfied as mentioned in Remark~\ref{rmk:A}. Combining the proof of \cite[Theorem~6.4]{general_discrete} and \cite[Theorem~3.7]{HX19}, we finally obtain $\lim_{\eps\rightarrow0}\|U^\eps;U\|_{\gamma,\eta;\geq\eps}=0$.
        
        For the part with scales small than $\eps$, by the definition of $\|\cdot\|_{\gamma,\eta;<\eps}$ we have
        \begin{equation*}
            \|U^\eps\|_{\gamma,\eta;<\eps} + \|U\|_{\gamma,\eta;<\eps}  \lesssim \eps^{\kappa} \big( \|U^\eps\|_{\gamma,\eta+\kappa;\eps} + \|U^\eps\|_{\gamma+\kappa,\eta;\eps} + \|U\|_{\gamma,\eta+\kappa} + \|U\|_{\gamma,\eta+\kappa} \big)
        \end{equation*}
        since $\|\cdot\|_{\gamma,\eta;<\eps} \leq \|\cdot\|_{\gamma,\eta;\eps}$. Therefore, the convergence is a consequence of the uniform boundedness of $U^\eps$.

        Finally, identification of the equation for $u_\eps := \rR^\eps U^\eps$ follows from exactly the same argument as \cite[Theorem~3.8]{HX19}, except that one further adds $2 a_\eps^2 C^{(\eps)}_{\<2'1's>}$ into the renormalisation constant $C_\eps$, which is also straightforward from the expansion of the right hand side of the abstract equation \eqref{e:fixed_point_0}. 
    \end{proof}

\section{The spectral gap inequality for Poisson point process}
\label{sec:spectral_gap}

In this section, we provide some preliminary knowledge for the Malliavin calculus of Poisson point process and give an $L^p$-version spectral gap inequality. Most of the materials in this section are contained in the text \cite{LecturesonPP}. 

\subsection{Preliminaries}

Let $(\Omega, \mathcal{F}, \mathbf{\PP})$ be a probability space and $(\UU, \uU, \mu)$ be a $\sigma$-finite measure space. We define $\mathbf{N}_\sigma(\UU)$ as the set of $\sigma$-finite measures on $\UU$ with values in $\{0, 1, 2, \dots, \infty\}$. The $\sigma$-algebra $\mathcal{N}_\sigma(\UU)$ on $\mathbf{N}_\sigma(\UU)$ is the smallest $\sigma$-algebra such that for every $W\in\uU$, the mapping $\mathbf{N}_\sigma(\UU) \ni \xi \mapsto \xi(W)$ is measurable. Let $\eta$ denote the Poisson point process on $\UU$ with intensity measure $\mu$.

\begin{rmk}
    In our context of the KPZ equation, we will use $\UU = \RR \times \TT_\eps$, $\mu$ be the Lebesgue measure, and $\eta = \eta^{(\eps)}$ the Poisson point process on $\RR \times \TT_\eps$ with unit intensity, as introduced in Section~\ref{sec:intro_results}. We still follow the more abstract formulation in \cite{LecturesonPP} since it does not cost additional efforts here. 
\end{rmk}

We define
    \begin{equation*}
        L^0_\eta = \Big\{ f(\eta) \;|\; f : \mathbf{N}_\sigma(\UU) \rightarrow \mathbb{R}\text{ is a } \mathcal{N}_\sigma(\UU)\text{-measurable function}\; \Big\},
    \end{equation*}
and we denote by $L^p_\eta$ the set of random variables in $L^0_\eta$ with finite $p$-th moment for $p > 0$. In words, $L_\eta^p$ is the subspace of $L^p(\Omega)$ where all randomness are from $\eta$. 

    For $\mathfrak F=f(\eta) \in L^0_\eta$ and $u \in \UU$, the difference operator $D_u$ is defined by
        \[(D_u  \mathfrak F) \coloneqq  f(\eta + \delta_u) -  f(\eta),\]
    where $\delta_u$ is the Dirac mass at $u$. We recursively extend this definition to higher orders, setting for $n \in \NN$ and $\vec{u} = (u_1, \dots, u_n) \in \UU^n$ the $n$-th order derivative $D^n$ at ${\vec{u}} \in \UU^n$ by
    \begin{equation*}
        D^n_{\vec{u}} \, \mathfrak F \coloneqq D_{u_n} D_{u_{n-1}} \dots D_{u_1} \mathfrak F\;.
    \end{equation*}
    By definition of $D_u$, we have the product formula
        \begin{equation} \label{e:prod_deri}
                D_u (\mathfrak F \mathfrak G) = (D_u \mathfrak F) \, \mathfrak G + \mathfrak F \, (D_u \mathfrak G) + (D_u \mathfrak F) \, (D_u\mathfrak G)\;.
        \end{equation}
    For a measure $\chi \in \N_\sigma (\UU)$, it can be expressed as a sum of delta masses (not necessarily at distinct points) in $\UU$ as
    \begin{equation*}
        \chi = \sum_{j} \delta_{u_j}\;.
    \end{equation*}
    As in \cite[Section~4.2]{LecturesonPP}, for $k \in \NN$ we define its $k$-th factorial measure $\chi^{\diamond k}$ as a sum of delta masses in $\UU^k$ by
    \begin{equation*}
        \chi^{\diamond k} := \sum \delta_{(u_{j_1}, \dots, u_{j_k})}\;,
    \end{equation*}
    where the sum are taken over $j_1, \dots, j_k$ such that $u_{j_i} \neq u_{j_{i'}}$ for any two distinct indices $i$ and $i'$. In short, $\chi^{\diamond k}$ is the $k$-th direct product of $\chi$ with itself with repetitions of points removed. 
    
    For $n \in \NN$ and $p \geq 1$, let $L_s^p (\UU^n, \uU^n, \mu^n)$ be the space of $L^p(\UU^n, \uU^n, \mu^n)$ functions that are symmetric under permutations of its $n$ variables. We write $L^{p,n}_{s}$ for $L^p_s(\UU^n, \uU^n, \mu^n)$ for simplicity. The $n$-th Wiener-It\^o multiple integral is the map
    \begin{equation*}
        I_{n}: L_s^{1,n} \cap L_s^{2,n} \rightarrow L^2(\Omega, \PP)
    \end{equation*}
    given by
    \begin{equation} \label{e:In}
        I_n(g) = \sum_{k=0}^{n} (-1)^{n-k} \binom{n}{k} \int_{\UU^n} g \, d\eta^{\diamond k} \, d\mu^{n-k},
    \end{equation}
    where we use the convention $d\eta^{\diamond 0} \, d\mu^{n} \coloneqq d\mu^{n}$ and $d\eta^{\diamond n} \, d\mu^{0}\coloneqq d\eta^{\diamond n}$. For every $k\in\{0,1,\dots,n\}$, by Fubini theorem, we have
        \begin{equation*}
          \EE\int_{\UU^n} |g|d\eta^{\diamond k}d\mu^{n-k}=\int_{\UU^n} |g|d\mu^n<\infty.
        \end{equation*}
    Then we have $I_n(g)\in L^{1}_{\eta}$ and $\EE I_n(g)=0$. For $m,n \in \NN$, $f \in L_{s}^{1,m} \cap L_{s}^{2,m}$ and $g \in L_{s}^{1,n} \cap L_{s}^{2,n}$, we have (see \cite[Corollary~12.8]{LecturesonPP})
    \begin{equation}\label{eqn:ISO}
        \EE \big( I_n(f) I_m(g) \big) = \boldsymbol{1}_{n = m} \, n! \int_{\UU^n} fg \, d\mu^n\;.
    \end{equation}
    This implies that $\frac{1}{\sqrt{n!}} I_n$ extends uniquely to a map from $L_{s}^{2,n}$ into $L^2(\Omega)$ with $\EE I_{n} (f) = 0$ for $n \geq 1$ and the property \eqref{eqn:ISO}. For $n \in \NN$ and $g \in L_{s}^{2,n}$, we call the extended random variable $I_n(g)$ the multiple Wiener-It\^o integral of order $n$ for $g$. Note that the extension to $L_s^{2,n}$ is valid for $I_n$ in the expression \eqref{e:In} as a whole sum, while single terms in that sum may be infinite for $g \in L_{s}^{2,n} \setminus L_{s}^{1,n}$. 
    
    The following lemma provides the characteristic function of $I_1(g)$ for $g\in L^{2}(\UU,\uU,\mu)$. The proof of this lemma is similar to \cite[Theorem~3.9]{LecturesonPP}.
    
    \begin{lem}\label{chf}
         For $t \in \RR$ and $g \in L^{2}(\UU,\uU,\mu)$\;, we have 
            \begin{equation} \label{e:chf}
                \mathbb Ee^{i t I_1(g)}=\exp\left(\int_{\UU}(e^{i t g}-itg-1) \, d\mu \right).
            \end{equation}
    \end{lem}
    \begin{proof}
        We first verify the identity \eqref{e:chf} constant multiples of indicator functions of the form $g = c \mathbf{1}_A$ with $\mu(A)<\infty$. For such a $g$, we have
        \begin{equation*}
            I_1(g) = c \big( \eta(A) - \mu(A) \big)\;,
        \end{equation*}
        where $\eta(A)$ is a Poisson random variable with mean $\mu(A)$. In this case, the identity \eqref{e:chf} follows directly from the characteristic function of the Poisson random variable. One then extends it to simple functions by independence of the Poisson field in disjoint domains and then all $L^2(\UU, \uU, \mu)$ functions by density. 
    \end{proof}

The Wiener-It\^{o} orthogonal chaos expansion theorem, as stated in \cite[Theorem~18.10]{LecturesonPP}, is as follows.
    
\begin{prop}\label{prop:WienerItoChaos}
    For $ \mathfrak F \in L^2_\eta$\,, we have the expansion
        \begin{equation}\label{eq:WienerItoChaos}
	        \mathfrak F = \sum_{n=0}^\infty I_n(f_n)\;,
        \end{equation}
    where
    \begin{equation*}
        f_n (\vec{u}) = \frac{1}{n!} \EE D^n_{\vec{u}} \mathfrak{F} \in L^{2}_s(\UU^n,\uU^n,\mu^n)\;,
    \end{equation*}
    and the series converges in $L^2(\Omega, \PP)$\,. Moreover, we have the formula
        \begin{equation} \label{e:chaos_expansion}
	       \| \mathfrak F\|^2_{L^2_\omega} = \sum_{n=0}^\infty n! \left\|f_n\right\|^2_{L^2(\UU^n)}\;.
        \end{equation}
\end{prop}

The following proposition, as stated in \cite[Theorem~3.3]{chaosexpansion}, provides the Wiener-It\^{o} orthogonal chaos expansion of $D_u \mathfrak F$. With this property, we find that the difference operator is actually the Malliavin derivative of Poisson point process.

\begin{prop}\label{DxFWienerItoChaos}
    Let $ \mathfrak F\in L^2_\eta$ be given by \eqref{eq:WienerItoChaos}. Suppose
        \begin{equation*}
            \sum_{n=1}^{\infty} n \cdot n!\left\|f_n\right\|^2_{L^2(\UU^n)} < +\infty\;.
        \end{equation*}
    Then we have 
        \begin{equation*}
            D_u \mathfrak F = \sum_{n=1}^\infty n I_{n-1} \big( f_n(u,\cdot) \big)\;.
        \end{equation*}
\end{prop}

We now present a formula for expectation of product of random variables of the form $I_{n_i}(f_i)$. This is needed in the computation of a three-point correlation in Section~\ref{sec:211_correlation} below. We follow the notations in \cite[Section~12]{LecturesonPP}. 

Define $a\coloneqq n_1+\dots+n_{\ell}$ and $\vec{n}\coloneqq(n_1,\cdots,n_{\ell})$. We define $\Pi_{a}$ as the set of all partitions of $\{1,2,\dots,{a}\}$. For $\sigma \in \Pi_{a}$, we denote $|\sigma|$ as the number of blocks in $\sigma$. Define 
        \begin{equation*}
            J_i\coloneqq\left\{j \in \mathbb{N}: n_1+\cdots+n_{i-1}<j \leq n_1+\cdots+n_i\right\}, \quad i=1, \ldots, \ell .
        \end{equation*}
    Let $\pi\coloneqq\left\{J_i: 1 \leq i \leq \ell\right\}$ and let $\Pi(\vec{n}) \subset \Pi_{a}$ denote the set of all $\sigma \in \Pi_{a}$ with $\left|J \cap J^{\prime}\right| \leq 1$ for all $J \in \sigma$ and for all $J^{\prime} \in \pi$. Let $\Pi_{\geq 2}(\vec{n})$ denote the set of all $\sigma \in \Pi(\vec{n})$ with $|J| \geq 2$ for all $J \in \sigma$.
 
    For $\sigma \in \Pi_{a}$, we can write $\sigma$ as $\sigma=\{J^{(1)},\dots,J^{(|\sigma|)}\}$, where $J^{(i)}\subset\{1,2,\dots,a\}$ and $\inf J^{(1)}<\cdots<\inf J^{(|\sigma|)}$. For every function $f: \UU^a \rightarrow \mathbb{R}$ and $\sigma \in \Pi_a$, we define $f_\sigma: \UU^{|\sigma|} \rightarrow \mathbb{R}$ by 
    \begin{equation*}
        f_\sigma (x_1, \dots, x_{|\sigma|}) = f(y_1, \dots, y_a) 
    \end{equation*}
    with $y_k = x_i$ if and only if $k \in J^{(i)}$.
    
    The following lemma (\cite[Corollary~12.8]{LecturesonPP}) gives a formula for the expectation of a product of multiple Wiener-It\^o integrals.

\begin{lem}\label{lem:wick}
    Let $f_i\in L_s^1(\UU^{n_i},\uU^{n_i},\mu^{n_i}),i = 1,\dots,\ell$, where $\ell,n_1,\dots,n_{\ell}\in\NN$. Then we have
        \begin{equation*}
            \mathbb{E}\left[\prod_{i=1}^{\ell} I_{n_i}\left(f_i\right)\right]=\sum_{\sigma \in \Pi_{\geq 2}(\vec{n})} \int_{\UU^{|\sigma|}}\left( \otimes_{i=1}^\ell f_{i}\right)_\sigma d\mu^{|\sigma|}.
        \end{equation*}
\end{lem}

\subsection{The spectral gap inequality for Poisson point process}

The primary tool in our calculation is the $L^p$ spectral gap inequality for Poisson point process in Proposition~\ref{prop:spec_gap} below. We first recall an $L^2$-version of it, as presented in \cite[Corollary~18.8]{LecturesonPP}.

\begin{prop}\label{prop:spec_gapL2}
For every $ \mathfrak F \in L_\eta^1$\;, we have
\begin{equation} \label{e:SGI_L2}
    \EE  \mathfrak F^2 \leq (\EE  \mathfrak F)^2 + \EE \int_{\UU} \left(D_u  \mathfrak F\right)^2 \, d\mu.
\end{equation}
\end{prop}

In the sequel, we will use $c_p$ to represent various positive constants that depend only on $p$, but their actual values may differ from line to line. We need the following lemma before proving the $L^p$-version spectral gap inequality. 

\begin{lem}\label{lemspec_gap}
For every $p\geq2$ and $ \mathfrak F \in L_{\eta}^p$\;, we have
\begin{equation*}
    \E \int_{\UU} \left(D_u| \mathfrak F|^{\frac{p}{2}}\right)^2 \, d \mu \le c_p \bigg(\EE \int_{\UU} \left|D_u  \mathfrak F\right|^p \, d \mu + \EE \Big(\int _{\UU}\left(D_u  \mathfrak F\right)^2 \, d\mu \Big)^{\frac{p}{2}}\bigg) + \frac12 \EE | \mathfrak F|^p,
\end{equation*}
where $c_p$ is independent of $ \mathfrak F$ and the measure space $(\UU,\uU,\mu)$.
\end{lem}

\begin{proof}
Combining the definition of $D_u$ and the inequality 
\begin{equation*}
    \left||b|^{\frac p2}-|c|^{\frac p2}\right| \le c_p \left(|b|^{\frac p2-1} + |c|^{\frac p2-1}\right)|b-c| \leq c_p \left(|b-c|^{\frac{p}{2}} + |c|^{\frac{p}{2}-1}|b-c|\right),
\end{equation*}
we have
\begin{equation*}
    \EE \int_{\UU} \left(D_u| \mathfrak F|^{\frac{p}{2}}\right)^2 \, d \mu
       \le c_p \bigg(\EE \int_{\UU}\left|D_u  \mathfrak F\right|^{p} \, d \mu+\EE\Big(\Big(\int_{\UU}\left(D_u  \mathfrak F\right)^2 \, d\mu \Big)| \mathfrak F|^{p-2}\Big)\bigg).
\end{equation*}
By Young inequality, we get
\begin{equation*}
    \EE\left(\left(\int_{\UU}\left|D_u  \mathfrak F\right|^2 \, d\mu \right)| \mathfrak F|^{p-2}\right)
       \le c_p\EE\left(\int_{\UU}\left(D_u  \mathfrak F\right)^2d\mu \right)^{\frac p2}+\frac12\EE| \mathfrak F|^p,
\end{equation*}
which completes the proof.
\end{proof}

Now we are prepared to establish the $L^p$-version spectral gap inequality of the Poisson point process.
\begin{prop}\label{prop:spec_gap}
    For every $p \ge 2$ and $ \mathfrak F \in L_\eta^1$\;, we have
    \begin{equation*}
        \| \mathfrak F\|_{L^p_{\omega}} \lesssim_p |\EE  \mathfrak F| + \|D_u  \mathfrak F\|_{(L^p_{\omega}L^2_u) \cap (L^p_\omega L^p_u)},
    \end{equation*}
    where the proportionality constant is independent of $ \mathfrak F$ and the measure space $(\UU,\uU,\mu)$\;. As a consequence of Minkowski inequality, we have
    \begin{equation} \label{e:sg}
        \| \mathfrak F\|_{L^p_{\omega}} \lesssim_p |\EE  \mathfrak F| + \|D_u  \mathfrak F\|_{(L_u^2\cap L_u^p)L^p_\omega}
    \end{equation}
    with the same proportionality constant.
\end{prop}
\begin{proof}
    We assume $ \mathfrak F \in L_\eta^p$ first. According to Proposition \ref{prop:spec_gapL2}, we can derive that
    \begin{equation*}
        \EE| \mathfrak F|^p \le \big(\EE| \mathfrak F|^{p/2}\big)^2 + \EE \int_{\UU} \left(D_u| \mathfrak F|^{p/2}\right)^2 \, d \mu.
    \end{equation*}
    For the first term on the right hand side, we claim that
    \begin{equation*}
        \big(\EE| \mathfrak F|^{p/2}\big)^2 \le \frac{1}{4} \| \mathfrak F\|^p_{L^p_{\omega}} + c_p\| \mathfrak F\|^p_{L^2_{\omega}}.
    \end{equation*}
    This claim directly holds for $2 < p \le 4$, since $\big(\EE| \mathfrak F|^{p/2}\big)^2 \le \| \mathfrak F\|_{L^2_{\omega}}^p$. For $p > 4$, by H\"older and Young inequalities, we have
    \begin{equation*}
        \big(\EE| \mathfrak F|^{p/2}\big)^2 \le \left(\| \mathfrak F\|^{2/(p-2)}_{L^2_{\omega}}\| \mathfrak F\|^{(p-4)/(p-2)}_{L^p_{\omega}}\right)^p \le \frac{1}{4} \| \mathfrak F\|^p_{L^p_{\omega}} + c_p\| \mathfrak F\|^p_{L^2_{\omega}},
    \end{equation*} 
    which finishes the proof of the claim. By Proposition~\ref{prop:spec_gapL2}, we get
    \begin{equation*}
        \| \mathfrak F\|_{L^2_{\omega}}^p \le \Big( |\EE  \mathfrak F|^2 + \EE \int_{\UU}\left(D_u  \mathfrak F\right)^2 d \mu\Big)^{p/2} \le  c_p|\EE  \mathfrak F|^p + c_p\Big( \EE \int_{\UU}\left(D_u  \mathfrak F\right)^2 \, d \mu\Big)^{p/2}.
    \end{equation*}
    Combining these bounds with Lemma~\ref{lemspec_gap}, we obtain
    \begin{equation*}
    \begin{split}
        \EE| \mathfrak F|^p \le& c_p\bigg(|\EE  \mathfrak F|^p + \Big( \EE \int_{\UU}\left(D_u  \mathfrak F\right)^2 \, d \mu\Big)^{p/2} + \EE \int_{\UU}\left|D_u  \mathfrak F\right|^p \, d \mu\\
        &+ \EE\Big(\int_{\UU}\left(D_u  \mathfrak F\right)^2 \, d\mu \Big)^{p/2}\bigg) + \frac{3}{4} \EE| \mathfrak F|^p.
    \end{split}
    \end{equation*}
    Therefore, the desired result for $ \mathfrak F \in L_\eta^p$ follows from
    \begin{equation*}
        \Big( \EE \int_{\UU}\left(D_u  \mathfrak F\right)^2 d \mu\Big)^{p/2} \leq \EE \Big(  \int_{\UU}\left(D_u  \mathfrak F\right)^2 \, d \mu\Big)^{p/2}.
    \end{equation*}
    For $ \mathfrak F\in L_\eta^1$, we define $\mathfrak F_n=(\mathfrak F\wedge n) \vee(-n)\in L_\eta^p$. The above argument yields that
    \begin{equation*}
        \| \mathfrak F_n\|_{L^p_{\omega}} \lesssim_p |\EE   \mathfrak F_n| + \|D_u  \mathfrak F\|_{(L^p_{\omega}L^2_u) \cap (L^p_\omega L^p_u)},
    \end{equation*}
    where we use $|D_u\mathfrak F_n|\le|D_u  \mathfrak F|$. Then we apply monotone and dominated convergence theorems to $\| \mathfrak F_n\|_{L^p_{\omega}}$ and $|\EE \mathfrak F_n|$ respectively to conclude our proof. 
\end{proof}

In our specific case, we only apply the above proposition to $\eta^{(\eps)}$, the Poisson point process on $\RR \times \TT_\eps$ with uniform intensity, as defined in the introduction. 

\begin{rmk}
    If the Malliavin differentiation $D$ were a continuous operation, then it is well known that one can extend \eqref{e:SGI_L2} to \eqref{e:sg} with $L_u^2 L_\omega^p$-norm alone of $D_u \mathfrak{F}$ on the right hand side (see for example \cite[Proposition~5.1]{variational_SPDE}). In our case with the difference operator, we have the additional $L_u^p L_\omega^p$-norm, and this is the main difference compared to the classical spectral gap inequality. 
\end{rmk}

\section{Preliminary estimates}

In this section, we will provide several useful estimates, which will be used repeatedly in the proof of the convergence of the stochastic objects in Section~\ref{sec:convergence}. An essential role is played by the kernel $P_{\eps_1,\eps_2}^\theta$ to be introduced below (in particular $\eps_1 = \eps_2$) as it appears in the Malliavin derivative of the free field $\Psi_\eps$.

\subsection{Some convolution bounds with singular kernels}

This section provides some estimates on the kernels. The following two lemmas describe the behaviour of the convolution of the kernels $|x|^{-\alpha}$ and $(|x|+\eps)^{-\beta}$.

\begin{lem}
\label{lem:convolution_singularity}
    For every $\alpha,\beta\in(0,3)$ such that $\alpha+\beta>3$\;, we have
    \begin{equation*}
        \int_{\RR\times\TT} \frac{1}{|x-y|^{\alpha}(|y-z|+\eps)^{\beta}} dy \lesssim \frac{1}{(|x-z|+\eps)^{\alpha+\beta-3}},
    \end{equation*}
    where the proportionality constant is independent of $x,z\in\RR\times\TT$ and $\eps\in(0,1)$.
\end{lem}
\begin{proof}
    For the case $|x-z|\geq\eps$, the desired result follows from
    \begin{equation*}
        \int_{\RR\times\TT} \frac{1}{|x-y|^{\alpha}|y-z|^{\beta}} dy \lesssim \frac{1}{|x-z|^{\alpha+\beta-3}}.
    \end{equation*}
    For the case $|x-z|<\eps$, we split the domain into $\{y: |y-z|\leq2\eps\}$ and $\{y: |y-z|>2\eps\}$. In the first domain, we have
    \begin{equation*}
        \int\limits_{|y-z|\leq2\eps} \frac{1}{|x-y|^{\alpha}(|y-z|+\eps)^{\beta}} dy \lesssim \int\limits_{|y-x|\leq3\eps} \frac{1}{|x-y|^\alpha \eps^\beta} dy \lesssim \frac{1}{(|x-z|+\eps)^{\alpha+\beta-3}}.
    \end{equation*}
    In the second domain, we have
    \begin{equation*}
        \int\limits_{|y-z|>2\eps} \frac{1}{|x-y|^{\alpha}(|y-z|+\eps)^{\beta}} dy \lesssim \int\limits_{|y-x|>\eps} \frac{1}{|x-y|^{\alpha+\beta}} dy \lesssim \frac{1}{(|x-z|+\eps)^{\alpha+\beta-3}}.
    \end{equation*}
    This completes the proof.
\end{proof}

\begin{lem}
\label{lem:convolution_singularity'}
    For every $\alpha\in(0,3)$ and $\beta\in (3,+\infty)$\;, we have
    \begin{equation*}
        \int_{\RR\times\TT} \frac{\eps^{\beta-3}}{|x-y|^{\alpha}(|y-z|+\eps)^{\beta}} dy \lesssim \frac{1}{(|x-z|+\eps)^{\alpha}},
    \end{equation*}
    where the proportionality constant is independent of $x,z\in\RR\times\TT$ and $\eps\in(0,1)$\;.
\end{lem}
\begin{proof}
    For the case $|x-z|\geq\eps$, we split the domain into $\{y: |y-z|\leq \frac{|x-z|}{2}\}$, $\{y: |y-z| \in\big(\frac{|x-z|}{2},2|x-z|\big]\}$ and $\{y: |y-z|>2|x-z|\}$. In the first domain, we have
    \begin{equation*}
        \begin{aligned}
            \int\limits_{|y-z|\leq\frac{|x-z|}{2}} \frac{\eps^{\beta-3}}{|x-y|^{\alpha}(|y-z|+\eps)^{\beta}} dy &\lesssim \frac{1}{|x-z|^\alpha} \int\limits_{|y-z|\leq\frac{|x-z|}{2}} \frac{\eps^{\beta-3}}{(|y-z|+\eps)^{\beta}} dy \\
            &\lesssim \frac{1}{(|x-z|+\eps)^{\alpha}}.
        \end{aligned}
    \end{equation*}
    In the second domain, we have
    \begin{equation*}
        \begin{aligned}
            \int\limits_{\frac{|x-z|}{2}<|y-z|\leq2|x-z|} \frac{\eps^{\beta-3}}{|x-y|^{\alpha}(|y-z|+\eps)^{\beta}} dy &\lesssim \frac{1}{(|x-z|+\eps)^\beta} \int\limits_{|y-x|\leq3|x-z|} \frac{\eps^{\beta-3}}{|x-y|^{\alpha}} dy \\
            &\lesssim \frac{1}{(|x-z|+\eps)^{\alpha}}.
        \end{aligned}
    \end{equation*}
    In the third domain, we have
    \begin{equation*}
        \int\limits_{|y-z|>2|x-z|} \frac{\eps^{\beta-3}}{|x-y|^{\alpha}(|y-z|+\eps)^{\beta}} dy \lesssim \int\limits_{|y-x|>|x-z|} \frac{\eps^{\beta-3}}{|x-y|^{\alpha+\beta}} dy \lesssim \frac{1}{(|x-z|+\eps)^{\alpha}}.
    \end{equation*}
    
    For the case $|x-z|<\eps$, we split the domain into $\{y: |y-z|\leq2\eps\}$ and $\{y: |y-z|>2\eps\}$. In the first domain, we have
    \begin{equation*}
        \int\limits_{|y-z|\leq2\eps} \frac{\eps^{\beta-3}}{|x-y|^{\alpha}(|y-z|+\eps)^{\beta}} dy \lesssim \int\limits_{|y-x|\leq3\eps} \frac{\eps^{-3}}{|x-y|^\alpha} dy \lesssim \frac{1}{(|x-z|+\eps)^{\alpha}}.
    \end{equation*}
    In the second domain, we have
    \begin{equation*}
        \int\limits_{|y-z|>2\eps} \frac{\eps^{\beta-3}}{|x-y|^{\alpha}(|y-z|+\eps)^{\beta}} dy \lesssim \int\limits_{|y-x|>\eps} \frac{\eps^{\beta-3}}{|x-y|^{\alpha+\beta}} dy \lesssim \frac{1}{(|x-z|+\eps)^{\alpha}}.
    \end{equation*}
    This completes the proof.
\end{proof}

\begin{rmk} \label{rmk:convolution_singularity'}
    In general, for $\alpha, \beta > 3$ with $\alpha \leq \beta$, we have the inequality
    \begin{equation*}
        \int_{\RR \times \TT} \frac{\eps^{\beta-3}}{(|x-y|+\eps)^\alpha (|y-z|+\eps)^\beta } \, dy \lesssim \frac{1}{(|x-z|+\eps)^\alpha}\;.
    \end{equation*}
    The proof is essentially the same as in Lemma~\ref{lem:convolution_singularity'}. 
\end{rmk}

\begin{lem}\label{lem:kernel_convolution_5}
    For every $\delta\in(0,1)$, we have       
        \begin{equation*}
            \int_{\RR\times\TT}\frac{|P'_{\eps}(y-r) -P'_{\eps}(z-r) |}{|y-z|^2\big(|y-z|+\eps \big)^2} \, dz \lesssim \frac{\eps^{-1+\delta}}{|y-r|^{2+\delta}},
        \end{equation*}
     where the proportionality constant is independent of $y,r\in\RR\times\TT$ and $\eps\in(0,1)$.
\end{lem}
\begin{proof}
    We partition the integration domain into four regions: $\{ |z-r|\leq\frac{|y-z|}{2} \}$, $\{ |y-r|\leq\frac{|y-z|}{2} \}$, $\{|z-r|>\frac{|y-z|}{2},\;|y-r|>\frac{|y-z|}{2},\;|z-r|<\eps \}$ and $\{|z-r|>\frac{|y-z|}{2},\;|y-r|>\frac{|y-z|}{2},\;|z-r|\geq\eps \}$. We use different bounds of $|P'_\eps(y-r) - P'_\eps(z-r)|$ in different regions as given in \eqref{e:kernel_x-y}. In the first region, the left hand side is bounded by
        \begin{equation*}
            \begin{aligned}
                |y-r|^{-2}(|y-r|+\eps)^{-2}\int_{|z-r|\leq|y-r|}|z-r|^{-2} dz \lesssim \frac{\eps^{-1+\delta}}{|y-r|^{2+\delta}}.
            \end{aligned}
        \end{equation*}
    In the second region, the left hand side is bounded by
        \begin{equation*}
            \begin{aligned}
                |y-r|^{-2}\int_{|y-z|\geq2|y-r|}|z-y|^{-2}(|z-y|+\eps)^{-2} dz \lesssim \frac{\eps^{-1+\delta}}{|y-r|^{2+\delta}}.
            \end{aligned}
        \end{equation*}
    In the third region, the left hand side is bounded by
        \begin{equation*}
            \begin{aligned}
                |y-r|^{-2}\int_{|z-y|/2<\eps\wedge|y-r| }|z-y|^{-2}(|z-y|+\eps)^{-2} dz \lesssim \frac{\eps^{-1+\delta}}{|y-r|^{2+\delta}}.
            \end{aligned}
        \end{equation*}
    In the fourth region, the left hand side is bounded by
        \begin{equation*}
            \begin{aligned}
                |y-r|^{-3}\int_{|z-y|<2|y-r|}|z-y|^{-1}(|z-y|+\eps)^{-2} dz \lesssim \frac{\eps^{-1+\delta}}{|y-r|^{2+\delta}}.
            \end{aligned}
        \end{equation*}
    This completes the proof. 
\end{proof}

\subsection{The mixed $L^p$ norms and related bounds}

In light of the spectral gap inequality (the norms on the right hand side of \eqref{e:sg}), one will necessarily encounter sequentially mixed $L^2$ and $L^p$ norms when successively taking Malliavin derivatives. Hence, it is natural to introduce the following norm. 

\begin{defn} \label{def:l_u^p}
    Let $n\in\NN^+$ and $\vec {p}=(p_1,\dots,p_n)$\;, where $p_i\in [1,+\infty)$ for $i=1,\dots,n$\;. For a function $f:(\RR\times\TT_\eps)^n \rightarrow \RR$, we define
        \begin{equation*}
            \|f(\vec{u})\|_{L^{\vec{p}}_{\vec{u}}} \coloneqq \Big\|\cdots\|f\|_{L^{p_n}_{u_n}}\cdots\Big\|_{L^{p_1}_{u_1}} = \| f \|_{L_{u_1}^{p_1}\cdots L_{u_n}^{p_n}},
        \end{equation*}
    where $\vec {u}=(u_1,\dots,u_n)$. For a finite set $\pP\subset[1,+\infty)$, we define \begin{equation*}
        \|f\|_{\Sigma_{\vec{u}}^{\pP}}\coloneqq \sum_{\vec{p}\in\pP^n} \|f\|_{L_{\vec{u}}^{\vec{p}}}. 
    \end{equation*}
\end{defn}

To treat the additional $L^p$-term appearing in the spectral gap inequality, the main idea is to bound the $L^{\vec{p}}$-norm of an integral with the form on the left hand side of \eqref{e:l2_control_lp} below by its $L^2$-norm. The following lemma will be used repeatedly in the subsequent sections. Note that the form of the integrand on the left hand side of \eqref{e:l2_control_lp} is closely related to bound on $(P_\eps^\theta)'$ in Lemma~\ref{lem:P_eps}.

\begin{lem}\label{lem:Lvecpcontrol}
    Let $k\in\NN^+$ and $p_i\ge2$ for $i=1,\dots,k$. Suppose $ \alpha_i\in(\frac32,3)$ for $i=1,\dots,k$ with $\sum_{i=1}^k \alpha_i< \frac{3k+3}{2}$. Then for every function $f:\RR\times\TT\rightarrow
    \RR^+$, we have
        \begin{equation} \label{e:l2_control_lp}
            \begin{split}
            &\phantom{11}\bigg\| \int_{\RR\times\TT}  f(x) \cdot \prod_{i=1}^{k} \frac{\eps^{\frac32}}{(|x-\eps u_i|+\eps)^{\alpha_i}}dx  \bigg\|_{L_{\vec{u}}^{\vec{p}}}\\
            &\lesssim  \left\| \int_{\RR\times\TT}  f(x) \cdot \big( |x-r| + \eps \big)^{\frac{3(k-1)}{2} - \sum_{i=1}^{k}\alpha_i} \, dx \right\|_{L_r^{2}}, 
            \end{split}
        \end{equation}
   where the integrations are taken over $\vec{u}\in(\RR\times\TT_\eps)^k$ and $r\in\RR\times\TT$. Furthermore, the proportionality constant is independent of $\eps\in(0,1)$.
\end{lem}
\begin{proof}
    We first assume that for every $i=1,\dots,k$, we have either $p_i=2$ or $p_i \geq 4$. In this case, we have
        \begin{equation*}
            \begin{aligned}
                &\bigg\| \int_{\RR\times\TT}  f(x)\prod_{i=1}^{k} \frac{\eps^{\frac32}}{(|x-\eps u_i|+\eps)^{\alpha_i}}dx \bigg\|_{L_{\vec{u}}^{\vec{p}}}\\
                =&  \bigg\| \iint\limits_{(\RR\times\TT)^2} f(x)f(x')\prod_{i=1}^{k} \bigg(\frac{\eps^{\frac32}}{(|x-\eps u_i|+\eps)^{\alpha_i}} \frac{\eps^{\frac32}}{(|x'-\eps u_i|+\eps)^{\alpha_i}}\bigg)dxdx' \bigg\|_{L_{\vec{u}}^{\vec{\frac{p}{2}}}}^{\frac12} \\
                \le& \bigg(\iint\limits_{(\RR\times\TT)^2}  f(x)f(x')\prod_{i=1}^{k} \left\|\frac{\eps^3}{(|x- \eps u_i|+\eps)^{\alpha_i} (|x'- \eps u_i|+\eps)^{\alpha_i} }\right\|_{L_{u_i}^{\frac{p_i}{2}}}dxdx' \bigg)^{\frac12} \\
                \lesssim&  \bigg( \iint\limits_{(\RR\times\TT)^2}  \frac{f(x)f(x')}{(|x-x'|+\eps)^{{\sum_{i=1}^{k}(2\alpha_i-3)}}} dxdx'\bigg)^{\frac12},  
             \end{aligned}
        \end{equation*}
    where the last inequality follows from Lemma~\ref{lem:convolution_singularity} (for $p_i=2$) and remark~\ref{rmk:convolution_singularity'} (for $p_i \geq 4$). Note that
    \begin{equation*}
        (|x-x'|+\eps)^{-\alpha} \lesssim \int_{\RR\times\TT} (|x-r|+\eps)^{-\frac{\alpha+3}{2}}(|x'-r|+\eps)^{-\frac{\alpha+3}{2}} \, dr
    \end{equation*}
    for $\alpha\in(0,3)$. The conclusion \eqref{e:l2_control_lp} then follows for the above range of $p_i$. The remaining cases for $p_i$ follow from the interpolation between $p_i=2$ and $p_i=4$.
\end{proof}

\begin{rmk} \label{rmk:norm_simplification}
    If the components $p_i$ in the vector $\vec{p}$ are different, then the order of integration of variables $\vec{u}$ in the norm $\|\cdot\|_{L_{\vec{u}}^{\vec{p}}}$ matters. These norms arise from taking Malliavin derivatives. The orders with which derivatives are taken are recorded by the spacetime points $u_1, u_2, \dots$ in increasing order. Hence, in view of the spectral gap inequality \eqref{e:sg}, the orders of integration are also $u_1, u_2, \dots$ with increasing subscripts. 
    
    In the current article, the finite set $\pP$ of exponents is restricted to $[2, +\infty)$. Most of the times when such a norm is concerned, it appears in the form of Lemma~\ref{lem:Lvecpcontrol} and we will use this lemma to control it. In this case, in view of the upper bound in \eqref{e:l2_control_lp}, the exact order of integration of the variables $\vec{u}$ in the $\|\cdot\|_{L_{\vec{u}}^{\vec{p}}}$ then does not matter. 

    In the actual bounds in our article, we will control the $\|\cdot\|_{L_{\omega}^p}$-norm of the stochastic objects for arbitrarily large but fixed $p$. Hence, we will mostly encounter $\pP = \{2, p\}$ for a fixed $p>2$. So we write for simplicity
    \begin{equation*}
        \|\cdot\|_{\Sigma_{\vec{u}}^{p}} := \|\cdot\|_{\Sigma_{\vec{u}}^{\{2,p\}}}
    \end{equation*}
    if $p > 2$. In most of the situations, this exponent $p$ is the same as the one appearing in the statements (that is, the same $p$ as the $L_\omega^p$-norm of the stochastic object in concern). In this situation, we will also omit the $p$ and simply write $\Sigma_{\vec{u}}$ for $\Sigma_{\vec{u}}^{p}$. 
\end{rmk}

The following five lemmas provide the estimates of the right hand side of \eqref{e:l2_control_lp} with five different kinds of $f$ which has at least two singular points. Lemma~\ref{lem:kernel_convolution_6} corresponds to $f(x)=|\varphi^\lambda(x)||K'_\eps(x-y)|$, Lemma~\ref{lem:kernel_convolution_2} corresponds to $f(x)=|\varphi^\lambda(x)||K'_\eps(x-y)-K'_\eps(-y)|$, Lemma~\ref{lem:kernel_convolution_1} corresponds to $f(y) = (|y-x|+\eps)^{-\alpha}|y-z|^{-\beta}$, Lemma~\ref{lem:kernel_convolution_4} corresponds to $f(y)=|K'_\eps(x-y)-K'_\eps(-y)|$, and Lemma~\ref{lem:kernel_convolution_3} corresponds to $f(y) = |K'_\eps(x-y)-K'_\eps(-y)| |y-z|^{-2}$.

\begin{lem} \label{lem:kernel_convolution_6}
    For every $\varphi\in \bar{C}^1_c$,  $\eps,\lambda\in(0,1)$ and every $\delta\in(0,\frac{1}{2})$, we have the bound
    \begin{equation*}
        \bigg\|\int _{\RR\times\TT} |\varphi^\lambda(x)| |K'_\eps(x-y)| (|x-r|+\eps)^{-2}
        dx \bigg\|_{L^2_r} \lesssim \frac{\lambda^{-\frac{1}{2}+\delta} \eps^{-\delta} \boldsymbol{1}_{|y|\lesssim1}}{|y|^2},
    \end{equation*}
    where the proportionality constant depends on $\delta$ only.
\end{lem}
\begin{proof}
    For the case $|y|\geq 2\lambda$, the left hand side can be bounded by
    \begin{equation*}
        \bigg( \iint\limits_{|x|,|x'|\leq\lambda} \frac{\boldsymbol{1}_{|y|\lesssim 1}}{|y|^4} \frac{|\varphi^\lambda(x)\varphi^\lambda(x')|}{|x-x'|+\eps} dxdx' \bigg)^{\frac{1}{2}} \lesssim  \frac{\lambda^{-\frac{1}{2}+\delta} \eps^{-\delta} \boldsymbol{1}_{|y|\lesssim1}}{|y|^2}.
    \end{equation*}
    For the case $|y|< 2\lambda$, we bound the left hand side by
    \begin{equation*}
        \bigg( \iint\limits_{|x|,|x'|\leq\lambda} \frac{\boldsymbol{1}_{|y|\lesssim 1}|\varphi^\lambda(x)\varphi^\lambda(x')|}{|x-y|^2 |x'-y|^2 (|x-x'|+\eps)} dxdx' \bigg)^{\frac{1}{2}} \lesssim \eps^{-\delta} \lambda^{-\frac{5}{2}+\delta} \lesssim \frac{\lambda^{-\frac{1}{2}+\delta} \eps^{-\delta} \boldsymbol{1}_{|y|\lesssim1}}{|y|^2},
    \end{equation*}
    where the first inequality follows from the change of the variable $(x-y,x'-y)\mapsto(x,x')$. This completes the proof.
\end{proof}

\begin{lem} \label{lem:kernel_convolution_2}
    For every $\varphi\in \bar{C}^1_c$,  $\eps,\lambda\in(0,1)$, $\alpha\in[2,\frac{5}{2}]$ and every $\delta\in(0,\frac{1}{8})$, we have the bound
    \begin{equation*}
        \bigg\|\int _{\RR\times\TT} |\varphi^\lambda(x)| |K'_\eps(x-y)-K'_\eps(-y)| (|x-r|+\eps)^{-\alpha}
        dx \bigg\|_{L^2_r} \lesssim \frac{\lambda^{-3\delta} \eps^{-\delta} \boldsymbol{1}_{|y|\lesssim1}}{|y|^{\alpha+\frac{1}{2}-4\delta}},
    \end{equation*}
    where the proportionality constant depends on $\alpha,\delta$ only.
\end{lem}
\begin{proof}
    We partition the integration domain into three regions: $\{ |x|\geq2|y| \}$, $\{ |x-y|\leq\frac{|y|}{2} \}$ and $\{ |x|<2|y|,\;|x-y|>\frac{|y|}{2} \}$. We use different bounds of $|K'_\eps(x-y) - K'_\eps(-y)|$ in different regions as given in \eqref{e:kernel_x-y}. In the first region, the left hand side is bounded by
    \begin{equation*}
        \bigg( \iint\limits_{|x|,|x'|\leq\lambda} \frac{\boldsymbol{1}_{|y|\lesssim\lambda}}{|y|^4} \frac{|\varphi^\lambda(x)\varphi^\lambda(x')|}{(|x-x'|+\eps)^{2\alpha-3}} dxdx' \bigg)^{\frac{1}{2}} \lesssim  \frac{\lambda^{-3\delta}\eps^{-\delta}\boldsymbol{1}_{|y|\lesssim1}}{|y|^{\alpha+\frac{1}{2}-4\delta}}.
    \end{equation*}
    In the second region, we can bound the left hand side by
    \begin{equation*}
        \begin{split}
        \bigg( \iint\limits_{|x|,|x'|\leq\lambda} \frac{\boldsymbol{1}_{|y|\lesssim\lambda}}{|x-y|^2 |x'-y|^2} \frac{|\varphi^\lambda(x)\varphi^\lambda(x')|}{(|x-x'|+\eps)^{2\alpha-3}} dxdx' \bigg)^{\frac{1}{2}} &\lesssim \lambda^{-\alpha-\frac{1}{2}+\delta} \eps ^{-\delta} \boldsymbol{1}_{|y|\lesssim\lambda}\\
        &\lesssim  \frac{\lambda^{-3\delta} \eps^{-\delta}\boldsymbol{1}_{|y|\lesssim1}}{|y|^{\alpha+\frac{1}{2}-4\delta}}.
        \end{split}
    \end{equation*}
    In the third region, the left hand side can be bounded by
    \begin{equation*}
    \begin{split}
        &\bigg( \iint\limits_{|x|,|x'|\leq\lambda} \Big(\frac{|x|^{\alpha-\frac{3}{2}-3\delta}|x'|^{\alpha-\frac{3}{2}-3\delta}\boldsymbol{1}_{|y|\lesssim1}}{|y|^{2\alpha+1-6\delta} } + \frac{\boldsymbol{1}_{|y|<\eps}}{|y|^4} \Big)
        \frac{|\varphi^\lambda(x)\varphi^\lambda(x')|}{(|x-x'|+\eps)^{2\alpha-3}} dxdx' \bigg)^{\frac{1}{2}}\\
        &\lesssim \frac{\lambda^{-3\delta}\eps^{-\delta}\boldsymbol{1}_{|y|\lesssim1}}{|y|^{\alpha+\frac{1}{2}-4\delta}}.
    \end{split}
    \end{equation*}
    This completes the proof.
\end{proof}

\begin{lem} \label{lem:kernel_convolution_1}
    Let $\alpha, \beta \in (0,3)$ and $\gamma \in (\frac{3}{2},3)$ with the further restrictions that
    \begin{equation*}
        \alpha + \beta>3\;, \quad \alpha + \gamma \leq \frac{9}{2}\;, \quad \beta + \gamma \leq \frac{9}{2}\;.
    \end{equation*}
    Then for every $\delta>0$ sufficiently small, we have the bound
    \begin{equation*}
        \bigg\|\int _{\RR\times\TT} \frac{1}{(|x-y|+\eps)^{\alpha}} \frac{1}{(|y-r|+\eps)^{\gamma}} \frac{1}{|z-y|^{\beta}} dy \bigg\|_{L^2_r} \lesssim \eps^{-\delta} \frac{1}{(|x-z|+\eps)^{\alpha+\beta+\gamma-\frac{9}{2}-\delta}},
    \end{equation*}
    where the proportionality constant is independent of $\eps \in (0,1)$. 
\end{lem}
\begin{proof}
    The proof for this lemma is similar to Lemma~\ref{lem:kernel_convolution_2}. We split the integration domain into $\{y: |y-x|\leq\frac{|x-z|}{2}\}$, $\{y: |y-z|\leq\frac{|x-z|}{2}\}$ and $\{y: |y-x|,|y-z|>\frac{|x-z|}{2}\}$, and then apply Lemma~\ref{lem:convolution_singularity}.           We omit the details here.
\end{proof}
\begin{lem}\label{lem:kernel_convolution_4}
    For every $\alpha\in(2,\frac{5}{2})$ and $\delta\in(0,\frac{1}{8})$, we have the bound
    \begin{equation*}
        \bigg\| \int_{\RR\times\TT} |K'_\eps(x-y) - K'_\eps(-y)| (|y-r|+\eps)^{-\alpha}  dy \bigg\|_{L^2_r} \lesssim\eps^{2+\delta-\alpha}(|x|^{\frac{1}{2}-\delta} + \eps^{\frac{1}{2}-\delta}),
    \end{equation*}
    where the proportionality constant depends on $\alpha$ and $\delta$ only.
\end{lem}
\begin{proof}
    The proof for this lemma is similar to Lemma~\ref{lem:kernel_convolution_2}. We partition the integration domain into three regions: $\{|y|\le\frac{|x|}{2} \}$, $\{ |x-y|\leq\frac{|x|}{2} \}$ and $\{ |y|>\frac{|x|}{2},\;|x-y|>\frac{|x|}{2} \}$. We omit the details here.
\end{proof}

\begin{rmk}
    The coefficient $\eps^{-\delta}$ appears in Lemma~\ref{lem:kernel_convolution_2} only when $\alpha=\frac{5}{2}$, and in Lemma~\ref{lem:kernel_convolution_1} only when either $2\alpha + 2\gamma=9$ or $2\beta+2\gamma=9$.
\end{rmk}

\begin{lem} \label{lem:kernel_convolution_3}
    For every $\eps\in(0,1)$ , $\delta\in(0,\frac{1}{2})$, $z\in\RR\times\TT$ and $x\in\RR\times\TT$ with $|x|\leq2$, we have the bound
    \begin{equation*}
        \bigg\|\int_{\RR\times\TT}  \frac{|K'_\eps(x-y)-K'_\eps(-y)|}{|y-z|^2 \big(|y-r|+\eps\big)^{\frac52}} dy \bigg\|_{L^2_r} \lesssim \eps^{-\frac12+\delta} \big(|x|^{\frac12-\delta}+\eps^{\frac12-\delta} \big) \, \bigg(\frac{1}{|z|^2}+\frac{1}{|z-x|^2}\bigg)\;,
    \end{equation*}
    where the proportionality constant depends on $\delta$ only.
\end{lem}
\begin{proof}
    We divide the integration domain into four parts $\{ |y| \leq \frac{|x|}{2}\}$, $\{ |y-x|\leq\frac{|x|}{2}\}$, $\{|y|>\frac{|x|}{2},\; |y-x|>\frac{|x|}{2} ,\;|y|\le\eps\}$  and $\{|y|>\frac{|x|}{2},\; |y-x|>\frac{|x|}{2} ,\;|y|>\eps\}$. In the first region, the left hand side is bounded by
        \begin{equation*}
            \bigg( \iint\limits_{ |y|,|y'| \le\frac{|x|}{2} } \frac{{1}}{|y-z|^2|y'-z|^2} \frac{{1}}{(|y-y'|+\eps)^2}\frac{{1}}{|y|^2|y'|^2} dydy' \bigg)^{\frac{1}{2}}.
        \end{equation*}
    If $|z|\ge|x|$, we can bound the integral by 
        \begin{equation*}
            \frac{1}{|z|^2}\bigg( \iint\limits_{|y|,|y'| \le\frac{|x|}{2}}\frac{{1}}{(|y-y'|+\eps)^2}\frac{{1}}{|y|^2|y'|^2} dydy' \bigg)^{\frac{1}{2}}\lesssim \frac{\eps^{-\frac12+\delta}|x|^{\frac12-\delta}}{|z|^2}.
        \end{equation*}
    If $|z|<|x|$, we can bound the integral by 
        \begin{equation*}
            \begin{aligned}
                &\frac{1}{|z|^2}\bigg( \iint\limits_{|y|,|y'| \in \big(\frac{|z|}{2},\frac{|x|}{2}\big) }\frac{{1}}{(|y-y'|+\eps)^2}\frac{{1}}{|y-z|^2|y'-z|^2} dydy' \bigg)^{\frac{1}{2}}\\
                &+\frac{1}{|z|^2}\bigg( \iint\limits_{|y|,|y'| \le\frac{|z|}{2}}\frac{{1}}{(|y-y'|+\eps)^2}\frac{{1}}{|y|^2|y'|^2} dydy' \bigg)^{\frac{1}{2}} \lesssim \frac{\eps^{-\frac12+\delta}|x|^{\frac12-\delta}}{|z|^2}.
            \end{aligned}
        \end{equation*}
    In the second region, the left hand side is bounded by
        \begin{equation*}
            \bigg( \iint\limits_{|y-x|,|y'-x| \le\frac{|x|}{2}} \frac{{1}}{|y-z|^2|y'-z|^2} \frac{{1}}{(|y-y'|+\eps)^2}\frac{{1}}{|y-x|^2|y'-x|^2} dydy' \bigg)^{\frac{1}{2}}.
        \end{equation*}
    A change of variable $(y-x,y'-x) \mapsto (y,y')$ reduces it to the previous situation with $z$ replaced by $z-x$. 
    In the third region, the left hand side is bounded by
        \begin{equation*}
            \begin{aligned}
                &\frac{1}{|z|^2}\bigg( \iint\limits_{|y|,|y'|\le \eps \wedge \frac{|z|}{2}} \frac{\eps^{-2}}{|y|^2|y'|^2} dydy' \bigg)^{\frac{1}{2}}+\frac{1}{|z|^2}\bigg( \iint\limits_{|y|,|y'|\in \big(\frac{|z|}{2},\eps\big]} \frac{\eps^{-2}}{|y-z|^2|y'-z|^2} dydy' \bigg)^{\frac{1}{2}}\lesssim\frac{1}{|z|^2}.
            \end{aligned}
        \end{equation*}
    In the fourth region, the left hand side is bounded by
        \begin{equation*}
            |x|^{\frac{1}{2}}\bigg( \iint\limits_{|y|,|y'| >\frac{|x|}{2};|y-x|,|y'-x| >\frac{|x|}{2}} \frac{{1}}{|y-z|^2|y'-z|^2} \frac{{1}}{(|y-y'|+\eps)^2}\frac{{1}}{|y|^{\frac{5}{2}}|y'|^{\frac{5}{2}}} dydy' \bigg)^{\frac{1}{2}}.
        \end{equation*}
    If $|z|<\frac14|x|$, we can bound the integral by 
        \begin{equation*}
            \begin{aligned}
                &|x|^{-\frac{3}{2}} \bigg( \iint\limits_{|y|,|y'|>\frac{|x|}{2}}\frac{{1}}{(|y-y'|+\eps)^2}\frac{{1}}{|y|^{\frac{5}{2}}|y'|^{\frac{5}{2}}} dydy' \bigg)^{\frac{1}{2}} \lesssim \frac{\eps^{-\frac12+\delta}|x|^{\frac12-\delta}}{|z|^2}.
            \end{aligned}
        \end{equation*}
     If $|z|\ge\frac14|x|$, we can bound the integral by  
        \begin{equation*}
            \begin{aligned}
                &\frac{|x|^{\frac{1}{2}}}{|z|^2}\bigg( \iint\limits_{|y-z|,|y'-z|\ge\frac{|z|}{2};|y|,|y'|>\frac{|x|}{2}} \frac{{1}}{(|y-y'|+\eps)^2}\frac{{1}}{|y|^{\frac{5}{2}}|y'|^{\frac{5}{2}}} dydy' \bigg)^{\frac{1}{2}}\\
                &+\frac{|x|^{\frac{1}{2}}}{|z|^\frac{5}{2}}\bigg( \iint\limits_{|y-z|,|y'-z|<\frac{|z|}{2}} \frac{{1}}{|y-z|^2|y'-z|^2} \frac{{1}}{(|y-y'|+\eps)^2} dydy' \bigg)^{\frac{1}{2}}\lesssim\frac{\eps^{-\frac12+\delta}|x|^{\frac12-\delta}}{|z|^2}.
            \end{aligned}
        \end{equation*}
    This concludes the proof. 
\end{proof}

\subsection{Regularisation and decomposition of the nonlinearity}
\label{sec:regularisation_nonlinearity}

Once using the spectral gap inequality in Proposition~\ref{prop:spec_gap}, we gain a factor $\sqrt{\eps}$. It turns out that successively taking Malliavin derivatives on the noise to eliminate negative powers of $\eps$ requires at least three derivatives of $F$, which is strictly stronger than Assumption~\ref{as:F}. 

To circumvent it, we employ the trick in \cite[Section~5]{HX19} to decompose $F$ into a regular part $F_\zeta$ and a small remainder $F - F_\zeta$, and use different methods to control these two parts. Here, we describe the regularisation $F_\zeta$ and give its first properties. Detailed bounds concerning the stochastic objects will be given in Section~\ref{sec:convergence} below. 

Fix a smooth function $\rho: \RR \rightarrow \RR$ with $\int_{\RR} \rho = 1$ and that its Fourier transform has compact support. For $\zeta \in (0,1)$, let $\rho_\zeta= \zeta^{-1}\rho(\cdot/\zeta)$, and $F_\zeta := F * \rho_\zeta$. 

\begin{rmk} \label{rmk:rho}
    The requirement that the Fourier transform of $\rho$ has compact support ensures that the chaos expansion series of $F_\zeta^{(\ell)}(\sqrt{\eps}\Psi_\eps)$ converges in $L^p(\Omega)$ for some $p>2$ and $\ell\in\{0,1,2\}$, which will be needed to swap the expectation and the summation in \eqref{e:expansion}. This strong restriction on $\rho$ will not affect the main statement since $F_\zeta$ is just an intermediate quantity in the proof. 
\end{rmk}

The following lemma provides some estimates of the derivatives of $F_\zeta$.

\begin{lem} \label{lem:F_zeta}
    Suppose $F$ satisfies Assumption~\ref{as:F} with $\beta$ and $M$ in that assumption. Then, we have
    \begin{equation} \label{e:F_zeta}
        |F^{(n)}_\zeta(w)| \lesssim_n  (1 + \zeta^{-(n-2)}) (1+|w|)^M\;,
    \end{equation}
    and
    \begin{equation} \label{e:F-F_zeta}
        |F''(w) - F''_\zeta(w)| \lesssim \zeta^\beta (1+|w|)^M\;.
    \end{equation}
    Both bounds are uniform in $\zeta \in (0,1)$ and $w \in \RR$. 
\end{lem}
\begin{proof}
    For $n=0,1,2$, \eqref{e:F_zeta} is a direct corollary of Assumption~\ref{as:F} and the definition of $F_\zeta$. For $n\geq2$, we have
    \begin{equation*}
    \begin{split}
        |F^{(n)}_\zeta(w)| &= |\big(F''\ast\rho^{(n-2)}\big) (w)| \lesssim \int_\RR (1+|w-x|)^M \zeta^{-n+1} \Big|\rho^{(n-2)}\Big( \frac{x}{\zeta}\Big)\Big| dx\\
        &= \zeta^{-n+2} \int_\RR (1+|w-\zeta x|)^M |\rho^{(n-2)}(x)| dx \lesssim \zeta^{-n+2} (1+|w|)^M. 
    \end{split}
    \end{equation*}
    For the bound \eqref{e:F-F_zeta}, we have
    \begin{equation*}
    \begin{split}
        |F''(w) - F''_\zeta(w)| &= \Big|\int_\RR \big(F''(w) - F''(w-x) \big) \rho_\zeta(x) dx \Big|\\
        &\lesssim \int_\RR |x|^\beta (1+|w|+|x|)^M |\rho_\zeta(x)| dx \lesssim \zeta^\beta (1+|w|)^M\;,
    \end{split}
    \end{equation*}
    where the last bound follows from the change of variable $x \mapsto \zeta x$. This completes the proof.
\end{proof}

\section{Convergence of the stochastic objects}
\label{sec:convergence}
	
This section aims to prove that $\hPi^\eps \rightarrow \Pi^{\KPZ}$ in distribution as $\eps\rightarrow0$, where $\Pi^{\KPZ}$ is the standard KPZ model described in \cite[Appendix~A]{HX19}. Recall from Section~\ref{sec:model} and \eqref{e:models_comparison} the class of intermediate models $\hPi^{\HS(\eps)}$ studied in \cite{KPZCLT}. The main theorem of \cite{KPZCLT} implies that $\hPi^{\HS(\eps)} \rightarrow \Pi^{\KPZ}$ in $\$ \cdot;\cdot \$_{\eps;0}$ in distribution as $\eps \rightarrow 0$\footnote{\cite{KPZCLT} assumes finite range correlation in the microscopic noise, which corresponds to $\theta$ being compactly supported. But the proof there also works for the $\theta$ with our decay assumption \eqref{e:theta_decay}. Moreover, the arguments in the current article (for Theorem~\ref{thm:convergence} below) can also show that $\widehat{\Pi}^{\HS(\eps)} \rightarrow \widehat{\Pi}^{\KPZ}$.}. Hence, it remains to show $\$ \widehat{\Pi}^\eps - \hPi^{\HS(\eps)} \$_\eps \rightarrow 0$ in probability as $\eps\rightarrow0$. 

Recall from \eqref{e:test_function_space} and \eqref{e:test_function} the test function space $\bar{C}_c^\alpha$ and notion of $\varphi_z^\lambda$ for a re-centered and rescaled version of $\varphi \in \bar{C}_c^\alpha$. Also recall we write $\varphi^\lambda$ for $\varphi_0^\lambda$ for simplicity. According to Kolmogorov-type convergence criterion, the desired convergence of the model will follow from the following theorem. 

\begin{thm} \label{thm:convergence}
	For every $\tau$ listed in Table~\ref{e:symbols}, every $p\geq2$ and every sufficiently small $\delta>0$, there exists $\delta' >0$ such that
	\begin{equation*}
		\sup_{z\in\RR^{+}\times\TT, \, \varphi\in \bar{C}_c^1} \Big( \EE|\scal{\widehat{\Pi}^\eps_z \tau - \hPi^{\HS(\eps)}_z \tau , \varphi^\lambda_z}|^p \Big)^{\frac{1}{p}} \lesssim_p \eps^{\delta'} \lambda^{|\tau|-\delta}
	\end{equation*}
    uniformly in $\eps, \lambda \in (0,1)$, where $|\tau|$ represents the homogeneity of $\tau$ (as specified in Table \eqref{e:noise}). As a consequence, we have $\hPi^{\eps} \rightarrow \Pi^{\KPZ}$ in distribution in $\$ \cdot;\cdot \$_{\eps;0}$. 
\end{thm}

\begin{rmk}
    There are actually more stochastic objects in the definition of regularity structures than those in Table~\eqref{e:symbols}. Rigorously speaking, one needs to prove the convergence for all of them. However, Table~\ref{e:symbols} include all objects with negative homogeneity, and by \cite[Proposition~6.3]{HQ}, the convergences for objects with positive homogeneity follow from those for the negative ones. This enables us to restrict the study to negative homogeneity ones only. 
\end{rmk}

There are ten objects in Table~\eqref{e:symbols}. We provide details for two of them: $\<1'>$ and $\<2'1'1'>$, in Sections~\ref{sec:1} and~\ref{sec:211} respectively. The object $\<1'>$ illustrates the use of spectral gap inequality, and derivation of its bounds contains ingredients that are useful for more complicated objects. The object $\<2'1'1'>$ is the most complicated one. The derivation of the bounds for it demonstrates the subtlety and the use of various additional tricks, and we hope it gives sufficient amount of details so that the readers are convinced that bounds for all other objects can be obtained with the same techniques but in much simpler manner. 

For simplicity, we will write $\tau_\eps$ for $\hPi_0^\eps \tau$, and $\tau_\eps^{(\zeta)}$ for $\hPi_0^\eps \tau$ with the modification that each appearance of $F$ or its derivatives are replaced by $F_\zeta$ or its derivatives. 

\subsection{The Malliavin derivative and bounds on the free field}

The free field $\Psi_\eps$ is the building block of all the stochastic objects. Recall from \eqref{e:noise} and \eqref{e:free_field} the definition of $\Psi_\eps$. Using Fubini Theorem to change the order of the integration and that $P_\eps'$ is odd in its spatial variable, we get the representation
\begin{equation*}
    \Psi_\eps(x) = \int_{\RR \times \TT_\eps} \eps^{-\frac{3}{2}} \bigg[ \int_{\RR \times \TT} P_\eps'(x-y) \, \theta^{(\eps)} \Big( \frac{y}{\eps} -u \Big) \, dy \bigg] \, \eta^{(\eps)}(du)\;, 
\end{equation*}
where we recall $y / \eps := (y_0 / \eps^2, y_1 / \eps)$ for $y = (y_0, y_1) \in \RR \times \TT$. Similarly, we have
\begin{equation*}
    \<1>_\eps(x) = \int_{\RR \times \TT_\eps} \eps^{-\frac{3}{2}} \bigg[ \int_{\RR \times \TT} P_0'(x-y) \, \theta^{(\eps)} \Big( \frac{y}{\eps} -u \Big) \, dy \bigg] \, \eta^{(\eps)}(du)\;.
\end{equation*}
Hence, it is natural to define the family of functions $P_{\eps_1. \eps_2}^{\theta}$ (for parameters $\eps_1 \in [0,1]$ and $\eps_2 \in (0,1)$) on $(\RR \times \TT) \times (\RR \times \TT_{\eps_2})$ by
    \begin{equation} \label{e:kernel_P^theta}
        P^\theta_{\eps_1,\eps_2} (x,u) \coloneqq \int_{\RR\times\TT}  P_{\eps_1}(x-y) \;\eps_2^{-\frac{3}{2}} \, \theta^{(\eps_2)} \Big(\frac{y}{\eps_2}-u \Big) dy\;.
    \end{equation}
As in the case for the heat kernel $P_\eps$ and its truncation $K_\eps$, we write
\begin{equation} \label{e:kernel_P^theta_derivative}
    (P_{\eps_1,\eps_2}^\theta)'(x,u) := (\d_{x_1} P_{\eps_1,\eps_2}^\theta )(x,u)
\end{equation}
for its partial derivative with respect to the spatial component of its first variable (the $x_1 \in \TT$ component of $x \in \RR \times \TT$). 

\begin{rmk}
In fact, as one can see from the expression \eqref{e:kernel_P^theta}, $P_{\eps_1, \eps_2}^\theta$ is actually function of $x - \eps_2 u$ on $\RR \times \TT$ and symmetric in its spatial ($\TT$) component. $(P_{\eps_1, \eps_2}^\theta)'$ is the derivative in its spatial component, and hence integrates to $0$. But we still write it as a function of both $x$ and $u$ to emphasise the difference of the two domains. 
\end{rmk}

For the kernel $P^\theta_{\eps_1,\eps_2}$, we will encounter two situations: either $\eps_1 = \eps_2 = \eps$, or $\eps_1 = 0$ and $\eps_2 = \eps$. In the former, we simply write
\begin{equation} \label{e:kernel_P^theta_simplify}
    P^\theta_\eps := P^\theta_{\eps,\eps} \quad \text{and} \quad (P^\theta_\eps)' := (P^\theta_{\eps,\eps})'\;.
\end{equation}
With the above notations, we then have
\begin{equation} \label{e:free_field_rep}
    \Psi_\eps(x) = \int_{\RR \times \TT_\eps} (P_{\eps}^{\theta})'(x,u) \, \eta^{(\eps)}(du)\;, \qquad \<1>_\eps(x) = \int_{\RR \times \TT_\eps} (P_{0,\eps}^{\theta})'(x,u) \, \eta^{(\eps)}(du)\;.
\end{equation}
Note that since $x \in \RR \times \TT$ and $u \in \RR \times \TT_\eps$ live in differently scaled domains, we should not expect $P_\eps^\theta$ to behave like the standard heat kernel. We have the following lemma regarding the behaviours of $(P_{\eps}^\theta)'$ and $(P_{\eps}^{\theta})' - (P_{0,\eps}^{\theta})'$. 

\begin{lem} \label{lem:P_eps}
   Suppose $\qQ$ satisfies Assumption~\ref{as:Q}. For every $\delta\in [0,1)$, we have
   \begin{equation*}
       |(P^\theta_\eps)' (x,u)| \lesssim \frac{\eps^{\frac32}}{(|x-\eps u|+\eps)^{2}}\;, \quad |(P^\theta_\eps-P^\theta_{0,\eps})' (x,u) |\lesssim\eps^{\delta}\frac{\eps^{\frac32}}{(|x-\eps u|+\eps)^{2+\delta}}\;,
   \end{equation*}
   where the proportionality constants are independent of $x\in\RR\times\TT$, $u\in\RR\times\TT_\eps$ and $\eps\in(0,1)$.
\end{lem}
\begin{proof}
    We provide details for the term $(P^\theta_\eps)'$. The bound for the difference $(P^\theta_\eps-P^\theta_{0,\eps})'$ can be obtained in a similar way. 
    
    Recall the decay of $\theta$ from \eqref{e:theta_decay}. For $v=(v_0,v_1)\in \RR\times\TT_\eps$, assuming without loss of generality that $v_1\in[-\frac{1}{2\eps},\frac{1}{2\eps}]$, we have
        \begin{equation}\label{e:thetaeps}
        \begin{aligned}
            |\theta^{(\eps)}(v_0,v_1)| &\le \sum_{k\in\ZZ} \Big|\theta \Big(v_0,v_1+\frac{k}{\eps}\Big) \Big| \lesssim \sum_{k\in\ZZ} \Big(1+\sqrt{|v_0|} + \Big|v_1+\frac{k}{\eps}\Big| \Big)^{-4-\delta_0}\\
            &\lesssim \frac{1+\eps\big(1+\sqrt{|v_0|}+|v_1|\big)}{\big(1+\sqrt{|v_0|}+|v_1|\big)^{4+\delta_0}}\lesssim(1+|v|)^{-3-\delta_0},
        \end{aligned}
        \end{equation}
    where the proportionality constant is independent of $\eps\in(0,1)$ and $v\in \RR\times\TT_\eps$. Combining \eqref{e:kernel_singularity'}, \eqref{e:thetaeps}, and Lemma~\ref{lem:convolution_singularity'}, we get
        \begin{equation*}
            |(P^\theta_\eps)' (x,u)|\lesssim\int _{\RR\times\TT} \frac{\eps^{\frac{3}{2}+\delta_0}}{|x-y|^2(|y-\eps u|+\eps)^{3+\delta_0}}dy\lesssim\frac{\eps^{\frac32}}{(|x-\eps u|+\eps)^{2}}.
        \end{equation*}
    This completes the proof.
\end{proof}

The following lemma gives the Malliavin derivative of the free field.
\begin{lem} \label{lem:derivative_genFF}
    For $G\in \cC^n(\RR,\RR)$ and $u_1,\dots,u_n\in \RR\times\TT_\eps$\;, we have
    \begin{equation} \label{e:discrete_deri}
        D^n_{\vec{u}} \big( G(\sqrt{\eps}\Psi_\eps(x)) \big) = \idotsint\limits_{\prod_{i=1}^{n} [ 0, \sqrt{\eps} (P^\theta_\eps)'(x,u_i) ]} G^{(n)} \Big(\sqrt{\eps}\Psi_\eps(x)+\sum_{i=1}^n r_i\Big) dr_1\cdots dr_n\;,
    \end{equation}
    where $\vec{u} = (u_1, \dots, u_n)$. In particular, for $G(x)=x$, we have 
    \begin{equation} \label{e:D_psi}
        D_{u}(\Psi_\eps(x))=(P_{\eps}^{\theta})'(x,u).
    \end{equation}
\end{lem}
\begin{proof}
    Recall the representation \eqref{e:free_field_rep}. For $n=1$, we have
    \begin{equation*}
            D_{u_1}(G(\sqrt{\eps}\Psi_\eps(x))) = G\big(\sqrt{\eps}\Psi_\eps(x)+\sqrt{\eps}(P_{\eps}^{\theta})'(x,u_1)\big)-G(\sqrt{\eps}\Psi_\eps(x))\;,
    \end{equation*}
    and the claim (for $n=1$) follows from the fundamental theorem of calculus. The proof for the cases $n\geq2$ follows by induction. 
\end{proof}

As a simple application of the spectral gap inequality, given that $F$ satisfies Assumption~\ref{as:F}, the following lemma guarantees $ F^{(\ell)}(\sqrt{\eps} \Psi_\eps(x))\in L^p(\Omega)$ uniformly in $\eps\in(0,1)$ for every $p\geq2$ and $\ell\in\{0,1,2\}$.

\begin{lem}\label{Lpfinite}
    Let $G$ be a continuous function with polynomial growth. Then for every $p\ge2$, we have 
        \begin{equation} \label{e:deri_free_field}
            \left\|G \big( \sqrt{\eps} \Psi_\eps(x) \big)\right\|_{L^p_\omega}\lesssim_p1,
        \end{equation}
    where the proportionality constant is independent of $x\in\RR\times\TT$ and $\eps\in(0,1)$.
\end{lem}
\begin{proof}
    Since $G$ grows at most polynomially and has no singularity, it suffices to show \eqref{e:deri_free_field} for $G$ being the identity function for arbitrary $p \geq 2$. Recall the definition of $\Sigma$-norm from Definition~\ref{def:l_u^p} and Remark~\ref{rmk:norm_simplification} that we write $\|\cdot\|_{\Sigma_u}$ for $\|\cdot\|_{\Sigma^{\{2,p\}}_u}$ here. By \eqref{e:D_psi}, the spectral gap inequality \eqref{e:sg} and Lemma~\ref{lem:P_eps}, we have 
    \begin{equation*}
        \begin{aligned}
            \left\|\sqrt{\eps} \Psi_\eps(x)\right\|_{L^p_\omega}\lesssim_p   \left\|\sqrt{\eps} (P^\theta_\eps)' (x,u) \right\|_{\Sigma_u L^p_\omega} \lesssim  \left\|\frac{\eps^2}{(|x-\eps u|+\eps)^{2}}\right\|_{\Sigma_u}\lesssim1\;,
        \end{aligned}
    \end{equation*}
    where the last inequality follows from the change of variable $\eps u \mapsto y \in \RR \times \TT$ and that $2p-3>0$ for $p \geq 2$. 
\end{proof}

\subsection{The coupling constant} \label{sec:coupling_constant}

    In this section, we will prove the convergence of $a_\eps\coloneqq \frac12\EE F''(\sqrt{\eps}\Psi_\eps)$ to the coupling constant $a$ as $\eps\rightarrow0$, where $a$ is given by \eqref{e:coupling}. 

    Recall from Remark~\ref{rmk:coupling} that $\bar{\eta}$ denotes the Poisson point process on $\RR^2$ with unit intensity, and $\bar{P}$ and $\bar{P}'$ denotes the Green's function for $\d_t - \lL$ on $\RR^2$ and its spatial derivative. For the same function $\theta$, define $(\bar{P}^\theta)': \RR^2 \rightarrow \RR$ by
    \begin{equation*}
        (\bar{P}^{\theta})'(u)\coloneqq \int_{\RR^2}\bar{P}'(-y)\theta(y-u)dy.
    \end{equation*}
    Then we have $(\bar{P}' * \bar{\xi})(0) = \bar{I}_1 \big( (\bar{P}^{\theta})' \big)$, where $\bar{I}_1$ is the first order Wiener-It\^o integral associated to $\bar{\eta}$. In particular, this implies
	\begin{equation*} 
		a = \frac{1}{2} \EE F''\big( I_1((\bar{P}^{\theta})') \big)\;.
	\end{equation*}

\begin{prop} \label{prop:coupling_constant}
     We have 
        \begin{equation*}
            \lim_{\eps\rightarrow0}a_\eps=a.
        \end{equation*}
\end{prop}
\begin{proof}
    Recall from Assumption~\ref{as:F} that $F''$ grows at most polynomially. By \eqref{e:deri_free_field} with $G(x)=x$ and the stationarity of $\Psi_\eps$, it suffices to show that 
        \begin{equation*}
            \sqrt{\eps}\Psi_\eps(0) \rightarrow I_1((\bar{P}^{\theta})')
        \end{equation*}
    in distribution as $\eps\rightarrow 0$. The representation \eqref{e:free_field_rep} implies that $\Psi_\eps (x) = I_1 \big( (P_{\eps}^{\theta})'(x,\cdot) \big)$. Hence, by Lemma~\ref{chf}, the characteristic function of $\sqrt{\eps}\Psi_\eps(0)$ is 
         \begin{equation*}
             \mathbb E e^{i t \sqrt{\eps}\Psi_\eps(0)}
             =\exp\bigg( \int_{\RR\times\TT_\eps} \left(e^{i t \sqrt{\eps}(P_{\eps}^{\theta})'(0,u)} - i t \sqrt{\eps}(P_{\eps}^{\theta})'(0,u) -1\right)du\bigg).
         \end{equation*}    
    For $u=(u_0,u_1)\in \RR\times\TT_\eps$, recalling that $\theta^{(\eps)}$ is the $\frac{1}{\eps}$-periodisation in space of the spacetime function $\theta$ on $\RR^2$, we have the expression (in terms of Fourier series)
        \begin{equation*}
            \sqrt{\eps}(P_{\eps}^{\theta})'(0,u) =-  \sum_{k\in \ZZ} e^{2\pi i\eps k u_1} \int_{-\infty}^{0}  (2\pi ik) \, e^{\eps^{-2}\qQ(2\pi \eps k)s} \, \widehat{\theta}\bigg(\frac{s}{\eps^2}-u_0,\eps k\bigg) ds\;,
        \end{equation*}
    where $\widehat{\theta}$ is the Fourier transform of $\theta$ in the space variable. Since the $\eps \rightarrow 0$ limit is concerned, we can regard $u=(u_0,u_1)$ as a point in $\RR^2$. By the change of variable $\mu = \eps k$ and Riemann sum approximation, we see that $\sqrt{\eps}(P_{\eps}^{\theta})'(0,u)$ is approximated by
        \begin{equation}\label{P_epslim}
            -\int_{\RR}\int_{-\infty}^{0}(2\pi i\mu) e^{\qQ(2\pi\mu)s}\hat{\theta}(s-u_0,\mu)e^{2\pi i \mu u_1} \,ds \, d\mu = (\bar{P}^{\theta})'(u)
        \end{equation}
    as $\eps \rightarrow 0$. By the inequality $|e^{ix}-ix-1|\lesssim |x|^2$ for $x\in \RR$ and Lemma~\ref{lem:P_eps}, we have
        \begin{equation*}
            \left|e^{i t \sqrt{\eps}(P_{\eps}^{\theta})'(0,u)}-i t \sqrt{\eps}(P_{\eps}^{\theta})'(0,u)-1\right|\lesssim t^2\left|\sqrt{\eps}(P_{\eps}^{\theta})'(0,u)\right|^2\lesssim\frac{t^2}{(1+|u|^2)^2}.
        \end{equation*}
    Then by Dominated Convergence Theorem and \eqref{P_epslim}, we have
        \begin{equation*}
             \mathbb E e^{i t \sqrt{\eps}\Psi_\eps(0)}\rightarrow\exp\left(\int_{\RR^2}\left(e^{i t (\bar{P}^{\theta})'(u)  }-i t (\bar{P}^{\theta})'(u)  -1\right) du \right) = \mathbb E e^{i t I_1((\bar{P}^{\theta})')}
        \end{equation*}
    as $\eps \rightarrow 0$. This proves the desired result.
\end{proof}

Recall from Section~\ref{sec:regularisation_nonlinearity} that $F_\zeta$ is the regularisation of $F$ by the rescaled mollifier $\rho_\zeta$. Define
\begin{equation} \label{e:a_eps^zeta}
    a_{\eps}^{(\zeta)}\coloneqq\frac12 \mathbb EF_\zeta''(\sqrt{\eps}\Psi_\eps)\;.
\end{equation}
We have the following corollary on the difference $a_\eps^{(\zeta)}-a_\eps$.

\begin{cor}
    We have
    \begin{equation} \label{e:aepszeta-aeps}
        |a_\eps^{(\zeta)} - a_\eps| \lesssim \zeta^\beta
    \end{equation}
    uniformly over $\eps, \zeta \in (0,1)$, where $\beta$ is the H\"older exponent in Assumption~\ref{as:F}. 
\end{cor}
\begin{proof}
    By \eqref{e:F-F_zeta}, we have
    \begin{equation*}
        |a_\eps^{(\zeta)} - a_\eps| = \frac{1}{2}|\EE(F_\zeta - F)''(\sqrt{\eps}\Psi_\eps)| \lesssim \zeta^\beta \EE(1+|\sqrt{\eps}\Psi_\eps|)^M \lesssim \zeta^\beta,
    \end{equation*}
    where the last inequality follows from Lemma~\ref{Lpfinite}.
\end{proof}

\subsection{Graphic representation and other notations}
\label{sec:notations_stochastic}

For a finite set $\pP$ of real numbers (at least $1$), recall the definition of the norm $\|\cdot\|_{\Sigma_{\vec{u}}^{\vec{p}}}$ from Definition~\ref{def:l_u^p}. The number of variables concerned is usually clear from the vector $\vec{u}$ in the actual context, and that the order of integration is $u_1, u_2, \dots$ with increasing subscripts. Most of the times when Lemma~\ref{lem:Lvecpcontrol} is applied to control such a norm, the order does not matter. 

For fixed $p>2$, recall from Remark~\ref{rmk:norm_simplification} that we write
\begin{equation} \label{e:norm_simplification}
    \|\cdot\|_{\Sigma_{\vec{u}}^{p}} := \|\cdot\|_{\Sigma_{\vec{u}}^{\{2,p\}}}
\end{equation}
for simplicity. If the exponent $p$ is the same as the one in the $L_\omega^p$-norm of the stochastic object in concern and if no confusion should arise, we will then omit this $p$ and simply write $\Sigma_{\vec{u}}$ for $\Sigma_{\vec{u}}^p$. This is actually most of the cases below. Occasionally we will have different exponent than the one in the original $L_\omega^p$-norm, we will then specify this exponent as in \eqref{e:norm_simplification}.

For $d\in\mathbb{N}$ and $\mathfrak F \in L^2_\eta$ with the Wiener-It\^{o} chaos expansion \eqref{eq:WienerItoChaos}, we define the operators which truncate the chaos expansion of $\mathfrak F$ by
\begin{equation*}
	\mathcal{T}^{(\ge d)} \mathfrak F = \sum_{n=d}^\infty I_n(f_n),\quad \mathcal{T}^{(\leq d)} \mathfrak F = \sum_{n=0}^d I_n(f_n),\quad \mathcal{T}^{(d)} \mathfrak F = I_d(f_d).
\end{equation*}
With this operator, we can define the following objects:
\begin{equation} \label{e:chaos_truncate}
	\begin{split}
	&\<2'->\;_\eps\coloneqq \tT^{(\ge 3)}\<2'>_\eps; \qquad \<1'->\;_\eps\coloneqq \tT^{(\ge 2)} \<1'>_\eps;\qquad
	\<0'->\;_\eps\coloneqq \tT^{(\ge 1)}\<0'>_\eps;\\
    &\<2'zeta->_\eps\coloneqq\tT^{(\ge 3)}\<2'>_\eps^{(\zeta)};\qquad \<1'zeta-> _\eps\coloneqq\tT^{(\ge 2)}\<1'>_\eps^{(\zeta)}.
	\end{split}
\end{equation}
Here, $\tau^{(\zeta)}$ has the same formulation as $\tau$ except that replacing the appearance of $F$ or $F'$ by $F_\zeta$ and $F'_\zeta$ respectively. To facilitate understanding, following \cite{HQ}, we represent various quantities that appear in our calculations by graphs. In these graphs, each vertex corresponds to a space-time point (either in $\RR \times \TT$ or $\RR \times \TT_\eps$), and each edge represents a kernel function. 

Recall from the introduction that we use $x,y,z$ to denote spacetime points in $\RR \times \TT$, and $u,v, u_1, u_2, \dots$ to denote points in $\RR \times \TT_\eps$. We use \begin{tikzpicture} \draw[symbols,kernel2] (0,0) -- (1,0); \draw (0,0) circle(0.06); \draw (1,0) circle(0.06); \draw (-0.2,0) node {}; \draw (1.2,0) node{};  \end{tikzpicture} to denote either $K_\eps'$ or $(P_\eps^\theta)'$, depending on the vertices at the two ends. The arrow \begin{tikzpicture} \draw (0,0) circle(0.06); \draw (1,0) circle(0.06); \draw[symbols,kernel1] (0,0) -- (1,0); \draw (-0.2,0) node {}; \draw (1.2,0) node{}; \end{tikzpicture} (with a thick ``pointing" triangle) represents the difference between the heat kernel $K_\eps'$ evaluated at two different points. We use \begin{tikzpicture} \draw (0,0) node[root]{}; \end{tikzpicture} to represent the origin (of $\RR \times \TT$), and \begin{tikzpicture} \draw (0,0) node[root]{}; \draw[symbols,testfcn] (1,0) -- (0,0); \draw (1,0) circle(0.06); \draw (1.2,0) node {\scriptsize $x$}; \end{tikzpicture} to represent the rescaled test function $\varphi^\lambda$ centered at the origin. We list these representations in the table below. 
\begin{equation*} 
	\renewcommand\arraystretch{2}
	\begin{tabular}{c|c|c|c|c}
		\hline
		type: & \begin{tikzpicture} \draw[symbols,kernel2] (0,0) -- (1,0); \draw (0,0) circle(0.06); \draw (1,0) circle(0.06); \draw (-0.2,0) node {\scriptsize $z$}; \draw (1.2,0) node{\scriptsize $y$};  \end{tikzpicture} & \begin{tikzpicture} \draw[symbols,kernel2] (0,0) -- (1,0); \draw (0,0) circle(0.06); \draw (1,0) circle(0.06); \draw (-0.2,0) node {\scriptsize $u$}; \draw (1.2,0) node{\scriptsize $x$};  \end{tikzpicture}& \begin{tikzpicture} \draw (0,0) node[root]{}; \draw[symbols,testfcn] (1,0) -- (0,0); \draw (1,0) circle(0.06); \draw (1.2,0) node {\scriptsize $x$}; \end{tikzpicture} & \begin{tikzpicture} \draw (0,0) circle(0.06); \draw (1,0) circle(0.06); \draw[symbols,kernel1] (0,0) -- (1,0); \draw (-0.2,0) node {\scriptsize $y$}; \draw (1.2,0) node{\scriptsize $x$}; \end{tikzpicture}\\
		\hline
		kernel: & $K'_\eps(y-z)$& $(P^{\theta}_\eps)'(x,u)$& $\varphi^\lambda(x)$ & $K'_\eps(x-y) - K'_\eps(-y)$\\
		\hline
	\end{tabular}
\end{equation*}
When two (or more) edges join together, it represents multiplication. All the other vertices represent dumb integration variables unless indicated. Furthermore, all vertices where the arrows point to are points in $\RR \times \TT$ (and $u,u_i \in \RR \times \TT_\eps$ only appears at the other side of the arrow and always made explicit), so there is no ambiguity in the above graphic notations. If the graph has the superscript $\cdot^{(\zeta)}$, then all the appearances of the noises $\<2'>,\<1'>,\<0'>$ in this graph are replaced by $\<2'>^{(\zeta)},\<1'>^{(\zeta)}$ and $\<0'>^{(\zeta)}$ respectively. Furthermore, since our aim in Theorem~\ref{thm:convergence} is to compare two $\eps$-dependent models for the same $\eps$, all the noise nodes in the graphs are with the $\eps$. Hence, since no confusion can arise in this situation, we omit the notation $\eps$ in the graphic representation in the computations (except that we still use $\tau_\eps$ or $\tau_\eps^{\zeta}$ for the precise symbols in Table \eqref{e:noise}). The following three examples illustrate the use of the notations:
\begin{equation} \label{e:graph_examples}
    \begin{split}
    \<1'uy> &= \int_{\RR \times \TT}  K_\eps'(y-z) \, (P_\eps^\theta)'(z,u) \, [\<1'>]_\eps(z) \, dz\;,\\
    \<2',-1',1'_1> &= \int_{\RR \times \TT} \varphi^\lambda(x) \, \big( K_\eps'(x-y) - K_\eps'(-y) \big) \,  [\<1'>]_\eps(x) \, dx\;,\\
    \<D2',1',1'_3>^{(\zeta)} &= \iint\limits_{(\RR \times \TT)^2} \big( K_\eps'(x-y) - K_\eps'(-y) \big) \, K_\eps'(y-z) \, [\<1'>^{(\zeta)}]_\eps(y) \, D_u [\<2'>^{(\zeta)}]_\eps(z)  \, dz \, dy\;.
    \end{split}
\end{equation}
The first one is a function of $y \in \RR \times \TT$ and $u \in \RR \times \TT_\eps$ since the dummy variable $z$ associated with the noise node $[\<1'>]$ is integrated out. The second one is a function of $y \in \RR \times \TT$ since the dummy variable $x$ associated with $[\<1'>]$ is integrated out. Finally, the last one is a function of $x \in \RR \times \TT$ and $u \in \RR \times \TT_\eps$, since $D_u [\<2'>]_\eps(z)$ is a function of $u$ and $z$ and the dummy variables $z$ and $y$ are integrated out. 

If the graph does not contain any noise node ($\<2'>$, $\<1'>$, $\<0'>$ or variants of them) and has $\|\cdot\|_{\Sigma_u}$ norm (or its variants) with it, then all the kernels in the graph should be understood as their absolute values. For example, we have
\begin{equation*}
    \left\|\<DDD2'> \right\|_{\Sigma_{u_1,u_2}} = \left\|  \int_{\RR \times \TT} |(P_\eps^\theta)'(z,u)| \, \Big( \prod_{i=1}^{2} |(P_\eps^\theta)'(z,u_i)| \Big) \, |K_\eps'(y-z)| dz    \right\|_{\Sigma_{u_1,u_2}}\;,
\end{equation*}
which is a function of $y \in \RR \times \TT$ and $u \in \RR \times \TT_\eps$, and the dummy variable $z$ where all the edges join is integrated out. Here, as explained above, $\Sigma_{u_1, u_2}$-norm denotes $\Sigma_{u_1, u_2}^p = \Sigma_{u_1, u_2}^{\{2,p\}}$ for some fixed $p$ in the corresponding context. 

\begin{remark} \label{rmk:graph_exception}
    There is one exception to the above rules: when the graph contains a component $\<2',1'y>$, with an abuse of notation, it means that the mean is subtracted, and this extends to situations when the nodes $\<2'>$ or $\<1'>$ are replaced by high-low chaos components or $\zeta$-regularised version. For example, we have (noting that the operation $\tT^{(\geq 1)}$ is the same as subtracting the mean)
\begin{equation} \label{e:graph_exception}
   \begin{split}
   \<2',1'y>(y) &\coloneqq \int_{\RR \times \TT} K_\eps'(y-z) \, \tT^{(\geq 1)} \big( \<1'>_\eps (y) \, \<2'>_\eps (z) \big) \, dz\;,\\
   \<2',1',1'x>(x) &:= \iint\limits_{(\RR \times \TT)^2} \big( K_\eps'(x-y) - K_\eps'(-y) \big) \, K_\eps'(y-z)\\
   &\phantom{1111111}[\<1'>]_\eps(x) \, \tT^{(\geq 1)} \big( [\<1'>]_\eps(y) \, [\<2'>]_\eps (z) \big) \, dz \, dy\;.
   \end{split}
\end{equation}
Note that the second graphic notation above only subtracts the mean of $[\<1'>]_\eps(y) \, [\<2'>]_\eps (z)$, and the expression for the corresponding object $\<2'1'1'>$ needs to further subtract the mean of the whole expression. 

But if such a component is with $D_u \<2'>$, $D_u \<1'>$ or their variants, then it still stays with its original meaning without the mean being subtracted. For example, we have
\begin{equation*}
    \<D2',D1',1'_5> = \iint\limits_{(\RR \times \TT)^2} \big( K_\eps'(x-y) - K_\eps'(-y) \big) \, K_\eps'(y-z) \, D_u [\<1'>]_\eps(y) D_u [\<2'>]_\eps(z) \, dz \, dy\;,
\end{equation*}
or see the last one in \eqref{e:graph_examples}. 
\end{remark}

\subsection{Convergence of the first order process to the free field} \label{sec:1}
    Recall $\<1>_\eps = P_0'\ast\xi_\eps$ is the building block of $\hPi^{\HS(\eps)}$. Our aim is to prove the following proposition.
\begin{prop}\label{prop:<1>_convergence}
    For every $p\geq 2$ and $\delta\in(0,\frac{1}{8})$\;, there exists $\delta'>0$ such that
	   \begin{equation*}
			\sup_{\varphi\in \bar{C}_c^1} \|\scal{\<1'>_\eps-\<1>_\eps, \varphi^\lambda}\|_{L^p_\omega} \lesssim _p\eps^{\delta'} \lambda^{-\frac{1}{2}-\delta}\;,
	   \end{equation*}
    where the proportionality constant is independent of $\eps, \lambda \in (0,1)$.
\end{prop}
\begin{proof}
    We split the object $\<1'>_\eps$ into a main part with regularised nonlinearity $\<1'>_\eps^{(\zeta)}$, and an error part $\<1'>_\eps - \<1'>_\eps^{(\zeta)}$. 
    The desired bounds of these two parts are given respectively in Lemmas~\ref{lem:<1>_convergence_1} and~\ref{lem:<1>_convergence_2}. Choosing $\zeta = \eps^{\frac{\delta}{2}}$ and $\nu=0$ in the statement of these two lemmas completes the proof of the proposition. 
\end{proof}

We now give the desired bounds on $\<1'>_\eps^{(\zeta)}$ and $\<1'>_\eps - \<1'>_\eps^{(\zeta)}$, starting with the first one. 

\begin{lem} \label{lem:<1>_convergence_1}
For every $p\geq 2$ and $\delta\in(0,\frac{1}{8})$\;, we have
    \begin{equation*}
        \sup_{\varphi\in \bar{C}_c^1} \|\scal{\<1'>_\eps^{(\zeta)}-\<1>_\eps, \varphi^\lambda}\|_{L^p_\omega} \lesssim_p \big( \eps^{\delta} \zeta^{-1}+\zeta^\beta \big) \, \lambda^{-\frac{1}{2}-\delta}
    \end{equation*}
uniformly over $\eps, \lambda, \zeta \in (0,1)$. 
\end{lem}
\begin{proof}
    We decompose $\<1'>_\eps^{(\zeta)}$ into
    \begin{equation*}
        \<1'>_\eps^{(\zeta)} = \<1'zeta->_\eps + \frac{a_{\eps}^{(\zeta)}}{a_{\eps}}\Psi_\eps + \Er_\eps^{(\zeta)}\;,
    \end{equation*}
    where $a_{\eps}^{(\zeta)}$ is given by \eqref{e:a_eps^zeta}, and the error term $\Er_\eps^{(\zeta)}$ has expression
    \begin{equation} \label{e:1_error}
        \Er_\eps^{(\zeta)}(x) \coloneqq \frac{1}{2a_\eps\sqrt{\eps}} I_1\big(\mathbb E D_{\bullet} F_\zeta'(\sqrt{\eps}\Psi_\eps(x))\big)-\frac{a_{\eps}^{(\zeta)}}{a_{\eps}}\Psi_\eps(x)\;.
    \end{equation}
    We expect the term $\frac{a_{\eps}^{(\zeta)}}{a_{\eps}}\Psi_\eps$ to be close to $\<1>_\eps$, while the other two vanishing to $0$ in the $\eps \rightarrow 0$ limit. Indeed, we will show the bounds
    \begin{equation} \label{e:<1>_decomposition}
        \begin{split}
            &\|\scal{\;\<1'zeta-> _\eps, \varphi^\lambda}\|_{L^p_\omega} \lesssim_p  \zeta^{-1}\eps^{\delta}\lambda^{-\frac{1}{2}-\delta}\;,\\
            &\|\scal{\Er_\eps^{(\zeta)}, \varphi^\lambda}\|_{L^p_\omega}\lesssim_p  \zeta^{-1}\eps^{\delta}\lambda^{-\frac{1}{2}-\delta},\\
            &\Big\|\Bscal{\frac{a_{\eps}^{(\zeta)}}{a_{\eps}}\Psi_\eps-\<1>_\eps, \varphi^\lambda}\Big\|_{L^p_\omega}\lesssim_p  \zeta^\beta\lambda^{-\frac{1}{2}}+\eps^{\delta}\lambda^{-\frac{1}{2}-\delta}.
        \end{split}
        \end{equation}
    Since $F_\zeta'$ is an odd function, by \eqref{e:symmetry_law}, we have $\mathbb EF_\zeta'(\sqrt{\eps}\Psi_\eps)=0$. We now proceed to proving the bounds \eqref{e:<1>_decomposition}, starting with $\<1'zeta-> _\eps$. 
        
    Using \eqref{e:sg} twice, we have
        \begin{equation} \label{e:<1>_two_derivative}
            \|\scal{\;\<1'zeta-> _\eps,\varphi^\lambda} \|_{L^p_{\omega}} \lesssim_p  \left\| \<DD1'zeta_>_{\;\eps} \right\|_{\Sigma_{\vec{u}}L^p_\omega},
        \end{equation}
    where $\vec{u}=(u_1,u_2)\in(\RR\times\TT_\eps)^2$ and we recall from Remark~\ref{rmk:norm_simplification} and Section~\ref{sec:notations_stochastic} that we write $\|\cdot\|_{\Sigma_{\vec{u}}}$ for $\|\cdot\|_{\Sigma_{\vec{u}}^{\{2,p\}}}$ here. By \eqref{e:F_zeta} with $n=3$ and Lemma~\ref{Lpfinite}, we have
    \begin{equation} \label{e:F'''_bound}
        \|F_\zeta^{(3)} (\sqrt{\eps}\Psi_\eps + r)\|_{L^p_\omega} \lesssim_p  \zeta^{-1}
    \end{equation}
    uniformly in $\zeta, \eps \in (0,1)$ and $|r|\lesssim1$. According to the definition of $[\<1'>^{(\zeta)}]_\eps$ and Proposition~\ref{DxFWienerItoChaos}, we have $D^2_{\vec{u}}\<1'zeta-> _\eps=D^2_{\vec{u}} \<1'>^{(\zeta)}_\eps$. Therefore, by \eqref{e:discrete_deri} we obtain
        \begin{equation} \label{e:bound_D^2_u<1>}
            \begin{aligned}
                \|D^2_{\vec{u}}\<1'zeta->_\eps(x) \|_{L^p_\omega} &\lesssim \eps^{-\frac12} \bigg| \iint\limits_{\prod_{i=1}^{2} \big[0, \sqrt{\eps}(P_\eps^\theta)'(x,u_i) \big]}  \|F_\zeta^{(3)}(\sqrt{\eps}\Psi_\eps(x)+r_1+r_2)\|_{L^p_\omega}dr_1dr_2 \bigg|\\
                &\lesssim_p  \frac{\sqrt{\eps}}{\zeta} \cdot \prod_{i=1}^{2} |(P_\eps^\theta)'(x,u_i)|,
            \end{aligned}
        \end{equation}
    where the last inequality follows from \eqref{e:F'''_bound} and Lemma~\ref{lem:P_eps} (that $\sqrt{\eps}|(P_{\eps}^{\theta})'|\lesssim 1$). Substituting it into \eqref{e:<1>_two_derivative} and applying Lemma~\ref{lem:P_eps} to control $|(P_\eps^\theta)'|$ again, we get
        \begin{equation*}\label{Du1u2Lp}
            \begin{aligned}
                \|\scal{\; \<1'zeta-> _\eps,\varphi^\lambda} \|_{L^p_{\omega}} \lesssim_p  \frac{\sqrt{\eps}}{\zeta} \left\| \<-1'_>\right\|_{\Sigma_{\vec{u}}} \lesssim \frac{\sqrt{\eps}}{\zeta} \bigg\|\int_{\RR\times\TT}|\varphi^{\lambda}(x)| \prod_{i=1}^2 \frac{\eps^{\frac32}}{(|x-\eps u_i|+\eps)^2} dx \bigg\|_{\Sigma_{\vec{u}}}.
            \end{aligned}  
        \end{equation*}
    Applying Lemma~\ref{lem:Lvecpcontrol} with $k=2$ and $\alpha_1 = \alpha_2 = 2$, we get 
        \begin{equation*}
            \begin{aligned}
                \|\scal{\; \<1'zeta-> _\eps,\varphi^\lambda} \|_{L^p_{\omega}}\lesssim_p  \zeta^{-1}\bigg( \iint\limits_{(\RR\times\TT)^2} \frac{\eps |\varphi^\lambda(x) \varphi^\lambda(x')|}{(|x-x'|+\eps)^{2}}  dx dx'\bigg)^{\frac{1}{2}} \lesssim \zeta^{-1}\eps^{\delta} \lambda^{-\frac{1}{2}-\delta},   
            \end{aligned}
        \end{equation*}
    which is the desired bound for $ \<1'zeta-> _\eps$.

    For the error term $\Er_\eps^{(\zeta)}$, since it has mean $0$, we apply the spectral gap inequality \eqref{e:sg} to get
    \begin{equation} \label{e:<1>_error_first}
        \| \scal{\Er_\eps^{(\zeta)}, \varphi^\lambda} \|_{L_\omega^p} \lesssim_p  \Big\| \int_{\RR \times \TT} |\varphi^\lambda(x)| \cdot \|D_u \Er_\eps^{(\zeta)}(x)\|_{L_\omega^p} \, dx \Big\|_{\Sigma_u}\;.
    \end{equation}
    The Malliavin derivative $D_u \Er_\eps^{(\zeta)}(x)$ has the explicit expression
    \begin{equation*}
        D_u \Er_\eps^{(\zeta)}(x) = \frac{1}{2 a_\eps \sqrt{\eps}} \int_{0}^{\sqrt{\eps} (P_\eps^\theta)'(x,u)} \EE \Big( F_\zeta''\big(\sqrt{\eps} \Psi_\eps(x) + r \big) - F_\zeta''\big(\sqrt{\eps} \Psi_\eps(x) \big) \Big) \, dr\;,
    \end{equation*}
    which gives the pointwise moment bound
    \begin{equation} \label{e:<1>_error}
        \|D_u \Er_\eps^{(\zeta)}(x)\|_{L^p_\omega} \lesssim \frac{\sqrt{\eps}}{\zeta} \cdot (P_{\eps}^{\theta})'(x,u)^2\;.
    \end{equation}
    Substituting it back into \eqref{e:<1>_error_first} and using Lemma~\ref{lem:P_eps}, we get
    \begin{equation*}
        \| \scal{\Er_\eps^{(\zeta)}, \varphi^\lambda} \|_{L_\omega^p} \lesssim_p  \frac{\sqrt{\eps}}{\zeta} \Big\| \int_{\RR \times \TT} |\varphi^\lambda(x)| \cdot \frac{\eps^{\frac{3}{2}}}{\big( |x-\eps u| + \eps \big)^{\frac{5}{2}}} \, dx \Big\|_{\Sigma_u}\;.
    \end{equation*}
    Applying Lemma~\ref{lem:Lvecpcontrol} with $k=1$ and $\alpha = \frac{5}{2}$, we obtain the desired bound for $\Er_\eps^{(\zeta)}$. 
    
    For the remaining part $\frac{a_{\eps}^{(\zeta)}}{a_{\eps}}\Psi_\eps-\<1>_\eps$, we split it into $\Psi_\eps-\<1>_\eps$ and $\frac{a_{\eps}^{(\zeta)}-a_{\eps}}{a_{\eps}}\Psi_\eps$. Similar as above but using the second inequality in Lemma~\ref{lem:P_eps}, we bound the term $\Psi_\eps-\<1>_\eps$ by
        \begin{equation*}
            \|\scal{\Psi_\eps-\<1>_\eps, \varphi^\lambda} \|_{L^p_\omega} \lesssim_p   \bigg\| \int_{\RR\times\TT} \varphi^\lambda(x) \, \big( (P_\eps^\theta)'(x,u) - (P_{0,\eps}^{\theta})'(x,u) \big) \, dx \bigg\|_{\Sigma_u} \lesssim \eps^{\delta} \lambda^{-\frac{1}{2}-\delta}.
        \end{equation*}
    The estimate for the term $\frac{a_{\eps}^{(\zeta)}-a_{\eps}}{a_{\eps}}\Psi_\eps$ follows from the bound $\|\scal{\Psi_\eps, \varphi^\lambda} \|_{L^p_\omega}\lesssim_p   \lambda^{-\frac12}$ and \eqref{e:aepszeta-aeps}. 

    This gives the desired bound for $\frac{a_{\eps}^{(\zeta)}}{a_{\eps}}\Psi_\eps-\<1>_\eps$ and completes the proof of the lemma. 
\end{proof}

    Now we focus on the remainder $\<1'>_\eps - \<1'>_\eps^{(\zeta)}$. 
\begin{lem}
    \label{lem:<1>_convergence_2}
    For every $p\geq 2$ and $\nu\in[0,\frac{1}{2})$\;, we have
		\begin{equation*}
			\sup_{\varphi\in \bar{C}_c^1} \|\scal{\<1'>_\eps - \<1'>_\eps^{(\zeta)}, \varphi^\lambda}\|_{L^p_\omega} \lesssim _p \zeta^{\beta}\eps^{-\nu} \lambda^{-\frac{1}{2} + \nu},
		\end{equation*}
    where the proportionality constant is independent of $\eps, \lambda, \zeta\in(0,1)$.
\end{lem}
\begin{proof}
    Since $\EE \<1'>_\eps = \EE \<1'>_\eps^{(\zeta)} = 0$, applying \eqref{e:sg} and using triangle inequality to move $L_\omega^p$-norm inside the inner product, we get
        \begin{equation*}
            \|\scal{\<1'>_\eps - \<1'>_\eps^{(\zeta)}, \varphi^\lambda}\|_{L^p_\omega} \lesssim_p  \Big\| \scal{\|D_u \big( \<1'>_\eps(\cdot) - \<1'>_\eps^{(\zeta)}(\cdot) \big) \|_{L_\omega^p}, |\varphi^\lambda| } \Big\|_{\Sigma_u}.
        \end{equation*}
    The Malliavin derivative has the expression
    \begin{equation*}
        D_u \big( \<1'>_\eps(x) - \<1'>_\eps^{(\zeta)}(x) \big) = \frac{1}{2 a_\eps \sqrt{\eps}} \int_{0}^{\sqrt{\eps} (P_\eps^\theta)'(x,u)} \big(F'' - F_\zeta''\big) \big( \sqrt{\eps} \Psi_\eps(x) + r \big) dr\;.
    \end{equation*}
    Hence, by \eqref{e:F-F_zeta} and Lemma~\ref{Lpfinite} we have
    \begin{equation} \label{e:1-1_zeta}
        \|D_u \big( \<1'>_\eps(x) - \<1'>_\eps^{(\zeta)}(x) \big) \|_{L_\omega^p} \lesssim_p  \zeta^\beta |(P_{\eps}^{\theta})'(x,u)|\;.
    \end{equation}
    Combining it with Lemmas~\ref{lem:P_eps} and~\ref{lem:Lvecpcontrol}, we obtain        
        \begin{equation*}
            \|\scal{\<1'>_\eps - \<1'>_\eps^{(\zeta)}, \varphi^\lambda}\|_{L^p_\omega} \lesssim_p  \zeta^\beta \bigg( \iint_{(\RR\times\TT)^2}  \frac{|\varphi^\lambda(x)\varphi^\lambda(x')|}{|x-x'|+\eps}  dxdx' \bigg)^{\frac{1}{2}}  \lesssim \zeta^\beta \eps^{-\nu}\lambda^{-\frac12+\nu}\;.
        \end{equation*}
    This concludes the proof.
\end{proof}

\begin{rmk}
    The case $\nu=0$ is sufficient for the convergence of $\<1'>_\eps$. The above more general version will be required for the convergence of $\<2'1'1'>_\eps$ in Proposition~\ref{prop:211_convergence_2}.
\end{rmk}

\subsection{A remark on the ``Wick square" of the free field}
\label{sec:square_process}

With essentially the same techniques and procedure as Proposition~\ref{prop:<1>_convergence}, one can show that for arbitrarily small $\delta > 0$, there exists $\delta' > 0$ such that
\begin{equation} \label{e:<2>_bound}
    \|\scal{\<2'>_\eps - \<2>_\eps, \varphi^\lambda}\|_{L_\omega^p} \lesssim_p  \eps^{\delta'} \lambda^{-1-\delta}.
\end{equation}
Here, $\<2>_\eps \coloneqq (\<1>_\eps)^2 - \EE (\<1>_\eps^2) = \hPi^{\HS(\eps)} \<2'>$. We do not repeat the detailed arguments here, but remark one difference of this object as compared to \cite{HQ, KPZCLT} worthy of noting. As in the definition \eqref{e:rs_def_1}, the expectation of the object is subtracted, but it has a non-zero ``first chaos" component for every fixed $\eps\in(0,1)$, in contrast to the situation in \cite{KPZCLT} that its first chaos component is identically zero. 

We now briefly argue that its first chaos component $\tT^{(1)} \<2'>_\eps$ vanishes in the right topology as $\eps \rightarrow 0$. Note that
\begin{equation*}
    D_u \tT^{(1)} \<2'>_\eps = D_u I_1 \big( \EE D_{\bullet} \<2'>_\eps \big) = \EE D_u \<2'>_\eps\;,
\end{equation*}
and
\begin{equation*}
    \big| \EE D_u \<2'>_\eps(x) \big| =  \frac{1}{a_\eps \eps} \left| \int_{0}^{\sqrt{\eps} (P_\eps^\theta)'(x,u)} \EE F'\big( \sqrt{\eps} \Psi_\eps(x) + r \big) \, dr \right| \lesssim |(P_\eps^\theta)'(x,u)|^2
\end{equation*}
uniformly in $\eps$ and $x, u$, where the last inequality follows from the Taylor expansion of $F'$ near $\sqrt{\eps} \Psi_\eps(x)$ and that $\EE F'\big( \sqrt{\eps} \Psi_\eps(x) \big) = 0$. Hence, by spectral gap inequality \eqref{e:sg}, we have
\begin{equation*}
    \|\scal{\tT^{(1)}\<2'>_\eps, \varphi^\lambda} \|_{L_\omega^p} \lesssim_p \|\scal{D_u \tT^{(1)} \<2'>_\eps, \varphi^\lambda }\|_{\Sigma_u L_\omega^p} \lesssim \left\| \int_{\RR \times \TT} |(P_\eps^\theta)'(x,u)|^2 |\varphi^\lambda (x)| \, dx  \right\|_{\Sigma_u}
\end{equation*}
Note that
\begin{equation*}
    |(P_\eps^\theta)'(x,u)|^2 \lesssim \frac{\eps^3}{(|x-\eps u| + \eps)^4} \lesssim \eps^\delta \cdot \frac{\eps^{\frac{3}{2}}}{(|x-\eps u|+\eps)^{\frac{5}{2} + \delta}}\;,
\end{equation*}
the desired bound \eqref{e:<2>_bound} then follows from applying Lemma~\ref{lem:Lvecpcontrol} with $k=1$ and $\alpha = \frac{5}{2} + \delta$.

\subsection{Convergence of the third order process} \label{sec:211}

As listed in \eqref{e:symbols}, there are two third order processes $ \<2'2'0'>_\eps$ and $\<2'1'1'>_\eps$. Since the proof for the convergence of $\<2'1'1'>_\eps$ is more complicated, we only demonstrate that $\<2'1'1'>_\eps-\hPi^{\HS(\eps)} \<2'1'1'>$ converges to $0$ in $\cC^{-\kappa}$ as $\eps\rightarrow0$ in this section. The aim of this section is to prove the following proposition. By Kolmogorov type criterion, it is sufficient to establish this convergence. 
    
\begin{prop} \label{prop:211_convergence}
    For every $p\geq2$ and $\delta\in(0,\frac{1}{8})$\;, there exists $\delta'>0$ such that
        \begin{equation*}
            \sup_{\varphi\in \bar{C}_c^1} \Big\|\scal{\;\<2'1'1'>_\eps-\hPi^{\HS(\eps)} \<2'1'1'>, \varphi^\lambda}\Big\|_{L^p_\omega} \lesssim_p \eps^{\delta'} \lambda^{-\delta}\;,
        \end{equation*}
    where $\hPi^{\HS(\eps)} \<2'1'1'>$ is the stochastic object in \cite{KPZCLT}, and the proportionality constant is independent of $\eps, \lambda \in (0,1)$. 
\end{prop}
\begin{proof}
As in the previous section, we again decompose the object into a main regularised part and a small error part as
\begin{equation} \label{e:211_decomposition}
    \<2'1'1'>_\eps = \<2'1'1'>_\eps^{(\zeta)} + \Big( \; \<2'1'1'>_\eps - \<2'1'1'>_\eps^{(\zeta)} \; \Big)\;.
\end{equation}
Here, $\<2'1'1'>_\eps^{(\zeta)}$ means that \textit{each appearance} of $F$ or its derivative is replaced by $F_\zeta$ and $F_\zeta'$ respectively. The desired bounds for the two parts in the decomposition \eqref{e:211_decomposition} are given in Propositions~\ref{prop:211_convergence_1} and~\ref{prop:211_convergence_2} below. Choosing $\zeta$ to be a sufficiently small power of $\eps$ completes the proof of the proposition. 
\end{proof}

We now state the relevant bounds on the two parts in the decomposition \eqref{e:211_decomposition}. The main part $\<2'1'1'>_\eps^{(\zeta)}$ satisfies the following bound. 

\begin{prop} \label{prop:211_convergence_1}
    For every $p\geq2$ and $\delta\in(0,\frac{1}{8})$\;, there exists $\delta'>0$ such that
        \begin{equation*}
            \sup_{\varphi \in \bar{C}_c^1} \Big\|\scal{\;\<2'1'1'>_\eps^{(\zeta)}-\hPi^{\HS(\eps)} \<2'1'1'>, \varphi^\lambda}\Big\|_{L^p_\omega} \lesssim_p \big( \eps^{\delta'} \zeta^{-2}+\zeta^{\beta} \big) \, \lambda^{-\delta},
        \end{equation*}
    where the proportionality constant is independent of $\eps,\zeta, \lambda \in (0,1)$.
\end{prop}

The following proposition demonstrates that the error part indeed vanishes in the limit. 

\begin{prop} \label{prop:211_convergence_2}
    For every $p\geq2$ and sufficiently small $\delta>0$, there exists $\delta'>0$ such that
    \begin{equation*}
        \sup_{\varphi\in \bar{C}_c^1} \Big\|\scal{\;\<2'1'1'>_\eps - \<2'1'1'>_\eps^{(\zeta)}, \varphi^\lambda}\Big\|_{L^p_\omega} \lesssim_p \eps^{\delta'} \lambda^{-\delta} + \zeta^\beta \, |\log \eps|\;,
    \end{equation*}
    where the proportionality constant is independent of $\eps, \lambda, \zeta \in (0,1)$.
\end{prop}

The rest of this section is devoted to the proof of the above two bounds. Proposition~\ref{prop:211_convergence_1} is much harder than Proposition~\ref{prop:211_convergence_2} since in the latter, one allows a small negative power in $\eps$ which can then be balanced out by a proper choice of $\zeta$, while there is no such smallness to play with in the main term $\<2'1'1'>_\eps^{(\zeta)}$. Hence, in what follows, we will focus on the proof of Proposition~\ref{prop:211_convergence_1}, and briefly sketch that for Proposition~\ref{prop:211_convergence_2}.  

By the definitions \eqref{e:rs_def_1} and \eqref{e:rs_def_2}, we have the explicit expression
\begin{equation*}
    \begin{split}
    \<2'1'1'>_\eps^{(\zeta)}(x) = \iint\limits_{(\RR \times \TT)^2} \big( &K_\eps'(x-y) - K_\eps'(-y)  \big) \, K_\eps'(y-z)\\
    &\tT^{(\geq 1)} \Big( \<1'>_\eps^{(\zeta)}(x) \, \tT^{(\geq 1)} \big( \<1'>_\eps^{(\zeta)}(y) \, \<2'>_\eps^{(\zeta)}(z) \big) \Big) dz dy\;.
    \end{split}
\end{equation*}
We split each noise node into a lower order chaos term and a higher order chaos term (with the notation in \eqref{e:chaos_truncate}) by
\begin{equation*}
    \<2'>_\eps^{(\zeta)} = \tT^{(\leq 2)} \<2'>_\eps^{(\zeta)} +  \<2'zeta->_\eps\;, \qquad \<1'>_\eps^{(\zeta)} = \tT^{(1)} \<1'>_\eps^{(\zeta)} + \<1'zeta-> _\eps\;.
\end{equation*}
Since the above expression involves a product of three noise terms, this decomposition gives a sum of eight terms in total, each containing a product of three new noises as either low or high chaos components as the original ones. 

As one may expect, the term with all three noise nodes with truncated chaos components should be close to $\hPi^{\HS(\eps)} \<2'1'1'>$, while all the other seven terms should vanish as $\eps \rightarrow 0$. The harder ones to bound are those with more higher order chaos, even if they should vanish in the limit. This corresponds to the difficulty with a non-polynomial $F$. Hence, in what follows, we will give details for two of them, namely the terms from
\begin{equation*}
    \<1'zeta->_\eps(x) \, \<1'zeta-> _\eps(y) \, \<2'zeta->_\eps(z)\;, \quad \text{and} \quad \tT^{( 1)}\big( \<1'>_\eps^{(\zeta)}(x) \big) \, \<1'zeta-> _\eps(y) \, \<2'zeta->_\eps(z)\;,
\end{equation*}
in Propositions~\ref{prop:211_convergence_3} and~\ref{prop:211_convergence_4} respectively. We also briefly sketch in Proposition~\ref{prop:211_convergence_5} the convergence of the term with all three contributions from lower order chaos components. The bounds for the other five terms are easier than the first two and can be obtained in simpler ways. This then proves Proposition~\ref{prop:211_convergence_1}. 

The rest of this section is organised as follows. In Sections~\ref{sec:211_correlation} and~\ref{sec:211_preliminaries}, we give preliminary bounds on multi-point correlations as well as various components of the object $\<2'1'1'>_\eps^{(\zeta)}$, which are then combined together in Section~\ref{sec:211_main} to prove Proposition~\ref{prop:211_convergence_1}. In Section~\ref{sec:211_error}, we give a sketch on the desired bounds on the error term $\<2'1'1'>_\eps - \<2'1'1'>_\eps^{(\zeta)}$ which, combined with Proposition~\ref{prop:211_convergence_1}, completes the proof of Proposition~\ref{prop:211_convergence}.

\subsubsection{Bounds on some multi-point correlation functions}
\label{sec:211_correlation}

We give bounds on two correlation functions that are needed in the proof of the main convergence theorem. The following lemma will be used in the sequel. 

\begin{lem} \label{lem:Holder}
Let $X, Y \in L^2_\eta$\,. Suppose $\{\lambda_k\}_{k \geq 0}$ is a sequence of real numbers such that $\lambda_k^2 \geq k!$, then we have the bound
\begin{equation*}
    \sum_{k \geq 0} \frac{1}{\lambda_k^2} \big| \bscal{\EE D^{k}_{\bullet}X\,, \,\EE D^{k}_{\bullet} Y}_{L^2(\UU^k)} \big| \leq \|X\|_{L_\omega^2} \|Y\|_{L_\omega^2}\;.
\end{equation*}
\end{lem}
\begin{proof}
Applying H\"older inequality first to the inner product $\scal{\cdot, \cdot}$ and then to the sum $k \geq 0$ weighted by $\frac{1}{\lambda_k^2}$, we have
\begin{equation*}
    \begin{split}
    \sum_{k \geq 0} \frac{1}{\lambda_k^2} \big| \bscal{\EE D^{k}_{\bullet}X\,, \,\EE D^{k}_{\bullet} Y} \big| &\leq \sum_{k \geq 0} \frac{1}{\lambda_k^2} \|\EE D^k_{\bullet} X\|_{L^2} \cdot \|\EE D^k_{\bullet} Y\|_{L^2}\\
    &\leq \Big( \sum_{k \geq 0} \frac{1}{\lambda_k^2} \|\EE D^k_{\bullet} X\|_{L^2}^2 \Big)^{\frac{1}{2}} \cdot \Big( \sum_{k \geq 0} \frac{1}{\lambda_k^2} \|\EE D^k_{\bullet} Y\|_{L^2}^2 \Big)^{\frac{1}{2}}\;,
    \end{split}
\end{equation*}
where for each $k$, the inner product $\scal{\cdot, \cdot}$ and the norm $\|\cdot\|$ are both in $L^2(\UU^k)$. The conclusion then follows from the identity (by \eqref{e:chaos_expansion})
\begin{equation*}
    \EE |X|^2 = \sum_{k \geq 0} \frac{1}{k!} \|\EE D^k_{\bullet} X\|_{L^2}
\end{equation*}
and the assumption that $\lambda_k^2 \geq k!$. 
\end{proof}

We now give the bound on covariance of the field $\<1'zeta->_\eps$.

\begin{lem} \label{lem:covariance}
    For every $y,z\in \RR \times \TT$, we have
        \begin{equation*}
            \big| \, \EE\big[\;\<1'zeta->_\eps(y)\; \<1'zeta->_\eps(z) \big] \, \big| \lesssim \zeta^{-2}\frac{\eps}{(|y-z|+\eps)^2},
        \end{equation*}
    where the proportionality constant is independent of $y,z\in \RR \times \TT$ and $\eps,\zeta \in (0,1)$.
\end{lem}
\begin{proof}
By \eqref{eqn:ISO} and Proposition~\ref{prop:WienerItoChaos}, we have the expression
\begin{equation*}
    \begin{split}
    \EE \, \big( \<1'zeta->_\eps(y)\, \<1'zeta->_\eps(z) \big) &= \sum_{k \geq 2} \frac{1}{k!} \bscal{D^{k}_{\bullet} \<1'>_\eps^{(\zeta)}(y)\,, D^{k}_{\bullet} \<1'>_\eps^{(\zeta)}(z)}_{L^2_k}\\
    &= \sum_{k \geq 2} \frac{1}{k!} \iint\limits_{(\RR \times \TT_\eps)^2} \bscal{D^{k-2}_{\bullet} D^2_{\vec{u}} \<1'>_\eps^{(\zeta)}(y)\,, D^{k-2}_{\bullet} D^2_{\vec{u}} \<1'>_\eps^{(\zeta)}(z)}_{L^2_{k-2}} \, d \vec{u}\;,
    \end{split}
\end{equation*}
where $\vec{u}=(u_1,u_2)\in(\RR\times\TT_\eps)^2$ and we have abbreviated $L^2_k$ for $L^2((\RR \times \TT_\eps)^k)$, and replaced $\<1'zeta->_\eps$ by $\<1'>_\eps^{(\zeta)}$ from the first line since the sum is from the chaos components $k \geq 2$. 

By Lemma~\ref{lem:Holder}, we have the bound
\begin{equation*}
    \bigg| \sum_{k \geq 2} \frac{1}{k!} \bscal{D^{k-2}_{\bullet} D^2_{\vec{u}} \<1'>_\eps^{(\zeta)}(y)\,, D^{k-2}_{\bullet} D^2_{\vec{u}} \<1'>_\eps^{(\zeta)}(z)} \bigg| \leq \|D^2_{\vec{u}} \, \<1'>_\eps^{(\zeta)}(y) \|_{L_\omega^2} \cdot \|D^2_{\vec{u}} \, \<1'>_\eps^{(\zeta)}(z) \|_{L_\omega^2}\;,
\end{equation*}
where we omitted $L^2_{k-2}$ in the inner product for simplicity. 
Plugging \eqref{e:bound_D^2_u<1>} into the above bound for the integrand and the expression for the correlation, we obtain
\begin{equation*}
    \big| \EE \, \big( \<1'zeta->_\eps(y)\, \<1'zeta->_\eps(z) \big) \big| \lesssim \frac{\eps}{\zeta^2} \Big( \int_{\RR \times \TT_\eps} \big| (P_\eps^\theta)'(y,v) \, (P_\eps^\theta)'(z,v) \big| \, dv \Big)^2\;.
\end{equation*}
The conclusion of the lemma then follows from the bound on $(P_\eps^\theta)'$ in Lemma~\ref{lem:P_eps} and the convolution bound in Lemma~\ref{lem:convolution_singularity}. 
\end{proof}

The following three point correlation between the fields $D_u\<2'zeta->_\eps,\;\<1'zeta->_\eps$ and $\Psi_\eps$ will appear as the expectation term in the application of the spectral gap inequality.

\begin{lem}\label{2'1'I_1cov}
    For every $x,y,z\in\RR\times\TT$ and $u\in\RR\times\TT_\eps$\;, we have
        \begin{equation*}
            \begin{aligned}
                &\phantom{11}\Big| \EE \, \big( \Psi_\eps(x) \, \<1'zeta->_\eps(y) \, D_u\<2'zeta->_\eps(z)\, \big) \Big|\\
                &\lesssim \frac{\zeta^{-2}\eps^{2}}{(|y-z|+\eps)(|z-\eps u|+\eps)^2}\left(\frac{1}{|x-z|+\eps}+\frac{1}{|x-y|+\eps}\right),
            \end{aligned}     
        \end{equation*}
     where the proportionality constant is independent of $x,y,z,u$ and $\eps,\zeta \in (0,1)$.
\end{lem}
\begin{proof}
Recall that $\Psi_\eps(x)=I_1\big((P_{\eps}^{\theta})'(x,\cdot) \big)$. By Propositions~\ref{prop:WienerItoChaos} and~\ref{DxFWienerItoChaos} and Remark~\ref{rmk:rho}, we have
\begin{equation} \label{e:expansion}
    \begin{split}
    &\phantom{111}\EE \, \big( \Psi_\eps(x) \, \<1'zeta->_\eps(y) \, D_u\<2'zeta->_\eps(z)\, \big)\\
    &= \sum_{k,\ell \geq 2} \frac{1}{k! \ell!} \EE \bigg( I_1 \big( (P_\eps^\theta)(x,\cdot) \big) \, I_k \Big( \EE D^k_{\bullet} \, \<1'>_\eps^{(\zeta)}(y) \Big) \, I_{\ell} \Big( \EE D^{\ell}_{\bullet} D_u \, \<2'>_\eps^{(\zeta)}(z) \Big) \bigg)\;,
    \end{split}
\end{equation}
where we have removed the operation $[\cdot]$ since the sums are already from $k,\ell \geq 2$ and there is already $D_u$ operation for the noise $\<2'>_\eps^{(\zeta)}$. 

By Lemma~\ref{lem:wick}, the non-zero terms in the sum \eqref{e:expansion} are $\ell = k$, $\ell = k-1$ and $\ell = k+1$, and we can split the sum by
\begin{equation*}
    \EE \, \big( \Psi_\eps(x) \, \<1'zeta->_\eps(y) \, D_u\<2'zeta->_\eps(z)\, \big) = I_1 + I_2 + I_3\;,
\end{equation*}
where
\begin{equation*}
    \begin{split}
    I_1 &= \sum_{k \geq 2} \frac{1}{(k-1)!} \iint\limits_{(\RR \times \TT_\eps)^2} (P_\eps^\theta)'(x,v_1) \, \Bscal{D^{k-2}_{\bullet} D^2_{\vec{v}} \, \<1'>_\eps^{(\zeta)}(y)\,, \, D^{k-2}_{\bullet} D^2_{\vec{v}} D_u \<2'>_\eps^{(\zeta)}(z)} \, d\vec{v}\;,\\
    I_2 &= \sum_{k \geq 2} \frac{1}{(k-1)!} \iint\limits_{(\RR \times \TT_\eps)^2} (P_\eps^\theta)'(x,v_1) \, \Bscal{\EE D^{k-2}_{\bullet} D^2_{\vec{v}} \<1'>_\eps^{(\zeta)}(y)\,, \, \EE D^{k-2}_{\bullet} D_{v_2} D_u \<2'>_\eps^{(\zeta)}(z)} \, d \vec{v}\;,\\
    I_3 &= \sum_{k \geq 2} \frac{1}{k!} \iint\limits_{(\RR \times \TT_\eps)^2} (P_\eps^\theta)'(x,v_1) \, \Bscal{\EE D^{k-1}_{\bullet}D_{v_2} \, \<1'>_\eps^{(\zeta)}(y)\,, \, \EE D^{k-1}_{\bullet} D^2_{\vec{v}} D_u \<2'>_\eps^{(\zeta)}(z)} \, d \vec{v}\;,
    \end{split}
\end{equation*}
where the inner product (for each $k$) are taken as $L^2((\RR \times \TT_\eps)^{k-2})$ for $I_1$ and $I_2$, and $L^{2}((\RR \times \TT_\eps)^{k-1})$ for $I_3$, and we write $\vec{v} = (v_1, v_2) \in (\RR \times \TT_\eps)^2$. 

We give details for $I_1$. By Lemmas~\ref{lem:Holder} and~\ref{lem:derivative_genFF}, we have
\begin{equation*}
    \begin{split}
    &\phantom{111}\sum_{k \geq 2} \frac{1}{(k-1)!} \big| \bscal{D^{k-2}_{\bullet} D^2_{\vec{v}} \, \<1'>_\eps^{(\zeta)}(y)\,, \, D^{k-2}_{\bullet} D^2_{\vec{v}} D_u \<2'>_\eps^{(\zeta)}(z)} \big|\\
    &\leq \|D^2_{\vec{v}} \<1'>_\eps^{(\zeta)}(y) \|_{L_\omega^2} \cdot \|D^2_{\vec{v}} D_u \<2'>_\eps^{(\zeta)}(z)\|_{L_\omega^2}\\
    &\lesssim \frac{\eps}{\zeta^2} \Big(\prod_{i=1}^{2} |(P_\eps^\theta)'(y,v_i)| \Big) \cdot \Big( |(P_\eps^\theta)'(z,u)| \cdot \prod_{i=1}^{2} |(P_\eps^\theta)'(z,v_i)| \Big)\;.
    \end{split}
\end{equation*}
Plugging it back to the integral defining $I_1$, we get
\begin{equation*}
    |I_1| \lesssim \frac{\eps}{\zeta^2} |(P_\eps^\theta)'(z,u)| \iint\limits_{(\RR \times \TT_\eps)^2} |(P_\eps^\theta)'(x,v_1)| \cdot \prod_{i=1}^{2} \Big( |(P_\eps^\theta)(y,v_i)| \cdot |(P_\eps^\theta)'(z,v_i)| \Big) \, d \vec{v}\;.
\end{equation*}
The desired bound for $|I_1|$ then follows from Lemmas~\ref{lem:P_eps} and~\ref{lem:convolution_singularity}. The bounds for $I_2$ and $I_3$ can be obtained in essentially the same way. This completes the proof of the lemma. 
\end{proof}

\subsubsection{Preliminary lemmas on various sub-processes}
\label{sec:211_preliminaries}

Before proceeding with the proof of Proposition~\ref{prop:211_convergence_1}, we first present some preliminary lemmas. The following two lemmas provide the bounds for the upper part of the tree $\<2'1'1'>_\eps$.

\begin{lem} \label{lem:20_bound}
    For every $p\geq2 $ and $\delta>0$, we have
    \begin{equation*}
        \left\|\<2'0_> \right\|_{L^{p}_\omega}+\left\|\<2'0_>^{(\zeta)} \right\|_{L^{p}_\omega} \lesssim_p \eps^{-\delta},
    \end{equation*}
    where the proportionality constant is independent of $\eps,\zeta\in(0,1)$ and $y\in\RR\times\TT$.
\end{lem}
\begin{proof}
    Applying \eqref{e:sg} twice, we obtain
    \begin{equation*}
    \begin{split}
        &\left\|\<2'0_> \right\|_{L^{p}_\omega}+\left\|\<2'0_>^{(\zeta)} \right\|_{L^{p}_\omega} \lesssim_p \left\| \<0'0_> \right\|_{\Sigma_{\vec{u}}} \\
        \lesssim& \bigg( \iint\limits_{|y-z|,|y-z'|\lesssim1} |y-z|^{-2} |y-z'|^{-2} (|z-z'|+\eps)^{-2} dzdz' \bigg)^{\frac{1}{2}},
    \end{split}
    \end{equation*}
    where the last inequality follows from Lemmas~\ref{lem:P_eps} and \ref{lem:Lvecpcontrol}. Thus, the desired result follows.
\end{proof}

\begin{lem} \label{lem:2_u}
    For every $p\geq 2$ and $\delta\in(0,\frac{1}{2})$\;, we have
        \begin{equation} \label{e:D2}
            \left\|\<1'uy>^{(\zeta)}  \right\|_{L^{p}_\omega}+\left\|\<D2',1',0'_6'>^{(\zeta)}  \right\|_{L^{p}_\omega} \lesssim_p \zeta^{-1} \eps^{\frac{1}{2}-\delta} \frac{\eps^{\frac{3}{2}}}{(|y-\eps u|+\eps)^{2-\delta}}, 
        \end{equation}
    where the proportionality constant is independent of $\eps,\zeta\in(0,1)$, $y\in\RR\times\TT$, and $u\in\RR\times\TT_\eps$.
\end{lem}
\begin{proof}
    Using \eqref{e:sg} twice, we obtain
        \begin{equation*}
        \begin{split}
            &\left\|\<1'uy>^{(\zeta)}  \right\|_{L^{p}_\omega}+\left\|\<D2',1',0'_6'>^{(\zeta)}  \right\|_{L^{p}_\omega} \lesssim_p \zeta^{-1}\sqrt{\eps}\left\|\<DDD2'> \right\|_{\Sigma_{u_1,u_2}}\\
            \lesssim& \zeta^{-1} \eps^{2} \bigg\| \int_{|y-z|\lesssim1} |K'_\eps(y-z)| (|z-\eps u|+\eps)^{-2} (|z-r|+\eps)^{-\frac{5}{2}} dz \bigg\|_{L_{r}^2},
        \end{split}
        \end{equation*}
    where the last inequality is derived from Lemmas~\ref{lem:P_eps} and \ref{lem:Lvecpcontrol}. Hence, the desired bound directly follows from Lemma~\ref{lem:kernel_convolution_1}.
\end{proof}

The following lemma provides the estimate of the medium part of the tree $\<2'1'1'>_\eps$.
\begin{lem}\label{lem:210_holder_boundlem}
     For every $p\geq2 $, $\delta\in(0,\frac{1}{2})$\;, $x,z\in\RR\times\TT$ with $|x|\lesssim1$\;, we have
        \begin{equation*}
            \left\| \<z,1',x>^{(\zeta)}\right\|_{L^p_{\omega}} \lesssim_p\zeta^{-1}\eps^{\delta}(|x|^{\frac12-\delta}+\eps^{\frac12-\delta})\bigg(\frac{1}{|z|^2}+\frac{1}{|z-x|^2}\bigg),
        \end{equation*}
    where the proportionality constant is independent of $x,z\in \RR \times \TT$ and $\eps,\zeta \in (0,1)$.
\end{lem}
\begin{proof}
    By applying \eqref{e:sg} twice and then  applying Lemmas~\ref{lem:P_eps} and \ref{lem:Lvecpcontrol}, we have
        \begin{equation*}
            \begin{split}
                &\left\| \<z,1',x>^{(\zeta)}\right\|_{L^p_{\omega}}\lesssim_p \zeta^{-1}\sqrt{\eps}\left\|\<z,DD1',x>^{(\zeta)}\right\|_{\Sigma_{\vec{u}}}\\
                \lesssim& \zeta^{-1}\sqrt{\eps}\bigg\|\int_{\RR\times\TT} |K'_\eps(x-y)-K'_\eps(-y)| |y-z|^{-2} (|y-r|+\eps)^{-\frac{5}{2}} dy \bigg\|_{L^2_r}.
            \end{split}
        \end{equation*}
    Therefore, the desired bound is a direct consequence of Lemma~\ref{lem:kernel_convolution_3}.
\end{proof}

With the above estimates, we can derive the following bound.    
\begin{lem}\label{lem:210_holder_bound}
    For every $p\geq2 $, $\delta\in(0,\frac{1}{8})$ and $x\in\RR\times\TT$ with $|x|\lesssim1$, we have
        \begin{equation*}
            \left\| \<2',1',1'_2>^{(\zeta)} \right\|_{L^p_{\omega}} \lesssim_p \zeta^{-2}\eps^{\frac{\delta}{2}} (|x|^{\frac{1}{2}-\delta} + \eps^{\frac{1}{2}-\delta}),
        \end{equation*}
    where the proportionality constant is independent of $x\in \RR \times \TT$ and $\eps,\zeta \in (0,1)$.
\end{lem}
\begin{proof}
    Recall from Remark~\ref{rmk:graph_exception} that the quantity in consideration has mean $0$. Hence, by \eqref{e:sg} and \eqref{e:prod_deri}, we have
        \begin{equation*}
            \left\| \<2',1',1'_2>^{(\zeta)} \right\|_{L^p_{\omega}} \lesssim_p \left\| \<D2',D1',1'_5>^{(\zeta)}  \right\|_{\Sigma_u L^p_{\omega}} + \left\| \<2',D1',1'_4>^{(\zeta)}  \right\|_{\Sigma_u L^p_{\omega}} +\left\| \<D2',1',1'_3>^{(\zeta)}  \right\|_{\Sigma_u L^p_{\omega}} \triangleq \sum_{i=1}^3I_i.
        \end{equation*}
    First, we consider the term $I_1$. Applying H\"older inequality, along with Lemmas~\ref{lem:P_eps},  ~\ref{lem:2_u}, and~\ref{lem:Lvecpcontrol}, we obtain
        \begin{equation*}
        \begin{split}
            I_1 \lesssim_p & \left\|\int_{\RR\times\TT} \left\|\<D2',1',0'_6'>^{(\zeta)}  \right\|_{L^{2p}_\omega}\frac{\eps^{\frac32}}{(|y-\eps u|+\eps)^2} 
            |K'_\eps(x-y) - K'_\eps(-y)|dy \right\|_{\Sigma_u}\\
            \lesssim_p & \zeta^{-1}\eps^{\frac12-\delta} \bigg\| \int_{\RR\times\TT} \frac{|K'_\eps(x-y) - K'_\eps(-y)|}{(|y-r|+\eps)^{\frac{5}{2}-\delta}} dy \bigg\|_{L^2_r}.
        \end{split}
        \end{equation*}
    By Lemma~\ref{lem:kernel_convolution_4}, we then obtain
        \begin{equation*}
            I_1\lesssim_p\zeta^{-1}\eps^{\delta}(|x|^{\frac{1}{2}-\delta} + \eps^{\frac{1}{2}-\delta})\;.
        \end{equation*}
    Next, we treat the term $I_2$. Since the expectation of the stochastic term in $I_2$ is non-zero, we decompose it into a mean-zero term and an error term as
    \begin{equation} \label{e:210_i2_decomposition}
        \<2',D1',1'_4>^{(\zeta)}= \<2',1',1'_4> ^{(\zeta)} + \left( \<2',D1',1'_4> ^{(\zeta)}-\<2',1',1'_4> ^{(\zeta)} \right).
    \end{equation}
    We first treat the error term in the parenthesis above. We have  
        \begin{equation} \label{e:210_i2_error}
            \begin{aligned}
                \left\|(\<D1'->-\<1'u>)^{(\zeta)} (y) \right\|_{L^{2p}_\omega}& \lesssim \eps^{-\frac12} \int_{0}^{\sqrt{\eps}|(P_{\eps}^{\theta})'(y,u)|}\int_{0}^{r_1} \left\|F_\zeta'''(\sqrt{\eps}\Psi_\eps(y)+r_2)\right\|_{L^{2p}_\omega}dr_2dr_1\\
                &\lesssim_p \zeta^{-1}\sqrt{\eps}(P_{\eps}^{\theta})'(y,u)^2\lesssim\zeta^{-1}\eps^{\delta}\frac{\eps^{\frac32}}{(|y-\eps u|+\eps)^{2+\delta}}.
            \end{aligned}
        \end{equation}
    Combining it with Lemma~\ref{lem:20_bound}, the desired bound for the error term follows from Lemmas~\ref{lem:Lvecpcontrol} and~\ref{lem:kernel_convolution_4}. 
    
    Now we turn to the first term in \eqref{e:210_i2_decomposition}. It has mean $0$ since $\EE[\; \<2'zeta->_\eps(z)\;\<0'zeta->_\eps(0)]$ is even and $K'(-z)$ is odd in the spatial variable of $z$. By \eqref{e:sg} and \eqref{e:prod_deri}, we have
        \begin{equation*}
            \begin{split}
                \left\| \<2',1',1'_4> ^{(\zeta)}\right\|_{\Sigma_u L^p_{\omega}} \lesssim_p &\left\| \<2',D1',1'_6>^{(\zeta)}\right\|_{\Sigma_{\vec{u}}L^p_{\omega}} + \left\| \<D2',1',1'_7> ^{(\zeta)}\right\|_{\Sigma_{\vec{u}} L^p_{\omega}}\\
                &+\left\| \<D2',D1',1'_8>^{(\zeta)}\right\|_{\Sigma_{\vec{u}} L^p_{\omega}} \triangleq \sum_{i=1}^3I_{2i}.
                \end{split}
        \end{equation*}
    For the term $I_{21}$, H\"older inequality yields that
    \begin{equation*}
        \begin{split}
        &\phantom{111}\left\| \<2',D1',1'_6>^{(\zeta)}\right\|_{L^p_{\omega}}\\&
        \leq \int_{\RR\times\TT} \left\|\<2'0_>^{(\zeta)}\right\|_{L^{2p}_\omega} \left\|\<D_20u1>^{(\zeta)} (y)\right\|_{L^{2p}_\omega} |K_\eps'(x-y) - K_\eps'(-y)| \, dy\;.
        \end{split}
    \end{equation*}
    By \eqref{e:F'''_bound} and Lemma~\ref{lem:P_eps}, we have
    \begin{equation*}
        \left\|\<D_20u1>^{(\zeta)} (y)\right\|_{L^{2p}_\omega} \lesssim_p \zeta^{-1}\sqrt{\eps} \prod\limits_{i=1}^2 \frac{\eps^{\frac{3}{2}}}{(|y-\eps u_i|+\eps)^2}.
    \end{equation*}
    Combining it with Lemmas~\ref{lem:20_bound}, ~\ref{lem:Lvecpcontrol} and~\ref{lem:kernel_convolution_4}, we get
        \begin{equation*}
            I_{21} \lesssim_p \zeta^{-1} \eps^{\frac{\delta}{2}} \bigg\|\int_{\RR\times\TT}  \frac{|K_\eps'(x-y) - K_\eps'(-y)|}{(|y-r|+\eps)^{2+\delta}} dy \bigg\|_{L_r^2} \lesssim \zeta^{-1} \eps^{\frac{\delta}{2}}(|x|^{\frac{1}{2}-\delta}+|\eps|^{\frac{1}{2}-\delta})\;.
        \end{equation*}
    The desired estimates for $I_{22}$ and $ I_{23}$ can be similarly derived as the above proof of $I_1$.
    
    Finally, we deal with the term $I_3$. We split it into a main part and an error term by
        \begin{equation*}
            \<D2',1',1'_3>^{(\zeta)}=2\<1'u,1',x>^{(\zeta)} + \left(\<D2',1',1'_3>^{(\zeta)}-2\<1'u,1',x>^{(\zeta)}\right). 
        \end{equation*}
    For the error term, similar to \eqref{e:210_i2_error}, we have
        \begin{equation*}
                \left\|(\<D2'->-2\<2'u>)^{(\zeta)} (z) \right\|_{L^{2p}_\omega} \lesssim_p (P_\eps^\theta)'(z,u)^2 \lesssim \frac{\eps^{\frac32}}{(|z-\eps u|+\eps)^{\frac52}}.
        \end{equation*}
    By H\"older inequality and Lemma~\ref{lem:210_holder_boundlem}, we obtain
        \begin{equation*}
            \begin{aligned}
                \left\| \<D2',1',1'_3>^{(\zeta)}-2\<1'u,1',x>^{(\zeta)}  \right\|_{\Sigma_u L^p_{\omega}}
                &\leq  \left\|\int_{|z|\lesssim1}\left\|(\<D2'->-2\<2'u>)^{(\zeta)} (z) \right\|_{L^{2p}_\omega} \left\| \<z,1',x>^{(\zeta)}\right\|_{L^{2p}_{\omega}}dz\right\|_{\Sigma_u}\\
                &\lesssim_p \zeta^{-1}\eps^{
                \frac{\delta}{2}}(|x|^{\frac12-\delta}+\eps^{\frac12-\delta}),
            \end{aligned}
        \end{equation*}
    which gives the desired bound of the error term. Now we turn to the main part of $I_3$. By \eqref{e:sg} we have
        \begin{equation*}
            \begin{aligned}
                \left\|\<1'u,1',x>^{(\zeta)}  \right\|_{\Sigma_u L^p_{\omega}} \lesssim_p& \left\|\EE\<1'u,1',x>^{(\zeta)}  \right\|_{\Sigma_u}+ \left\|\<D21'u_1,1',x>^{(\zeta)}  \right\|_{\Sigma_{\vec{u}} L^p_{\omega}} +\left\|\<D21'u_1,D21',x>^{(\zeta)}  \right\|_{\Sigma_{\vec{u}} L^p_{\omega}}\\
                &+\left\|\<1'u_1,D21',x>^{(\zeta)}  \right\|_{\Sigma_{\vec{u}} L^p_{\omega}}
                \triangleq \sum_{i=1}^4I_{3i}.
            \end{aligned}
        \end{equation*}
    First we handle the term $I_{31}$. Note that we have
        \begin{equation*}
            \EE\<1'u,1',x> ^{(\zeta)}= \EE\<1'u,1',x>^{(\zeta)} - \EE \<1',1'u,x>^{(\zeta)}
        \end{equation*}
    since the second term is $0$. Then we obtain
        \begin{equation} \label{e:010_i11}
            I_{31} \leq\left\| \int_{\RR\times\TT} |K'_\eps(x-y) - K'_\eps(-y)|\left| \<E1'u,1'> - \<E1',1'u> \right| dy \right\|_{\Sigma_u},
        \end{equation}
    where the dashed line represents the covariance $\EE\big[\;\<1'zeta->_\eps(y)\; \<1'zeta->_\eps(z) \big]$.
    By Lemma~\ref{lem:covariance}, we have
        \begin{equation*}
            \begin{aligned}
                \left| \<E1'u,1'> - \<E1',1'u> \right|&\lesssim\zeta^{-2} \int_{\RR\times\TT} \frac{\eps |K'_\eps(y-z)|}{\big(|y-z|+\eps \big)^2} \cdot |(P^{\theta}_\eps)'(y,u) - (P^{\theta}_\eps)'(z,u)| dz\\
                &\lesssim\eps^{-\frac12}\zeta^{-2}\iint_{(\RR\times\TT)^2}\frac{|P'_{\eps}(y-r) -P'_{\eps}(z-r) |\left|\theta^{(\eps)} \Big(\frac{r}{\eps}-u \Big)\right|}{|y-z|^2\big(|y-z|+\eps \big)^2} dzdr.
            \end{aligned}
        \end{equation*}
    By (\ref{e:thetaeps}), Lemmas~\ref{lem:kernel_convolution_5} and \ref{lem:convolution_singularity'}, we obtain
        \begin{equation}\label{E1'u,1'-E1',1'u}
            \left| \<E1'u,1'> - \<E1',1'u> \right|\lesssim\zeta^{-2}\frac{\eps^{\frac{3}{2}+\delta}}{(|y-\eps u|+\eps)^{2+\delta}}.
        \end{equation}
    By Lemmas~\ref{lem:Lvecpcontrol} and~\ref{lem:kernel_convolution_4}, we can deduce the desired bound
    \begin{equation*}
        I_{31}\lesssim\zeta^{-2} \eps^{\frac{\delta}{2}}(|x|^{\frac12-\delta}+\eps^{\frac12-\delta}).
    \end{equation*}
    For the term $I_{32}$, H\"older inequality implies that  
        \begin{equation*}
           I_{32} \leq\left\|\int_{|z|\lesssim1} \left\|\<D_21u1>^{(\zeta)}(z)\right\|_{L^{2p}_\omega}\left\| \<z,1',x>^{(\zeta)}\right\|_{L^{2p}_{\omega}} dz\right\|_{\Sigma_{\vec{u}}}.
        \end{equation*}
    Substituting
    \begin{equation*}
        \left\|\<D_21u1>^{(\zeta)}(z)\right\|_{L^{2p}_\omega}\lesssim_p \prod_{i=1}^2 \frac{\eps^{\frac{3}{2}}}{(|z-\eps u_i|+\eps)^2}
    \end{equation*}
    and Lemma~\ref{lem:210_holder_boundlem} into this inequality, and then applying Lemma~\ref{lem:Lvecpcontrol}, we can conclude that $I_{32}\lesssim_p\zeta^{-1}\eps^{\frac{\delta}{2}}(|x|^{\frac12-\delta}+\eps^{\frac12-\delta})$.

    Next, we consider the term $I_{33}$. By H\"older inequality, $I_{33}$ can be bounded by 
        \begin{equation*}
            \left\|\int_{\RR\times\TT}\left\|\<D_21u1y>^{(\zeta)}\right\|_{L_\omega^{2p}}\left\|\<D_{2}1'->^{(\zeta)}(y)\right\|_{L_\omega^{2p}}|K'_\eps(x-y)-K'_\eps(-y)|dy \right\|_{\Sigma_{\vec{u}}}.
        \end{equation*}
    For simplicity, we write 
        \begin{equation*}
            f(y,u_1,u_2)\coloneqq \left\|\<D_21u1y>^{(\zeta)}\right\|_{L_\omega^{2p}} \quad\text{and}\quad g(y,u_2)\coloneqq \left\|\<D_{2}1'->^{(\zeta)}(y)\right\|_{L_\omega^{2p}}.
         \end{equation*}
    Note that the variable $u_2$ appears in two different places in the integrand, so we separate them first. Using $\|h\|_{L^q}=\|h^2\|_{L^{\frac q2}}^{\frac 12}$ and applying Minkowski inequality, we obtain
        \begin{equation*}
            \begin{split}
                 I_{33}\lesssim \bigg(\iint\limits_{(\RR\times\TT)^2}&\|f(y,u_1,u_2)f(y',u_1,u_2)g(y,u_2)g(y',u_2)\|_{\Sigma^{\{1,\frac p2\}}_{\vec{u}}}\\
                 &|K'_\eps(x-y)-K'_\eps(-y)| \cdot |K'_\eps(x-y')-K'_\eps(-y')|dydy'\bigg)^{\frac12}.
            \end{split}
        \end{equation*}
    By H\"older inequality, we get
        \begin{equation} \label{e:210_i33}
            \begin{aligned}
                  I_{33}\lesssim \bigg(\iint\limits_{(\RR\times\TT)^2}& \left\|\left\|f(y,u_1,u_2)\right\|_{\Sigma_{u_2}^{\{4,2p\}}}\right\|_{\Sigma_{u_1}}\left\|\left\|f(y',u_1,u_2)\right\|_{\Sigma_{u_2}^{\{4,2p\}}}\right\|_{\Sigma_{u_1}}\|g(y,u_2)g(y',u_2)\|_{\Sigma_{u_2}}\\&|K'_\eps(x-y)-K'_\eps(-y)||K'_\eps(x-y')-K'_\eps(-y')|dydy'\bigg)^{\frac12}.  
            \end{aligned}
        \end{equation}
    Note that we have $\big\|\left\|f(y,u_1,u_2)\right\|_{\Sigma_{u_2}^{\{4,2p\}}}\big\|_{\Sigma_{u_1}}\lesssim_p\eps^{-\frac{\delta}{2}}$ by the proof of Lemma~\ref{lem:20_bound}. The bound $g(y,u_2)\lesssim_p \frac{\eps^{\frac32}}{(|y-\eps u_2|+\eps)^{2}}$ together with Remark~\ref{rmk:convolution_singularity'} imply that    
        \begin{equation*}
            \|g(y,u_2)g(y',u_2)\|_{\Sigma_{u_2}} \lesssim_p\frac{\eps}{(|y-y'|+\eps)^2}.
        \end{equation*}
    Substituting these bounds into \eqref{e:210_i33} and applying Lemma~\ref{lem:kernel_convolution_4}, we obtain $I_{33}\lesssim_p\eps^{\frac{\delta}{2}}(|x|^{\frac12-\delta}+\eps^{\frac12-\delta})$.

    For the term $I_{34}$, by H\"older inequality we have
         \begin{equation*}
            I_{34}\lesssim_p \left\|\int_{\RR\times\TT} \left\|\<1'u1y>^{(\zeta)}\right\|_{L_\omega^{2p}} \frac{\eps^{\frac32}}{(|y-\eps u_2|+\eps)^2} |K'_\eps(x-y)-K'_\eps(-y)|dy\right\|_{{\Sigma}_{\vec{u}}}.
        \end{equation*}
    By Lemmas~\ref{lem:2_u}, \ref{lem:Lvecpcontrol} and \ref{lem:kernel_convolution_4}, we can derive that $I_{34}\lesssim _p\zeta^{-1}\eps^{\delta}(|x|^{\frac12-\delta}+\eps^{\frac12-\delta})$. The bounds for $I_{3i}(i=1,2,3,4)$ establishes that
    \begin{equation*}
        I_3\lesssim_p\zeta^{-2}\eps^{\frac{\delta}{2}}(|x|^{\frac12-\delta}+\eps^{\frac12-\delta}). 
    \end{equation*}
    Combining the estimates of $I_1,I_2,I_3$ together, the proof is completed.
\end{proof}

\subsubsection{Convergence of the regularised part -- proof of Proposition~\ref{prop:211_convergence_1}}
\label{sec:211_main}

We are now ready to prove Proposition~\ref{prop:211_convergence_1}, focusing specifically on two types of trees. The following tree consists of $[\tau]$ for every noise $\tau$ appearing in $\<2'1'1'>$. 

\begin{prop}\label{prop:211_convergence_3}
    For every $p\geq2$ and $\delta\in(0,\frac{1}{8})$\;, the bound
        \begin{equation}\label{e:<211>_converge_3}
            \left\| \<2',1',1'>^{(\zeta)} - \EE\<2',1',1'>^{(\zeta)} \right\|_{L^p_\omega} \lesssim_p \zeta^{-2} \eps^{\delta} \lambda^{-3\delta}
        \end{equation}
    holds uniformly in $\eps,\zeta,\lambda\in(0,1)$ and $\varphi\in \bar{C}_c^1$.
\end{prop}
\begin{proof}
    By \eqref{e:sg}, the quantity is bounded by the $\Sigma_u L_\omega^p$ norm of its Malliavin derivative, which by \eqref{e:prod_deri} has the expression
        \begin{equation} \label{e:decompose_d211}
            \begin{split} 
                &\<D2',1',1'>^{(\zeta)} + \<2',D1',1'> ^{(\zeta)}+ \<2',1',D1'>^{(\zeta)} + \<D2',D1',1'>^{(\zeta)}\\
                &+ \<D2',1',D1'>^{(\zeta)} + \<2',D1',D1'>^{(\zeta)} + \<D2',D1',D1'>^{(\zeta)} \triangleq \sum_{k=1}^7I_{k}^u\;.
            \end{split}
        \end{equation}
    We first consider the terms $I_1^u$ and $I_4^u$, where the top $[\<2'>]$ term has a Malliavin derivative and the bottom $[\<1'>]$ term does not. Applying H\"older inequality and naively controlling the middle $\|[\<1'>]\|_{L_y^\infty L_\omega^{3p}}$ and $\|D_u [\<1'>]\|_{L_{u,y}^\infty L_\omega^{3p}}$ by $\eps^{-\frac{1}{2}}$, we get
    \begin{equation} \label{e:211:1+4}
        \|I_1^u + I_4^u\|_{\Sigma_u L^p_{\omega}} \lesssim_p \eps^{-\frac{1}{2}} \left\| \int_{\RR\times\TT}  \left\|\<D2',1',0'_6'>^{(\zeta)}  \right\|_{L^{3p}_\omega} \left\|\<2',-1',1'_1>^{(\zeta)}\right\|_{L^{3p}_{\omega}} dy\right\|_{\Sigma_u}. 
    \end{equation}
    For the second term in the integrand above, we can apply \eqref{e:sg} twice to get
     \begin{equation} \label{e:2',-1',1'_1zeta}
     \begin{split}
        &\left\|\<2',-1',1'_1>^{(\zeta)} \right\|_{L^{3p}_{\omega}} \lesssim_p \frac{\sqrt{\eps}}{\zeta} \left\|\<1',1',0'_7_>\right\|_{\Sigma_{\vec{u}}}\\
        \lesssim& \frac{\sqrt{\eps}}{\zeta} \bigg\|\int_{\RR\times\TT} |\varphi^\lambda(x)| \cdot  \frac{|K'_\eps(x-y) - K'_\eps(-y)|}{\big(|x- r| +\eps \big)^\frac{5}{2}} dx \bigg\|_{L^2_r} \lesssim \frac{\eps^{\frac{1}{2}-\delta}}{\zeta} \cdot \frac{\lambda^{-3\delta} \boldsymbol{1}_{|y|\lesssim 1}}{|y|^{3-4\delta}},
    \end{split}
    \end{equation}
    where the second inequality follows from Lemmas~\ref{lem:P_eps} and~\ref{lem:Lvecpcontrol}, and the last bound follows from Lemma~\ref{lem:kernel_convolution_2}. Combining it with \eqref{e:D2} which controls the first term in the integrand in \eqref{e:211:1+4}, we conclude with Lemma~\ref{lem:Lvecpcontrol} that
    \begin{equation*}
        \|I_1^u+I_4^u\|_{\Sigma_u L^p_{\omega}}\lesssim_p \zeta^{-2}\eps^{\delta} \lambda^{-3\delta}.
    \end{equation*}
    The term $I_2^u$ can be treated in a similar way. We have
    \begin{equation*}
        \|I_2^u\|_{\Sigma_u L_\omega^p} \lesssim_p \left\| \int_{\RR\times\TT} \eps^{-\delta}  \frac{\eps^{\frac{3}{2}}}{(|y-\eps u|+\eps)^2} \left\|\<2',-1',1'_1>^{(\zeta)}\right\|_{L^{3p}_{\omega}} dy \right\|_{\Sigma_u} \lesssim_p \zeta^{-1}\eps^{\delta} \lambda^{-3\delta}\;,
    \end{equation*}
    where the first inequality follows from H\"older inequality and Lemma~\ref{lem:20_bound}, and the second bound follows from \eqref{e:2',-1',1'_1zeta} and Lemma~\ref{lem:Lvecpcontrol}. 

    We now turn to $I_3^u$. Recall the notation in Remark~\ref{rmk:graph_exception} and \eqref{e:graph_exception}. Similar as before, we can bound it by
        \begin{equation*}
              \|I_3^u\|_{\Sigma_u L_\omega^p} \lesssim_p \left\| \int_{\RR\times\TT} \left\| \<2',1',1'_2>^{(\zeta)}\right\|_{L^{2p}_{\omega}}\frac{\eps^{\frac{3}{2}}}{(|x-\eps u|+\eps)^2}|\varphi^{\lambda}(x)|dx\right\|_{\Sigma_u}.   
        \end{equation*}
    By Lemmas~\ref{lem:210_holder_bound} and~\ref{lem:Lvecpcontrol}, we obtain
        \begin{equation*}
            \|I_3^u\|_{\Sigma_u L^p_{\omega}} \lesssim_p \zeta^{-2}\eps^{\frac{3\delta}{2}}   \left\| \int_{\RR\times\TT} (|x|^{\frac{1}{2}-3\delta} + \eps^{\frac{1}{2}-3\delta})\frac{\eps^{\frac32}}{(|x-\eps u|+\eps)^2}|\varphi^{\lambda}(x)|dx\right\|_{\Sigma_u}\lesssim\zeta^{-2}\eps^{\frac{3\delta}{2}} \lambda^{-3\delta}.
        \end{equation*}
    We now treat the terms $I_5^u$ and $I_7^u$. By H\"older inequality and naively controlling the middle $\|[\<1'>]\|_{L_y^\infty L_\omega^{3p}}$ and $\|D_u [\<1'>]\|_{L_{u,y}^\infty L_\omega^{3p}}$ by $\eps^{-\frac{1}{2}}$, we have
        \begin{equation*}
              \|I_5^u + I_7^u\|_{\Sigma_u L_\omega^p} \lesssim_p \eps^{-\frac{1}{2}} \left\| \int _{\RR\times\TT} \left\|\<D2',1',0'_6'>^{(\zeta)}  \right\|_{L^{3p}_\omega}\left\|\<2',-1',D1'_1>^{(\zeta)}\right\|_{L^{3p}_{\omega}} dy\right\|_{\Sigma_u}.   
        \end{equation*}
    The integration variable $u$ appears twice in the integrand in different contexts. We first separate them with the same method for controlling $I_{33}$ in Lemma~\ref{lem:210_holder_bound}. For simplicity, we write
    \begin{equation*}
        f(u,y)\coloneqq \left\|\<D2',1',0'_6'>^{(\zeta)}  \right\|_{L^{3p}_\omega} \quad \text{and} \quad g(u,y)\coloneqq \left\|\<2',-1',D1'_1>^{(\zeta)}\right\|_{L^{3p}_{\omega}}.
    \end{equation*}
    As in the proof of Lemma~\ref{lem:210_holder_bound}, by triangle and H\"older inequalities, we have
        \begin{equation*}
        \begin{split}
            \|I_5^u + I_7^u\|_{\Sigma_u L_\omega^p}  &\lesssim_p \eps^{-\frac12} \bigg(\iint_{(\RR\times\TT)^2}  \| f(u,y) f(u,y') g(u,y) g(u,y') \|_{\Sigma^{\{1,\frac p2\}}_u}  dydy'\bigg)^{\frac12}\\
            &\lesssim \eps^{-\frac12} \bigg(\iint_{(\RR\times\TT)^2}\|f(u,y) f(u,y')\|_{\Sigma_u} \|g(u,y)\|_{\Sigma^{\{4,2p\}}_u} \|g(u,y')\|_{\Sigma^{\{4,2p\}}_u}dydy'\bigg)^{\frac12}.
        \end{split}
        \end{equation*}
    By Lemma~\ref{lem:2_u} and the estimate in \eqref{e:2',-1',1'_1zeta}, we have
    \begin{equation*}
        f(u,y) \lesssim_p \zeta^{-1} \eps^{2-\delta}(|y-\eps u|+\eps)^{-2+\delta} \quad\text{and}\quad
        \|g(u,y)\|_{\Sigma^{\{4,2p\}}_u} \lesssim_p \zeta^{-1}\frac{\lambda^{-3\delta} \eps^{\frac{1}{2}-\delta} \boldsymbol{1}_{|y|\lesssim1}}{|y|^{3-4\delta}}\;.
    \end{equation*}
    Substituting them into the above bound for $I_5^u + I_7^u$, we conclude that
    \begin{equation*}
        \|I_5^u+I_7^u\|_{\Sigma_u L^p_{\omega}}\lesssim_p \zeta^{-2}\eps^{\delta}\lambda^{-3\delta}\;.
    \end{equation*}
    Finally, the desired bound for the term $I_6^u$ can be obtained in the same way as for the terms $I_5^u$ and $I_7^u$. Combining the bounds for all these terms, we complete the proof.
\end{proof}

    Now we turn to the second tree. It is similar to the previous one except that the lowest noise node is $\tT^{(1)}\<1'>_\eps^{(\zeta)}$ instead of $[\<1'>^{(\zeta)}]_\eps$. We have the following proposition. 

\begin{prop}\label{prop:211_convergence_4}
    For every $p\geq2$ and $\delta\in(0,\frac{1}{8})$\;, the bound
        \begin{equation*}
            \left\| \<2',1',S1>^{(\zeta)} - \EE\<2',1',S1>^{(\zeta)} \right\|_{L^p_\omega} \lesssim \zeta^{-2} \eps^{\delta} \lambda^{-3\delta}
        \end{equation*}
    holds uniformly in $\eps,\zeta,\lambda\in(0,1)$ and $\varphi\in \bar{C}_c^1$.
\end{prop}
\begin{proof}
    We decompose the lowest noise node into
    \begin{equation*}
        \tT^{(1)} \<1'>_\eps^{(\zeta)}(x) = \frac{a_\eps^{(\zeta)}}{a_\eps} \cdot \Psi_\eps(x) + \Big( \tT^{(1)} \<1'>_\eps^{(\zeta)}(x) - \frac{a_\eps^{(\zeta)}}{a_\eps} \cdot \Psi_\eps(x) \Big)\;,
    \end{equation*}
    which leads to the decomposition (for the object of study)
    \begin{equation} \label{e:211_l_decomposition}
        \<2',1',S1>^{(\zeta)} = \frac{a_{\eps}^{(\zeta)}}{a_{\eps}} \<2',1',1>^{(\zeta)} + \left( \<2',1',S1>^{(\zeta)} - \frac{a_{\eps}^{(\zeta)}}{a_{\eps}} \<2',1',1>^{(\zeta)} \right)\;,
    \end{equation}
    and the same for its expectation. Recall from Remark~\ref{rmk:graph_exception} that in the graphic notation, the expectation of product between two top noise nodes is already subtracted. Also, the horizontal arrow denotes the noise $\Psi_\eps(x)$ with the dummy variable $x$ integrated out. 
    
    Bounds for the two parts in \eqref{e:211_l_decomposition} (with expectation subtracted further) are provided by Lemmas~\ref{lem:2'1'1_error} and~\ref{lem:211_convergence_6} respectively. This completes the proof of the proposition. 
\end{proof}
    
The following lemma provides the estimate of the error part (the second term) in the decomposition \eqref{e:211_l_decomposition}.

\begin{lem} \label{lem:2'1'1_error}
    For every $p\geq2$ and $\delta\in(0,\frac{1}{8})$\;, we have the bound
        \begin{equation*}
            \left\| \left( \<2',1',S1>^{(\zeta)} - \EE\<2',1',S1>^{(\zeta)}\right) - \frac{a_{\eps}^{(\zeta)}}{a_{\eps}}\left(\<2',1',1>^{(\zeta)} - \EE\<2',1',1>^{(\zeta)}\right) \right\|_{L^p_\omega} \lesssim\zeta^{-2} \eps^{\delta} \lambda^{-3\delta}
        \end{equation*}
    uniformly in $\eps,\zeta,\lambda\in(0,1)$ and $\varphi\in \bar{C}_c^1$.
\end{lem}
\begin{proof}
    This stochastic object is very similar to the one in Proposition~\ref{prop:211_convergence_3}, except that instead of $[\<1'>^{(\zeta)}]_\eps$, the lowest noise node is $\Er_\eps^{(\zeta)}$ given by the expression \eqref{e:1_error}. By \eqref{e:<1>_error_first} and \eqref{e:<1>_error}, $\Er_\eps^{(\zeta)}$ satisfies all the desired bounds of $[\<1'>^{(\zeta)}]_\eps$, in particular the same scaling behaviour and gaining of a factor $\sqrt{\eps}$ by increasing $\frac{1}{2}$-degree singularity. Hence, the desired bound follows from the same procedure as the proof of Proposition~\ref{prop:211_convergence_3}.
\end{proof}

\begin{lem}\label{lem:211_convergence_6}
    For every $p\geq2$ and $\delta\in(0,\frac{1}{8})$\;, the bound
    \begin{equation*}
        \left\| \<2',1',1>^{(\zeta)} - \EE\<2',1',1>^{(\zeta)} \right\|_{L^p_\omega} \lesssim_p \zeta^{-2} \eps^{\delta} \lambda^{-3\delta}
    \end{equation*}
    holds uniformly in $\eps,\zeta,\lambda\in(0,1)$ and $\varphi\in \bar{C}_c^1$.
\end{lem}
\begin{proof}
    The quantity can be bounded by the $\Sigma_u L_\omega^p$ norm of its Malliavin derivative. Its Malliavin derivative has the expression
        \begin{equation} \label{e:decompose_d211l}
            \begin{aligned}
                &\<Du2',1',1>^{(\zeta)} +\<2',Du1',1>^{(\zeta)} +\<2',1',Du1>^{(\zeta)} +\<Du2',Du1',1>^{(\zeta)} \\
                &+\<Du2',1',Du1>^{(\zeta)} + \<2',Du1',Du1>^{(\zeta)} +\<Du2',Du1',Du1>^{(\zeta)}\triangleq\sum_{i=1}^{7}I_i^u.
            \end{aligned}
        \end{equation}
    The terms can be treated similarly to those in Proposition~\ref{prop:211_convergence_3} except $I_1^u$ and $I_2^u$. We only provide details for $I_1^u$, and $I_2^u$ can be bounded in a similar way. By \eqref{e:sg}, we have
    \begin{equation*}
        \|I_1^{u}\|_{\Sigma_u L^p_{\omega}} \lesssim_p \|\EE I_1^{u}\|_{\Sigma_u} + \|D_{u_2} I_1^{u_1}\|_{\Sigma_{\vec{u}} L^p_{\omega}}.
    \end{equation*}
    where $D_{u_2} I_1^{u_1}$ has the expression
        \begin{equation*}
            \begin{aligned}
                D_{u_2} I_1^{u_1} &= \<DD2',1',1>^{(\zeta)} + \<D_{1}2',D_{2}1',1>^{(\zeta)} + \<D_{1}2',1',0'_2>^{(\zeta)} +\<DD2',D_{2}1',1>^{(\zeta)}\\
                &+\<DD2',1',D_{2}1>^{(\zeta)} +\<D_{1}2',D_{2}1',D_{2}1>^{(\zeta)} + \<DD2',D_{2}1',D_{2}1>^{(\zeta)} \triangleq \sum_{i=1}^{7}I_{1i}^{\vec{u}},
            \end{aligned}
        \end{equation*}
    where $\vec{u}=(u_1,u_2)\in(\RR\times\TT_\eps)^2$. The complicated terms are $\EE I_{1}^u$, $I_{11}^{\vec{u}}$ and $I_{13}^{\vec{u}}$. We provide estimates for $\EE I_{1}^u$ and $I_{13}^{\vec{u}}$, as $I_{11}^{\vec{u}}$ can be treated similarly to $I_{13}^{\vec{u}}$. By Lemma~\ref{2'1'I_1cov} and using the brutal bound $|K_\eps'(x-y) - K_\eps'(-y)| \lesssim |y|^{-2} + |y-x|^{-2}$, we get
    \begin{equation*}
        |\EE I_1^u| \lesssim I_{101}^u + I_{102}^u\;,
    \end{equation*}
    where
    \begin{equation*}
        \begin{split}
        I_{101}^{u} &= \frac{\eps^\delta}{\zeta^2} \, \iiint\limits_{(\RR \times \TT)^3} \frac{\eps^{\frac32} |\varphi^{\lambda}(x)|}{(|z-\eps u|+\eps)^2|x-z|^{\frac12+2\delta}|y-z|^{3-\delta}}\bigg(\frac{1}{|y|^2}+\frac{1}{|y-x|^2}\bigg) \, dx dy dz\;,\\
        I_{102}^u &= \frac{\eps^\delta}{\zeta^2} \, \iiint\limits_{(\RR \times \TT)^3} \frac{\eps^{\frac32} |\varphi^{\lambda}(x)|}{(|z-\eps u|+\eps)^2|x-y|^{\frac12+2\delta}|y-z|^{3-\delta}} \bigg(\frac{1}{|y|^2}+\frac{1}{|y-x|^2}\bigg) \, dx dy dz\;.
        \end{split}
    \end{equation*}
    The difference between these two are the terms $|x-z|^{\frac{1}{2}+2\delta}$ and $|x-y|^{\frac{1}{2}+2\delta}$ on their denominators respectively. Note that we have
    \begin{equation*}
        \int_{\RR\times\TT} |\varphi^{\lambda}(x)| \cdot |x-y|^{-\alpha} dx \lesssim (|y|+\lambda)^{-\alpha}
    \end{equation*}
    for $0<\alpha<3$. Integrating out $y$ and $x$ successively for $I_{101}^u$, and $z$ and $x$ successively for $I_{102}^u$, we get
    \begin{equation*}
        \begin{split}
        I_{101}^u &\lesssim \frac{\eps^\delta}{\zeta^2} \int_{\RR\times\TT}\frac{\eps^{\frac32}}{(|z-\eps u|+\eps)^2} \bigg(\frac{1}{(|z|+\lambda)^{\frac12+2\delta}|z|^{2-\delta}}+\frac{1}{(|z|+\lambda)^{\frac52+\delta}}\bigg) \, dz\;,\\
        I_{102}^u &\lesssim \frac{\eps^\delta}{\zeta^2} \int_{\RR\times\TT}\frac{\eps^{\frac32}}{(|y-\eps u|+\eps)^{2-\delta}} \bigg(\frac{1}{(|y|+\lambda)^{\frac12+2\delta}|y|^2}+\frac{1}{(|y|+\lambda)^{\frac52+ 2\delta}}\bigg) \, dy\;,
        \end{split}
    \end{equation*}
    whose upper bounds are almost the same except the small $\delta$ are placed on different factors. Both fit Lemma~\ref{lem:Lvecpcontrol} in the same way. An application of that lemma yields
    \begin{equation*}
        \|\EE I_1^u\|_{\Sigma_u} \lesssim \|I_{101}^u\|_{\Sigma_u} + \|I_{102}^u\|_{\Sigma_u} \lesssim \frac{\eps^\delta \lambda^{-3\delta}}{\zeta^2}\;.
    \end{equation*}
    This completes the proof of $\EE I_1^u$. For the term $I_{13}^{\vec{u}}$, we divide it into two parts by 
        \begin{equation*}
            \<D_{1}2',1',0'_2>^{(\zeta)} =\<1'u1,1',0u2>^{(\zeta)} +\left(\<D_{1}2',1',0'_2>^{(\zeta)}-\<1'u1,1',0u2>^{(\zeta)} \right)\triangleq \bar{I}^{\vec{u}}_{13}+(I^{\vec{u}}_{13}-\bar{I}^{\vec{u}}_{13}).
        \end{equation*}
    The error term $I^{\vec{u}}_{13}-\bar{I}^{\vec{u}}_{13}$ can be easily bounded as in \eqref{e:210_i2_error}, so we omit the details. For the main part $\bar{I}^{\vec{u}}_{13}$, we again use the spectral gap inequality \eqref{e:sg} to control it by
        \begin{equation} \label{e:211_13_decomposition}
            \begin{aligned}
                \|\bar{I}^{\vec{u}}_{13}\|_{\Sigma_{\vec{u}} L^p_{\omega}} &\lesssim_p \left\|\EE\<1'u1,1',0u2>^{(\zeta)} \right\|_{\Sigma_{\vec{u}}} +\left\|\<D_{3}1'u1,1',0u2>^{(\zeta)} \right\|_{\Sigma_{\vec{u}} L^p_{\omega}}\\
                &+ \left\|\<1'u1,D_{3}1',0u2>^{(\zeta)} \right\|_{\Sigma_{\vec{u}} L^p_{\omega}} 
                +\left\|\<D_{3}1'u1,D_{3}1',0u2>^{(\zeta)} \right\|_{\Sigma_{\vec{u}} L^p_{\omega}} \triangleq\sum_{i=0}^{3}I_{13i}.
            \end{aligned}
        \end{equation}
   Terms $I_{130}$ and $I_{131}$ are the harder ones. For $I_{130}$, similar to the proof of the term $I_{31}$ in Lemma~\ref{lem:210_holder_bound}, we first write
    \begin{equation*}
        \EE\<1'u1,1',0u2>^{(\zeta)}  = \EE\<1'u1,1',0u2>^{(\zeta)}  - \EE \<1',1'u1,0u2>^{(\zeta)}\;,
    \end{equation*}
    since the second term on the right hand side above is $0$. By triangle inequality, we then have
    \begin{equation} \label{e:i130}
        I_{130} \leq \left\| \int_{\RR\times\TT} \left\|\<1',1',0'_7'u2>\right\|_{\Sigma_{u_2}} \left| \<E1'u1,1'> - \<E1',1'u1> \right| dy \right\|_{\Sigma_{u_1}},
    \end{equation}
    where the dashed line represents the covariance $\EE\big[\;\<1'zeta->_\eps(y)\; \<1'zeta->_\eps(z) \big]$ (and the $z$ variable is integrated out). Similar to \eqref{e:2',-1',1'_1zeta}, by Lemmas~\ref{lem:Lvecpcontrol} and~\ref{lem:kernel_convolution_2}, we have
    \begin{equation}\label{e:y1u1}
        \left\|\<1',1',0'_7'u2>\right\|_{\Sigma_{u_2}} \lesssim \frac{\lambda^{-3\delta} \eps^{-\delta} \boldsymbol{1}_{|y|\lesssim1}}{|y|^{\frac{5}{2}-4\delta}}.
    \end{equation}
    Substituting it and \eqref{E1'u,1'-E1',1'u} into \eqref{e:i130}, we obtain
    \begin{equation*}
        I_{130}\lesssim \zeta^{-2}\left\| \int_{\RR\times\TT} \frac{\lambda^{-3\delta} \eps^{-\delta} \boldsymbol{1}_{|y|\lesssim1}}{|y|^{\frac{5}{2}-4\delta}}  \frac{\eps^{\frac32+2\delta}}{(|y-\eps u_1|+\eps)^{2+2\delta}} dy \right\|_{\Sigma_{u_1}} \lesssim  \zeta^{-2}\eps^{\delta} \lambda^{-3\delta}.
    \end{equation*}
    Next we turn to $I_{131}$ in \eqref{e:211_13_decomposition}. We split this stochastic object into $D_{u_3} [\<1'>]_\eps(z) \cdot (P_\eps^\theta)'(z,u_1)$ and the rest, with the $z$ variable integrated out. Using H\"older inequality to replace the $L_\omega^p$-norm of the integral over $z$ by the integration of the $L_{\omega}^{2p}$-norms of each, and applying Lemma~\ref{lem:derivative_genFF}, we get
    \begin{equation} \label{e:211_131}
        I_{131} \lesssim_p \frac{\sqrt{\eps}}{\zeta} \left\| \int_{\RR\times\TT}  \left| \<zu1u3>\right| \;  \left\|  \<1',1',0'_10>  \right\|_{\Sigma_{u_2,u_4,u_5}^{2p}}  \,dz \right\|_{\Sigma_{u_1,u_3}^p}\;,
    \end{equation}
    where we have applied the spectral gap inequality \eqref{e:sg} twice to the lower noise node $[\<1'>]$. Now, with the same trick as in the estimate of $I_{33}$ in Lemma~\ref{lem:210_holder_bound}, the object with $\Sigma_{u_2, u_4, u_5}^{2p}$-norm in the integrand above can be controlled by
    \begin{equation*}
        \left\| \quad\iint\limits_{|y-z|,|y'-z|\lesssim1} \frac{1}{|y-z|^2|y'-z|^2(|y-y'|+\eps)^2} \left|\<1',1',0'_7'u2>\right| \left|\<1',1',0'_7'u2'>\right|dydy'\right\|_{\Sigma^{\{1,p\}}_{u_2}}^{\frac12}\;,
    \end{equation*}
    where we have used Lemma~\ref{lem:Lvecpcontrol} to integrate out $u_4$ and $u_5$ first. Now, using triangle and H\"older inequalities to move the $\Sigma_{u_2}^{\{1,p\}}$-norm inside so that the two terms with $u_2$ in the integrand are equipped with $\Sigma_{u_2}^{2p}$-norm each, and applying \eqref{e:y1u1} as well as Lemma~\ref{lem:Lvecpcontrol} again, we get
    \begin{equation*}
        \left\|  \<1',1',0'_10>  \right\|_{\Sigma_{u_2,u_4,u_5}^{2p}} \lesssim_p \eps^{-\delta}\lambda^{-3\delta}\left\| \int_{|y-z|\lesssim1,|y|\lesssim1} \frac{1}{|y-z|^2(|y-r|+\eps)^{\frac52}|y|^{\frac{5}{2}-4 \delta}}dy\right\|_{L^2_r}\;.
    \end{equation*}
    Splitting the integration domain into $\{y: |y|\leq\frac{|z|}{2}\}$, $\{y: |y-z|\leq\frac{|z|}{2}\}$ and $\{y: |y|,|y-z|>\frac{|z|}{2}\}$, we obtain the bound
    \begin{equation*}
        \phantom{111}\left\| \<1',1',0'_10>  \right\|_{\Sigma^{2p}_{u_2,u_4,u_5}}\lesssim_p\frac{\eps^{-\frac12+\delta}\lambda^{-3\delta}  \boldsymbol{1}_{|z|\lesssim1}}{|z|^{2-2\delta}}\;.
    \end{equation*}
    Plugging the above bound back into \eqref{e:211_131} gives the desired control for the term $I_{131}$. 
    The bounds for the other terms can be obtained in similar but simpler ways. This completes the proof of Lemma~\ref{lem:211_convergence_6}. 
\end{proof}

\begin{remark}
The decomposition \eqref{e:decompose_d211l} of the Malliavin derivative of the object is completely analogous to the decomposition \eqref{e:decompose_d211} in Proposition~\ref{prop:211_convergence_3}, just replacing the lowest noise node by its first chaos component. Unlike here we need to control $\|\EE I_1^u\|_{\Sigma_u}$ for the term $I_1^u$ from \eqref{e:decompose_d211l}, we did not control the corresponding term from \eqref{e:decompose_d211}. The reason is that the lowest noise node in the stochastic object from Proposition~\ref{prop:211_convergence_3} contains high chaos components only, which provides extra powers of $\eps$ to play with. This enables us to decompose the tree into different components with the help of \eqref{e:2',-1',1'_1zeta} and thus circumvents the expectation term from the spectral gap inequality. 
\end{remark}

The following proposition shows the convergence of the tree with lowest order chaos on each vertex to $\hPi^{\HS(\eps)} \<2'1'1'>$ in $\cC^{0-}$.

\begin{prop}\label{prop:211_convergence_5}
    For every $p\geq2$ and $\delta\in(0,\frac{1}{8})$\;, the bound
        \begin{equation*}
            \left\| \<S2,S1,S1>^{(\zeta)} - \EE\<S2,S1,S1>^{(\zeta)} - \scal{\hPi^{\HS(\eps)} \<2'1'1'>,\varphi^\lambda} \right\|_{L^p_\omega} \lesssim_p \zeta^{-2} \eps^{\delta} \lambda^{-3\delta}+\zeta^\beta\lambda^{-3\delta}
        \end{equation*}
    holds uniformly in $\eps,\zeta,\lambda\in(0,1)$ and $\varphi\in \bar{C}_c^1$.
\end{prop}
\begin{proof}
    Similar to Lemma~\ref{lem:2'1'1_error}, we replace each noise node with low chaos component by $\Psi_\eps$ or re-centered $\Psi_\eps^2$ with normalised coefficient, so that we have the decomposition
    \begin{equation} \label{e:2'1'1'_lll_decomposition}
        \<S2,S1,S1>^{(\zeta)} = \left( \<S2,S1,S1>^{(\zeta)} - \; \;  \frac{\big( a_\eps^{(\zeta)} \big)^3}{a_\eps^3} \cdot \<2,1,1>\right) + \frac{\big( a_\eps^{(\zeta)} \big)^3}{a_\eps^3} \cdot \<2,1,1>\;,
    \end{equation}
    where we make an abuse of notation for
    \begin{equation} \label{e:2'1'1'_kernel_high}
        \begin{split}
        \<2,1,1> = \iiint\limits_{(\RR \times \TT)^3} &\varphi^\lambda(x) \, \big( K_\eps'(x-y) - K_\eps'(-y) \big) \, K_\eps'(y-z)\\
        &\Psi_\eps(x) \tT^{(\geq 1)} \Big( \Psi_\eps(y) \tT^{(\geq 1)} \big(\Psi_\eps^2(z)\big) \Big) dz dy dx\;.
        \end{split}
    \end{equation}
    Here, we have multiplied $( a_\eps^{(\zeta)} / a_\eps )^3$ since there are three noise nodes. The difference between \eqref{e:2'1'1'_kernel_high} with expectation subtracted and $\scal{\hPi^{\HS(\eps)} \<2'1'1'>,\varphi^\lambda}$ is that the stochastic object $\Psi_\eps = P_\eps' * \xi_\eps$ is obtained from convolution with $P_\eps'$ instead of $P_0'$, and that the kernels appearing in the graph are $K_\eps'$ instead of $K_0'$. Hence, it follows immediately with the bounds in \cite[Section~4.2]{KPZCLT}, the difference $|a_\eps^{(\zeta)} - a_\eps|$ in \eqref{e:aepszeta-aeps}, and the difference of the kernels in Proposition~\ref{prop:kernel} that
    \begin{equation*}
        \left\| \, \frac{(a_\eps^{(\zeta)})^3}{a_\eps^3} \left( \<2,1,1> - \, \EE\<2,1,1> \right) - \scal{\hPi^{\HS(\eps)} \<2'1'1'>, \varphi^\lambda} \, \right\|_{L^p_\omega} \lesssim_p \eps^{\delta} \lambda^{-3\delta} + \zeta^\beta \lambda^{-3\delta}\;.
    \end{equation*}
    As for the first term on the right hand side of \eqref{e:2'1'1'_lll_decomposition}, it is the difference of product of three terms. By replacing each term in the product one by one, one ends up with a sum of three differences, each with the same type as in Lemma~\ref{lem:2'1'1_error}. As explained in the proof of Lemma~\ref{lem:2'1'1_error}, the estimates of these differences can be obtained by similar procedures in Proposition~\ref{prop:211_convergence_3} and Lemma~\ref{lem:211_convergence_6}. This shows that the $L_\omega^p$-norm of the first term on the right hand side (with expectation subtracted) is bounded by $\zeta^{-2} \eps^\delta \lambda^{-3\delta}$. This completes the proof of the proposition.
\end{proof}

\subsubsection{Convergence of the error part -- proof of Proposition~\ref{prop:211_convergence_2}}
\label{sec:211_error}

Now we consider the remainder $\<2'1'1'>_\eps - \<2'1'1'>_\eps^{(\zeta)}$. As mentioned earlier, since this term is just below regularity $0$, the extra smallness from $\zeta$ allows us to treat it in a slightly higher regularity space and that the small negative power of $\eps$ arising from enhancing the space can be balanced out by choosing $\zeta$ depending on $\eps$ in a proper way. We need the following lemma from \cite{kong_zhao}. 

\begin{lem} \cite[Lemma~3.8]{kong_zhao}
\label{lem:deterministic}
Let $\alpha_1, \alpha_2 \in (0,1)$ with $\alpha_1 > \alpha_2$. Then for every $\delta>0$, we have
\begin{equation*}
    \big| \scal{g, \big( f - f(z) \big) \varphi_z^\lambda} \big| \lesssim \|f\|_{\cC^{\alpha_1}} \|g\|_{\cC^{-\alpha_2}} \lambda^{\alpha_1 - \alpha_2 - \delta}
\end{equation*}
uniformly over $f \in \cC^{\alpha_1}$, $g \in \cC^{-\alpha_2}$ and $z$ in compact domains. 
\end{lem}

We also need the following bounds for $\<2'1'>_\eps - \<2'1'>^{(\zeta)}_\eps$ and $\<2'1'>_\eps$.

\begin{prop} \label{prop:21_convergence}
    For every $p\geq2$ and $\delta\in(0,\frac{1}{8})$\;, we have
    \begin{equation} \label{e:21_bound}
        \begin{split}
        \sup_{\varphi\in \bar{C}_c^1} \Big\|\scal{\;\<2'1'>_\eps - \<2'1'>^{(\zeta)}_\eps, \varphi^\lambda}\Big\|_{L^p_\omega} &\lesssim_p \eps^{-3\delta} \lambda^{-\frac{1}{2}+2\delta} \zeta^{\beta}\;,\\ 
        \sup_{\varphi\in \bar{C}_c^1} \Big\|\scal{\;\<2'1'>_\eps, \varphi^\lambda}\Big\|_{L^p_\omega} &\lesssim_p \lambda^{-\frac{1}{2}-\delta},
        \end{split}
    \end{equation}
    where the proportionality constants are independent of $\eps, \lambda, \zeta\in(0,1)$. As a consequence, we have
    \begin{equation} \label{e:210_norm}
        \big\|K_\eps' * \big( \, \<2'1'>_\eps - \<2'1'>_\eps^{(\zeta)} \, \big) \big\|_{L_\omega^p \cC^{\frac{1}{2}+\delta}} \lesssim_p \eps^{-3\delta} \zeta^{\beta}\;, \qquad \|K_\eps' * \<2'1'>_\eps\|_{L_\omega^p \cC^{\frac{1}{2}-2\delta}} \lesssim_p 1\;.
    \end{equation}
\end{prop}
\begin{proof}
     It is standard that the two bounds on the norms in \eqref{e:210_norm} follow from the bounds \eqref{e:21_bound}, Kolmogorov's continuity criterion, and the effect of convolution with $K_\eps'$. So it suffices to prove the two bounds in \eqref{e:21_bound}. Similar to Proposition~\ref{prop:211_convergence_1} (but much simpler), for every $p\geq2$ and sufficiently small $\delta>0$, there exists $\delta'>0$ such that the bound on the $\zeta$-regularised object $\<2'1'>_\eps^{(\zeta)}$ 
     \begin{equation} \label{e:21zeta}
        \sup_{\varphi\in \bar{C}_c^1} \lambda^{\frac{1}{2}+\delta} \Big\|\scal{\;\<2'1'>^{(\zeta)}_\eps, \varphi^\lambda}\Big\|_{L^p_\omega} \lesssim_p \big( \eps^{\delta'} \zeta^{-2} + 1 \big) \, \lambda^{-\frac{1}{2}-\delta}\;
     \end{equation}
     holds.
     We have the additional $\oO(1)$ constant $1$ instead of a positive power of $\zeta$ since \eqref{e:21zeta} gives a uniform bound on the stochastic object only but not comparing the difference to the limiting object. Assuming the first bound in \eqref{e:21_bound}, these together imply the second bound in \eqref{e:21_bound} by choosing $\zeta$ being a small positive power of $\eps$. So it remains to prove the first bound in \eqref{e:21_bound}. 

    The difference $\<2'1'>_\eps - \<2'1'>_\eps^{(\zeta)}$ can be decomposed into a sum of two terms, the first one consisting the product $\<2'>_\eps \big( \<1'>_\eps - \<1'>_\eps^{(\zeta)} \big)$, and the second consisting $\big( \<2'>_\eps - \<2'>_\eps^{(\zeta)} \big) \<1'>_\eps^{(\zeta)}$. We provide details for the first one only, and the bounds for the second one are essentially the same. For the first one, it suffices to prove the bound
    \begin{equation*}
        \left\| \;\<2',1'> - \EE\;\<2',1'> \right\|_{L^p_\omega} \lesssim_p \eps^{-3\delta} \zeta^{\beta} \lambda^{-\frac{1}{2}+2\delta}.
    \end{equation*}
    By \eqref{e:sg} and \eqref{e:prod_deri}, we have
    \begin{equation*}
    \begin{split}
        \left\| \;\<2',1'> - \EE\;\<2',1'> \right\|_{L^p_\omega} \lesssim_p& \left\| \<1',1'> \right\|_{\Sigma_u L^p_{\omega}} + \left\| \<2',0'> \right\|_{\Sigma_u L^p_{\omega}}\\
        &+ \left\| \<1',0'_u> \right\|_{\Sigma_u L^p_{\omega}} \triangleq \sum_{i=1}^3I_i.
    \end{split}
    \end{equation*}
    The term $I_1$ is the most complicated one, so we focus on its details only. Applying \eqref{e:sg} to $I_1$, we get
    \begin{equation*}
        \begin{split}
            I_1 \lesssim_p &\left\|\EE \<1',1'> \right\|_{\Sigma_u} + \left\| \<1',0'> \right\|_{\Sigma_{\vec{u}} L^p_{\omega}} + \left\| \<0',1'> \right\|_{\Sigma_{\vec{u}} L^p_{\omega}} \\
            &+ \left\| \<0',0'_u> \right\|_{\Sigma_{\vec{u}} L^p_{\omega}} \triangleq \sum_{i=1}^4I_{1i},
        \end{split}
    \end{equation*}
    where $\vec{u}=(u_1,u_2)\in(\RR\times\TT_\eps)^2$. We only provide the proof of $I_{13}$, the remaining terms can be handled similarly. We use triangle and H\"older inequalities to separate $D^2_{\vec{u}} \<2'>$ and the rest part, and applying the spectral gap inequality to its second component to get
    \begin{equation} \label{e:error_I13}
        I_{13} \lesssim_p \left\| \int_{\RR \times \TT} \|D^2_{\vec{u}} \<2'>_\eps(y) \|_{L_{\omega}^{2p}} \cdot \left\| \<z,D_31'>\right\|_{\Sigma_{v}^{2p} L^{2p}_{\omega}} \, dy \right\|_{\Sigma_{\vec{u}}}\;.
    \end{equation}
    We first deal with the second term in the integrand above. By \eqref{e:1-1_zeta} and Lemma~\ref{lem:P_eps}, we have
    \begin{equation*}
        \|D_v(\<1'>_\eps - \<1'>_\eps^{(\zeta)})(x)\|_{L_\omega^{2p}} \lesssim_p \zeta^\beta |(P_\eps^\theta)'(x,v)| \lesssim \zeta^\beta \cdot \frac{\eps^{\frac{3}{2}}}{(|x-\eps v| + \eps)^2}\;.
    \end{equation*}
    Plugging it into the corresponding stochastic object, and applying Lemma~\ref{lem:Lvecpcontrol} with $k=1$, $\alpha=2$ and then Lemma~\ref{lem:kernel_convolution_6}, we get
    \begin{equation*}
        \begin{split}
        \left\| \<z,D_31'>\right\|_{\Sigma_{v} L^{2p}_{\omega}} &\lesssim_p \zeta^\beta \left\| \int_{\RR \times \TT} |\varphi^\lambda(x)| \, |K_\eps'(x-y)| \, \frac{\eps^{\frac{3}{2}}}{(|x-\eps v|+\eps)^2} \, dx \right\|_{\Sigma_v}\\
        &\lesssim \frac{\zeta^{\beta}\eps^{-2\delta}\lambda^{-\frac12+2\delta}\boldsymbol{1}_{|y|\lesssim1}}{|y|^2}\;.
        \end{split}
    \end{equation*}
    As for the term $D^2_{\vec{u}} \<2'>$, by Lemmas~\ref{lem:derivative_genFF} and~\ref{lem:P_eps}, we have
    \begin{equation*}
        \|D^2_{\vec{u}} \<2'>_\eps (y)\|_{L_{\omega}^{2p}} \lesssim_p \prod_{i=1}^{2} |(P_\eps^\theta)'(y, u_i)| \lesssim \prod_{i=1}^{2} \frac{\eps^{\frac{3}{2}}}{(|y-\eps u_i| + \eps)^2}\;.
    \end{equation*}
    Plugging the above two bounds back into \eqref{e:error_I13}, and applying Lemma~\ref{lem:Lvecpcontrol} with $k=2$ and $\alpha_1 = \alpha_2 = 2$, we conclude that $I_{13} \lesssim _p\eps^{-3\delta} \zeta^{\beta} \lambda^{-\frac{1}{2}+2\delta}$. This completes the proof for the most complicated term from the decomposition of the object. All other terms can be controlled in similar or simpler ways. This completes the proof of the first bound in \eqref{e:21_bound} and hence Proposition~\ref{prop:21_convergence}.
\end{proof}

\begin{rmk}
    It is essential that the first bound in \eqref{e:21_bound} has ``$+2\delta$" in the exponent of $\lambda$ (and $\cC^{\frac{1}{2}+\delta}$-norm for the first one in \eqref{e:210_norm}). This regularity gain for the difference $\<2'1'>_\eps - \<2'1'>_\eps^{(\zeta)}$ allows the use of Lemma~\ref{lem:deterministic} (making the assumption $\alpha_1 > \alpha_2$ satisfied). 
\end{rmk}

We are now ready to prove the estimate of the remainder.

\begin{proof} [Proof of Proposition~\ref{prop:211_convergence_2}]
We first decompose $\<2'1'1'>_\eps - \<2'1'1'>_\eps^{(\zeta)}$ as
\begin{equation} \label{e:211_remainder_decomposition}
    \<2'1'1'>_\eps - \<2'1'1'>_\eps^{(\zeta)} = \<2'1'0>_\eps (\<1'>_\eps - \<1'>_\eps^{(\zeta)}) + \<1'>_\eps^{(\zeta)} (\<2'1'0>_\eps - \<2'1'0>_\eps^{(\zeta)}) + ( C^{(\eps,\zeta)}_{\<2'1'1's>} - C^{(\eps)}_{\<2'1'1's>} ),
\end{equation}
where $C^{(\eps)}_{\<2'1'1's>} = \EE\big[ \<2'1'0>_\eps\cdot\<1'>_\eps\big]$ and $C^{(\eps,\zeta)}_{\<2'1'1's>} = \EE\big[ \<2'1'0>_\eps^{(\zeta)} \cdot\<1'>_\eps^{(\zeta)} \big]$. Note that here, $\<2'1'0>_\eps$ is defined as
\begin{equation*}
    \<2'1'0>_\eps(x) := (K_\eps' * \<2'1'>_\eps)(x) - (K_\eps' * \<2'1'>_\eps)(0)\;,
\end{equation*}
and similarly for its $\zeta$-regularised version as well as their differences. Together with Proposition~\ref{prop:21_convergence}, this enables us to apply Lemma~\ref{lem:deterministic} with $f$ being the two objects in \eqref{e:210_norm}.

Testing the first term in the decomposition \eqref{e:211_remainder_decomposition}, and applying Lemmas~\ref{lem:<1>_convergence_2},~\ref{lem:deterministic} and Proposition~\ref{prop:21_convergence}, we get
\begin{equation*}
    \big\| \scal{\<1'>_\eps - \<1'>_\eps^{(\zeta)}, \; \<2'1'0>_\eps \varphi^\lambda }  \big\|_{L_\omega^p} \lesssim \|\<1'>_\eps - \<1'>_\eps^{(\zeta)}\|_{L_{\omega}^{2p} \cC^{-\frac{1}{2}+\frac{\nu}{2}}} \|K_\eps' * \<2'1'>_\eps\|_{L_{\omega}^{2p} \cC^{\frac{1}{2}-\delta}} \lambda^{\delta-\nu} \lesssim_p \zeta^{\beta} \eps^{-\nu} \lambda^{\delta-\nu}\;.
\end{equation*}
Choosing $\nu = 3\delta$ gives the desired bound for this term. The bound for the second term in \eqref{e:211_remainder_decomposition} can be obtained in the same way. Finally, the difference of the two constants $C^{(\eps,\zeta)}_{\<2'1'1's>} - C^{(\eps)}_{\<2'1'1's>}$ can be bounded by $\zeta^{\beta} |\log\eps|$. The proof of Proposition~\ref{prop:211_convergence_2} is then completed by re-defining $\delta$. 
\end{proof}

\bibliographystyle{Martin}
\bibliography{Refs}

\end{document}